\documentclass[10pt]{article}

\usepackage{graphics}
\usepackage{graphicx}

\usepackage[a4paper, left=35mm,right=35mm,top=34mm,bottom=34mm]{geometry}
\usepackage[utf8]{inputenc}
\usepackage[T1]{fontenc}
\usepackage[english]{babel}

\usepackage{enumerate}
\usepackage{graphicx}
\usepackage{hyperref}
\hypersetup{
    colorlinks=true,
    linkcolor=blue,
    filecolor=magenta,      
    urlcolor=cyan,
}
\usepackage{lipsum}
\usepackage{amsfonts}
\usepackage{graphicx}
\usepackage{epstopdf}
\usepackage{algorithmic}

 \usepackage{dsfont}
 \usepackage{psfrag}
 \usepackage{color}
 \usepackage{tikz}
 \usepackage{subfigure}
\usepackage{mathtools,amsthm,amssymb,amsfonts}
\usepackage{algorithm}
\usepackage{algorithmic}
\makeatother
\theoremstyle{plain}

\usepackage{type1cm} 
\usepackage{makeidx} 
\usepackage{graphicx} 
\usepackage{multicol} 
\usepackage[bottom]{footmisc}

\usepackage{newtxtext} %
\usepackage{newtxmath} 

\usepackage{xcolor}

\usepackage{amsmath}
\usepackage{amsfonts}
\usepackage{enumerate}
\usepackage{parskip}
\usepackage{color}
\usepackage{hyperref}
\usepackage{graphicx}
\usepackage{verbatim}
\usepackage{psfrag}
\usepackage{dsfont}

\usepackage{cite}
\usepackage{marginnote}
\makeindex 

\newtheorem{assumption}{Assumption}
\newtheorem{proposition}{\textbf{Proposition}}
\newtheorem{ex}{\textbf{Example}}

\newtheorem{remark}{\textbf{Remark}}
\newtheorem{theorem}{\textbf{Theorem}}
\newtheorem{lemma}{\textbf{Lemma}}
\newtheorem{definition}{\textbf{Definition}}

\def\eq{equation}

\newcommand{\dom}{\Omega}

\def\ep{\varepsilon}
\def\eps{\varepsilon}

\def\s{\sigma}

\def\S*{\Sigma_*}

\def\<{\langle}
\def\>{\rangle}

\def\Dim{d}

\def\N{\mathbb{N}}

\def\R{\mathbb{R}}
\def\bbS{\mathbb{S}}
\def\rn{\mathbb{R}^n}

\def\H{\mathcal{H}}

\def\H{\mathcal{H}}
\def\S{\mathcal{S}}

\def\L{\mathcal{L}}
\def\tC{\widetilde{C}}
\def\Fj{F^{(j)}}
\def\Ej{E^{(j)}}
\def\Gj{G^{(j)}}
\def\Ij{I^{(j)}}

\def\Lj{L^{(j)}}
\def\Nj{N^{(j)}}
\def\Oj{O^{(j)}}

\def\rj{r^{(j)}}
\def\Sj{S^{(j)}}
\def\Uj{U^{(j)}}
\def\Vj{V^{(j)}}
\def\Wj{\Sigma^{(j)}}
\newcommand{\W}{\Sigma}
\def\bZj{\overline{Z}^{(j)}}
\def\Zj{Z^{(j)}}
\def\zj{z^{(j)}}
\def\bzj{\overline{z}^{(j)}}
\def\Zj{Z^{(j)}}
\def\bz{\overline{z}}
\def\bZ{\overline{Z}}
\def\ej{\eta^{(j)}}
\def\mj{\mu^{(j)}}
\def\nj{\nu^{(j)}}
\def\tnj{\tilde{\nu}^{(j)}}
\def\1{\mathbf{1}}
\newcommand{\const}{\operatorname{const}}

\newcommand{\conv}{\operatorname{conv}}
\newcommand{\diam}{\operatorname{diam}}

\newcommand{\esup}{\operatorname{ess\,sup}}

\DeclareMathOperator*{\elim}{ess\,lim}
\newcommand{\nor}{\operatorname{nor}}
\newcommand{\Int}{\operatorname{int}}
\newcommand{\rea}{\operatorname{reach}}
\newcommand{\reg}{\operatorname{Reg}}

\newcommand{\Frac}{{\rm frac}}
\newcommand{\var}{{\rm var}}
\newcommand{\bd}{\partial}
\newcommand{\sC}{\mathcal{C}}
\newcommand{\sM}{\mathcal{M}}
\newcommand{\asC}{\widetilde{\sC}}
\newcommand{\asM}{\widetilde{\sM}}
\newcommand{\sN}{\mathcal{N}}

\newcommand{\Unp}{\operatorname{Unp}}
\newcommand{\cl}[1]{\overline{#1}}
\newcommand{\sT}{\mathcal{T}}

\def\del {\partial}
\def\eps{\varepsilon}

\def\R{\mathbb{R}}

\def\N{\mathbb{N}}
\def\sH{\mathcal{H}}

 \def\dx{{\rm d}x}
 \def\dy{{\rm d}y}

\def\dz{{\rm d}z}
\def\dr{{\rm d}r}

\def\dv{{\rm d}v}
\def\ds{{\rm d}s}

\def\dist{\operatorname{dist}}
\def\erf{\operatorname{erf}}
\def\dtau{{\rm d}\tau}

\begin{document}

\title{Fractal curvatures and short-time asymptotics of heat content}
\author{Anna Rozanova-Pierrat\thanks{CentraleSup\'elec, Universit\'e Paris-Saclay, France {anna.rozanova-pierrat@centralesupelec.fr}}, Alexander Teplyaev\thanks{University of Connecticut, USA {teplyaev@uconn.edu}}, Steffen Winter\thanks{Institute of Stochastics, Karlsruhe Institute of Technology, Englerstr. 2, D-76131 Karlsruhe, Germany {steffen.winter@kit.edu}}, and Martina Z\"ahle\thanks{Institute of Mathematics, University of Jena, Ernst-Abbe-Platz 2, D-07743 Jena, Germany {martina.zaehle@uni-jena.de}}}
%
%
\maketitle

\begin{abstract}
	The aim of our paper is twofold. First, we present new mathematical developments on the analysis of de Gennes' hypothesis on the short-time asymptotics of the heat content for bounded domains with smooth boundary and with fractal boundary. Second, we discuss new findings and concepts related to fractal curvatures for domains with fractal boundary. We conjecture that fractal curvatures and their scaling exponents will emerge in the short-time heat content asymptotics of domains with fractal boundary and the results discussed here are small initial contributions towards a resolution. 
\end{abstract}

\section{Introduction} \label{sec:1}

In this paper, we discuss some numerical and theoretical results towards the conjecture that in domains with fractal boundary under appropriate assumptions, the short-term heat content is described by (fractal) curvature measures (see Section~\ref{sec:curv-meas}). This conjecture is connected with many theoretical and numerical results, some of which we describe in detail, while others are available in the literature. The conjecture aligns with the idea proposed by Claude Bardos that the short-term heat content asymptotics is geometric in nature, with curvature measures being the most relevant geometric quantities.

We discuss the heat content in Subsection~\ref{sec:2}, including several classical results concerning its short-time asymptotic behavior. It is important to note that the leading term in these asymptotics is given by the perimeter in two dimensions and by the surface area in three dimensions. Furthermore, curvature appears explicitly in the higher-order terms, as demonstrated in Theorem~\ref{ThFinalr} in Subsection~\ref{SubsReg} and in \cite[Theorem~7.1]{BARDOS-2016}.

Our research is related to the results of van den Berg and coauthors on the heat content on Riemannian manifolds and polygonal domains \cite{vdBG15,vdBG16, VAN_DEN_BERG-1994, VANDENBERG-1994, VAN_DEN_BERG-1990, Gittins23}, as well as on fractal domains \cite{MR1831409,MR1785451,MR1620833,MR1335459,MR1305782}. Studying the short-time heat content asymptotics in fractal domains with self-similar boundaries may be particularly connected with complex dimensions and related work \cite{LapidusPearse2006,Lapidus-P-2008,Lapidus-P-2010,Lapidus-P-W-2011,LapidusPearse,FLECKINGER-1995,LEVITIN-1996}. Additionally, there is substantial literature on the analysis on fractals connected to our work, particularly by Strichartz and his collaborators \cite{StrBook,StrDiffEq,StrHarmonic,StrVicsek,StrRevisited,StrSig,PARgaps,StrExact,ADT,Grabner2008,Grabner2012,GrabnerWoess1997,Grabner1997,ST,HSTZ,KT2004,hambly2011asymptotics,strichartz2012spectral,alonso2023oscillations,kajino2014log,hinz2018fractal,hinz2023boundary,alonso2021besov}.

The paper is organized as follows. In Sections~\ref{SecMS} and~\ref{SubsecBardos}, we describe for a two-media problem approximations of the heat content by the volume of the interior Minkowski sausage. We provide asymptotics for smooth domains and corresponding results for fractal domains. The general idea of de Gennes about the proportionality of the heat content for small times to the interior Minkowski sausage of the radius equal to the diffusion length is proved for a large class of two-sided extension domains (see Section~\ref{SubsecBardos}) including domains with fractal boundaries (with a $d$-upper regular measure). 
Curvature measures and fractal curvatures for compact sets are reviewed in Section~\ref{sec:curv-meas}. Then the notion of fractal curvature is adapted to domains with fractal boundary in Section~\ref{fractal bound}, and in particular to domains with piecewise self-similar boundaries.  Then fractal curvatures of such domains are studied in detail, in Section~\ref{SecAssociatedfractalcurvatureslocalapproach} via a local approach  and in Section~\ref{sec:rel-fc} using self-similar tilings.


\subsection*{Acknowledgments} The authors are deeply grateful to Claude Bardos for insightful suggestions about the heat content asymptotics. AR-P thanks Denis Grebenkov for advise and help in the presentation of numerical results, performed several years ago in collaboration with Bernard Sapoval. AR et AT were supported by CNRS INSMI IEA project "Functional and applied analysis with fractal or non-Lipschitz boundaries".
AT was supported in part by NSF Grant DMS-2349433 and the Simons Foundation.
SW and MZ acknowledge support by the German Research Foundation (DFG) via grant no.\ 433621248. 

\subsection{Heat content and de Gennes' hypothesis} \label{sec:2}

The speed of heat exchange between a hot and cold medium is significantly influenced by the shape of the interface, especially at short time scales~\cite{ROZANOVA-PIERRAT-2012,BARDOS-2016,DE_GENNES-1982}. For example, the shape of a radiator greatly affects the rate of diffusive heat transfer, a phenomenon also observed in numerical simulations. Consider a numerical computation (performed using COMSOL Multiphysics for a model described by the linear heat equation) on a cavity initially composed of a hot and a cold medium (see the bottom line of Fig.~\ref{FigIntHeatDom1}), separated by interfaces of different lengths. It can be observed that the speed of heat propagation increases with the interface length for any fixed (sufficiently small) time.
This can be compared with the top line of Fig.~\ref{FigIntHeatDom1}, which shows the propagation of heat from a hot boundary into a cold medium. In the two-medium configuration, two distinct processes are observed: heat propagation (from hot to cold) and cold propagation (from cold to hot). This phenomenon is absent in the upper figures, where the hot boundary maintains a constant temperature equal to $1$ at all times. From both theoretical and practical perspectives, the case involving two media is more informative and of greater interest.

\begin{figure}[!ht]
	\begin{center}
	\includegraphics[width=\linewidth]{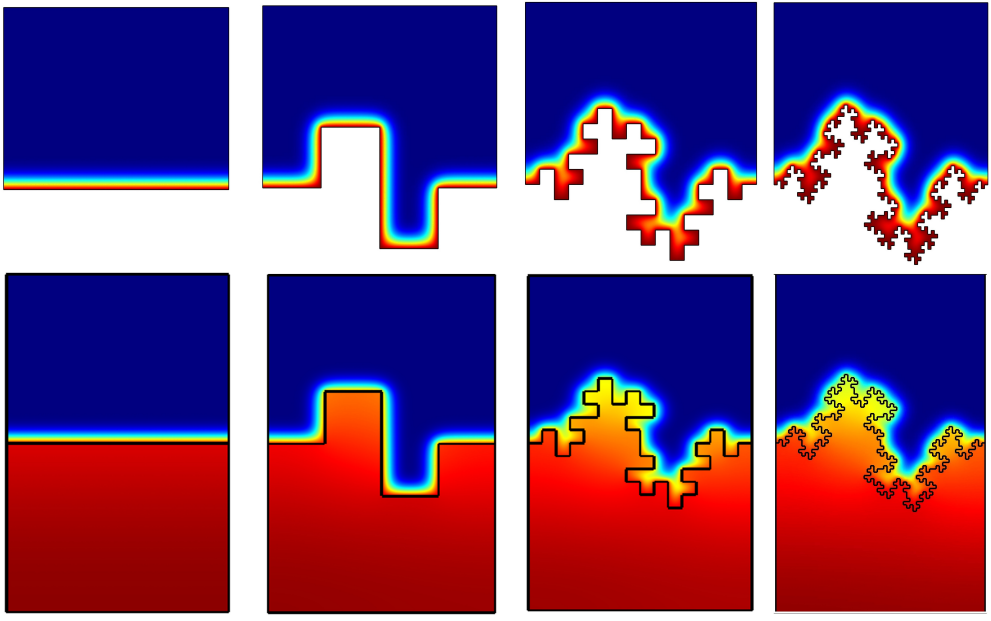}
	\end{center}
	\caption{\label{FigIntHeatDom1} Influence of the geometry on the heat propagation: colors encode the temperature distributions for media with different boundaries observed at some fixed point in time ($t=0.1$); the color scale ranging from red (hot) to blue (cold). Top row: Dirichlet condition, equal to $1$, imposed on the boundary at the bottom. Bottom row: propagation between a hot and a cold medium in a thermo-isolated cavity.}
\end{figure}

Once again, the beneficial interest is to find boundary designs that allow for heat transfer as fast as possible. Hence, the aim is to study the behavior of heat diffusion for a short time. Let us denote by $N(t)$ the heat content at time $t$, which is defined as the integral over the domain of heat propagation of the solution (at time $t$) of the heat equation in this domain with a thermo-isolated exterior boundary. 
{For the case of one medium with a hot part of the boundary, as on the top of Fig.~\ref{FigIntHeatDom1}, $N(t)=\int_\Omega u(x,t)\dx$, while for the bottom configuration of Fig.~\ref{FigIntHeatDom1} for two media, one initially hot, $\Omega_+$, and the other initially cold, $\Omega_-$, $N(t)=\int_{\Omega_-}u(x,t)\dx$ with $u$ the solution of the corresponding heat propagation problem.} In Fig.~\ref{FigIntHeat2}, which shows plots of the relative heat content $N(t)/N(\infty)$ for the different domains in Fig.~\ref{FigIntHeatDom1}, we can observe that for long times, i.e.\ as $t\to +\infty$, there is no influence of the geometry on the heat content, since it converges to the same equilibrium state (a constant temperature of the two media) for all boundary designs. 
However, the geometric influence is significant in the regime of short times, i.e., \ as $t\to +0+$. %

In the following, we use the term 'prefractal boundary' to refer to a Lipschitz boundary that approximates a fractal boundary. Typically it arises from a very simple boundary by repeated replacements of parts of the boundary by scaled copies of the whole boundary, leading to finer and finer structures. The number of repetitions is known as the order of the prefractal generation and the fractal boundary emerges in the limit as the order approaches $+\infty$.

For the prefractal example 
{with two media} in Fig.~\ref{FigIntHeatDom1}, the following length scales are relevant for the heat transfer (see Fig.~\ref{FigIntHeat2}): 
\begin{itemize}
 \item $\ell=1/64$, the smallest cut-off 
 {(geometric length)} of the third prefractal generation (length of the smallest prefractal features);
\item $L=1$, the largest cut-off (the base of the prefractal or the width of the whole cavity); 
\item $\ell_D=\sqrt{D_+ t}$, the diffusion length at time $t$
{, where $D_+$ is the diffusion coefficient of the initially hot medium (see also the caption of Fig.~\ref{FigIntHeat2})}.
\end{itemize}

In Fig.~\ref{FigIntHeat2}, one can see the existence of three time regions exhibiting different speeds of heat propagation (following different asymptotics indicated by the blue and red dotted line). These different speeds are asymptotically characterized by different powers of $t$. 
Exactly this dependence was pointed out by de Gennes~\cite{DE_GENNES-1982}, 
{which we refer to as the "de Gennes hypothesis"}, in the following way:
\begin{itemize}
 \item for \emph{short} times, i.e.\ for  diffusion lengths $\ell_D\in]0,o(\ell)[$, the prefractal shape behaves like a Lipschitz curve with a finite length
$$\displaystyle N(t)\propto \sH^{1}(\del \Omega)\sqrt{D_+ t},$$ %
with $\propto$ meaning \emph{proportional} and $\sH^{1}$ being the Hausdorff measure on $\del \Omega$, 
{which is the interface between initially hot and cold media,}
\item for \emph{intermediate} times, i.e. for  $\displaystyle \ell_D\in[O(\ell),o(L)[$, the third prefractal generation behaves like the fractal it approximates:
$$\displaystyle N(t)\propto\left(D_+ t\right)^{\frac{2-d}{2}},$$ with $d$ the Minkowski dimension of the fractal, 
\item for \emph{long} times, i.e.\ for $\ell_D\in[O(L),+\infty[$, one reaches the saturation regime, provided the domain is bounded: the heat content $N(t)$ converges (exponentially fast) to a constant for $t\to +\infty$. For unbounded domains one observes the same behavior as for a flat boundary: $\displaystyle N(t)\propto \sqrt{t}$, as $t\to +\infty$.%
\end{itemize}
The blue and red dotted lines in Fig.~\ref{FigIntHeat2} follow the predicted powers of $t$ in first two regimes. The same type of result, but this time with exact asymptotes, is shown in Fig.~\ref{FigF3} at the end of Section~\ref{SubsecBardos}.
\begin{figure}[!ht]
 \begin{center}
\psfrag{t}{{\small $\sqrt{D_+ t}/b$}}
\psfrag{N}{{\small $\qquad$ $\qquad$ mean value $N(t)/N(\infty)$}}
 \includegraphics[width=\linewidth]{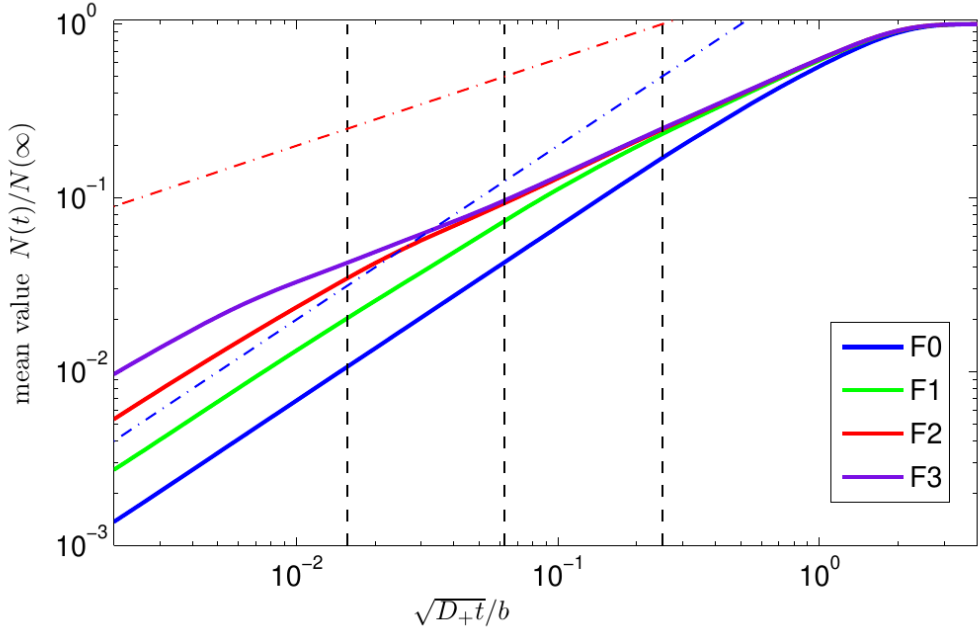}
\end{center}
\caption{\label{FigIntHeat2} Heat propagation in time for the four pre-fractal cavities shown at the bottom line of Fig.~\ref{FigIntHeatDom1}. Each cavity is split symmetrically into two media of equal volume (equal to half the volume of $[0,l]\times [0,b]$ with $l=1$, $b=3$). At $t=0$, one medium is hot, and the other one is cold (with diffusion coefficients given by $D_-=1$ and $D_+=1/100$, respectively). 
 Asymptotes: $2^i\sqrt{D_+t}/b$ where $i=0,1,2,3$ is the pre-fractal generation.}
 \end{figure}

Let $\Omega$ be a bounded domain in $\R^n$. In the case when there is no resistivity of the boundary to heat propagation, de Gennes~\cite{DE_GENNES-1982}
argued that, as $t\to 0+$, the heat content $N(t)$ of $\Omega$ is proportional to the volume
$\mu(\del \Omega,\sqrt{D_+ t})$ of the interior Minkowski sausage of
$\Omega$ of width equal to the diffusion length $\sqrt{D_+
t}$, that is, 
\begin{equation}\label{EqDefmu}
\mu(\del \Omega,r) := \lambda_n\left(\{ x\in \Omega|\operatorname{dist}(x,\del \Omega)<r\}\right),
\end{equation}
where $\lambda_n$ denotes the Lebesgue measure in $\R^n$. 
{Recall that the (relative) $d$-dimensional Minkowski content of $\del \Omega$ relative to $\Omega$ is defined by
$\mathcal{M}^d(\del \Omega,\Omega):=\lim_{r\to 0+} r^{d-n}\mu(\del \Omega,r)$, see also \eqref{EqSteffen} below.
Actually, the previously explained different scaling regimes correspond to two different boundary types:
the behavior for short times corresponds to the one of a Lipschitz curve (in fact, the prefractal approximation is a Lipschitz curve). However, at intermediate times the fractal nature of the boundary curve dominates and slows down the heat propagation.
%
In particular, the de Gennes conjecture would imply that
\begin{itemize}
\item
for a regular boundary $\del \Omega$, $N(t)$ is proportional to
$\sH^{n-1}(\del \Omega) \sqrt{D_+t}$ as $t\to 0+$,
where $\sH^{n-1}$ denotes the Hausdorff measure and thus $\sH^{n-1}(\del\Omega)$ is the surface area or perimeter of $\del\Omega$ 
{(see also Eq.~\eqref{EqMainAss} with $d=n-1$ and $r=\sqrt{D_+t}$)};
\item
for a fractal boundary $\del \Omega$ of Minkowski dimension $d$ (and with existing relative Minkowski content $\mathcal{M}^d(\del \Omega,\Omega)$),
$N(t)$ is proportional to $\mathcal{M}^d(\del \Omega,\Omega) (D_+t)^{\frac{n-d}{2}}$.
\end{itemize}

The de Gennes scaling argument was further investigated in~\cite{ROZANOVA-PIERRAT-2012} both experimentally and numerically. The main question of the approximation of the heat content for short times is directly related to the approximation of the volume of an interior Minkowski sausage of the interface between two media.

The interior Minkowski sausage can also be compared with isolines of the heat solution taken at a fixed moment $t_0$ as done numerically in Fig.~\ref{FigMinskS} for the case of a cold medium with a hot boundary. The idea also comes from~\cite{ROZANOVA-PIERRAT-2012}, where it was shown that for a good approximation of the interior Minkowski sausage by temperature isolines, it is important to consider rather small temperature values (much less than $1$).
\begin{figure}[!ht]
	 \begin{center}
 \includegraphics[width=0.5\linewidth]{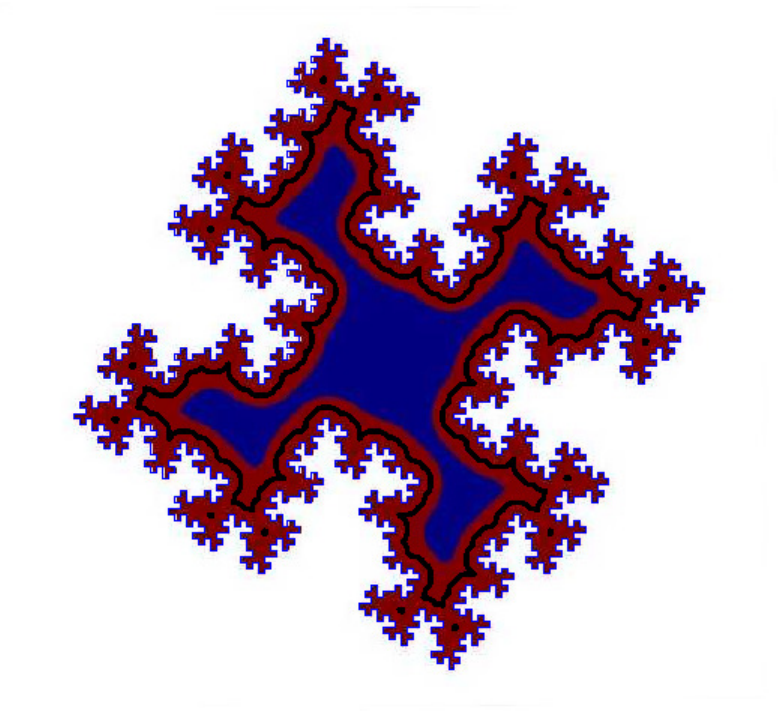}
 \vspace{2cm}
 \includegraphics[width=0.4\linewidth]{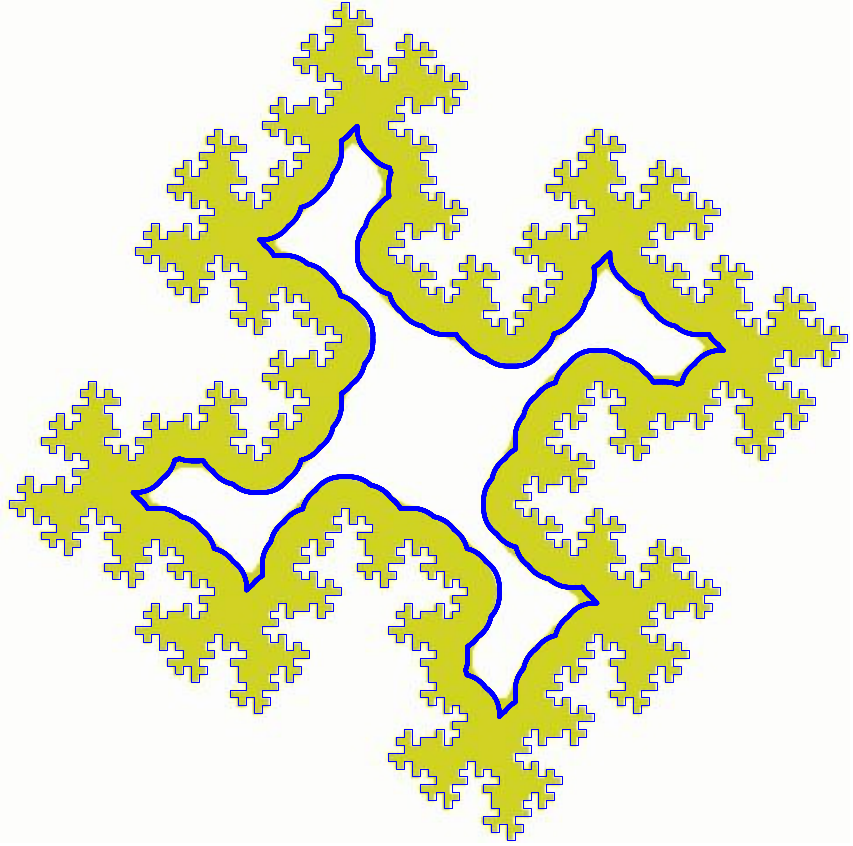}
 \end{center}
 \vspace{-2cm}\caption{\label{FigMinskS} The boundary of interior Minkowski sausage of width $\sqrt{D_+ t_0}$ (on the left, given by the black line) and width $2\sqrt{D_+ t_0}$ (on the right, given by the blue line) is compared to the isolines $u(x,t_0)=0.1$ of temperature at time $t_0=0.1$. ($u$ is the solution of the heat equation with Dirichlet boundary condition equal to $1$ and homogeneous initial data.) The isoline is the boundary between the red and blue region (left) and between the yellow and the white region (right), respectively.}
\end{figure}
 On the way to specify and formalize the de Gennes
 scaling argument~\cite{DE_GENNES-1982}, the first approximately asymptotic formula relating the heat content propagation and the volume of the interior Minkowski sausage for fractal boundaries under physically relevant boundary conditions  were given in~\cite{BARDOS-2016} (see also Section~\ref{SubsecDirichlet}
 for related results under other boundary conditions).  
 Moreover, for the case of a $C^3$-regular boundary, a two-term asymptotic formula has been obtained in~\cite{BARDOS-2016} involving the mean curvature of the boundary in the second term.

 In Section~\ref{SubsecDirichlet}, we review the existing results on the asymptotic expansion of the heat content in the case of a single medium with a uniformly hot boundary (i.e., with Dirichlet boundary condition equal to $1$ at all times on all of the boundary, as, e.g., \ in Fig.~\ref{FigMinskS}).
 In Section~\ref{SecMS} we give different approximations of the volume of the interior Minkowski sausage, focusing on the case of ($d$-dimensional) Minkowski measurable boundaries. We also give sufficient conditions on the convergence of domains that imply the convergence of the volumes of the parallel sets and the corresponding heat contents. The particular case of the von Koch snowflake domain in $\R^2$, which has a non-Minkowski measurable boundary, is also discussed.
 In Section~\ref{SubsecBardos}, we describe the analogous results in the de Gennes context obtained for the heat exchange of two media with a possibly resistive boundary~\cite{BARDOS-2016}. 
 First we detail in Section~\ref{SubsReg} the regular case of a boundary of class $C^3$ involving the two first asymptotic terms. 
 Then, we consider different regular and non-regular boundaries firstly in the particular case $D_+=D_-$ in Section~\ref{SSpartic}. 
 The used method is then applied to a transmission problem stated in~\eqref{prb1}--\eqref{prb1end}. 
 Using the asymptotic expansion of the heat content in the regular case involving the volumes of parallel sets, we pass to the limit in these formulas using the assumptions of Section~\ref{ssMinkMeas} to obtain the analogous results in the context $d$-dimensional Minkowski measurable and non-measurable (von Koch fractal in $\R^2$) boundaries in Section~\ref{sssArbC}.

 \subsection{Results for one medium with Dirichlet boundary condition}\label{SubsecDirichlet}

 Regarding the heat content asymptotics for domains with Dirichlet boundary conditions (i.e., the case of one medium, as $t\to 0+$), there are a number of classical results available in the literature for different classes of domains, including Riemannian manifolds \cite{vdBG15,VANDENBERG-1994},  smooth domains in $\R^n$ \cite{VAN_DEN_BERG-1994}, polygonal domains \cite{VAN_DEN_BERG-1990,vdBG16}  and even some fractal domains \cite{MR1831409,MR1785451,MR1620833,MR1335459,MR1305782}. Homogeneous Dirichlet boundary conditions model a constant temperature zero at the boundary. They do not model the heat exchange between the interior and the exterior media separated by the boundary, which we consider.

We highlight a few known asymptotic results. 
\begin{enumerate}
\item[(1)] If  $\Omega\subset \R^n$ is a bounded connected domain with a $C^3$-regular boundary $\del \Omega$, then (as $t\to 0+$) the heat content of $\Omega$ is determined up to the third order by 
$$
\displaystyle N(t)=\sqrt{t}\frac{2}{\sqrt{\pi}} \sH^{n-1}(\del \Omega)
-t\frac{n-1}{2}\int_{\del \Omega} H(x)\sH^{n-1}(\dx)+O(t^\frac{3}{2}),
$$
where $H(x)$ denotes the mean curvature at $x\in\del \Omega$, see \cite{VAN_DEN_BERG-1994}.
\item[(2)] If $\Omega\subset \R^2$ is a bounded connected domain with a polygonal boundary $\del \Omega$, then, according to \cite{VAN_DEN_BERG-1990}, the heat content, as $t\to 0+$, is given up to some exponential error by 
\begin{align*}
 N(t)&=\sqrt{t}\frac{2}{\sqrt{\pi}}\sH^{n-1}(\del \Omega)-t\sum_{i=1}^\ell c(\gamma_i)+O(e^{-\frac{r}{t}}),
\end{align*}
where $\ell$ is the number of vertices of $\del\Omega$, $\gamma_i$, $i=1,\ldots, \ell$, are the interior angles at these vertices and 
\begin{align*}
    c(\gamma_i)=4\int_0^{+\infty} \frac{ \sinh ((\pi-\gamma_i)z)}{\sinh(\pi z)\cosh(\gamma_i z)}\dz.
\end{align*}
\item[(3)]
If $\Omega\subset\R^n$, $n\geq 2$ satisfies a certain estimate regarding its Newtonian capacity (which holds e.g. if $\Omega\subset\R^2$ is simply connected), then, by \cite{MR1305782}, for all $t>0$,
\begin{align*}
   c_1 \mu(\del \Omega,c_2 \sqrt{t})\leq N(t) \leq c_3 t^{-1} \int_0^\infty e^{-r^2/(8t)} r\mu(\del \Omega, r) \dr, 
\end{align*}
where the positive constants $c_1$ and $c_2$ depend on $\Omega$. In particular, if $\del \Omega$ is Minkowski nondegenerate (see \eqref{Eqliminfsup} below for the definition) with Minkowski dimension $d\in]n-1,n[$, then there are constants $c>0$ and $t_0>0$ such that, for all $0<t\leq t_0$,
\begin{align*}
   c^{-1} t^{\frac{n-d}{2}}\leq N(t) \leq c t^{\frac{n-d}{2}}.
\end{align*}
\item[(4)] 
A refined estimate is obtained in \cite{FLECKINGER-1995} for $\Omega\subset \R^2$ being the triadic von Koch snowflake, for which the heat content is
given by ~
$$ 
N(t)=t^{\frac{2-d}{2}}p(\ln t) - t q(\ln t) + O(e^{-\frac{1}{1152t}}),$$
 as $t\to 0+$, where $d=\log 4/\log 3$ is the Minkowski dimension of $\del \Omega$ and $p$ and $q$ are continuous $(\ln 9)$-periodic functions;
\item[(5)] If $\Omega\subset\R^n$ is a \emph{self-similar fractal spray} (i.e., a certain countable disjoint union of scaled similar copies of a given basic open set), then in~\cite{LEVITIN-1996} the heat content was given up to the second-order term by 
$$
N(t)= t^{\frac{n-d}{2}} \nu+o(t^{\frac{n-d}{2}}),
$$
where, as before, $d$ denotes the Minkowski dimension of $\del \Omega$. Moreover, if the generating IFS is nonlattice, then $\nu$ is a constant, while in the case of a lattice IFS, $\nu$ is a positive continuous periodic function in $t$. 
\end{enumerate}
Observe the occurrence of the boundary length/surface area in the first asymptotic term in the classical cases (1) and (2) and of some curvature expression in the constant of the second asymptotic term in these two cases. 
Note that the order of the first term in these cases is $\frac 12=\frac {n-(n-1)}2$, where $n-1$ is the (Minkowski) dimension of the boundary and that the surface area may be regarded as the ($n-1$-dimensional) Minkowski content of $\del \Omega$. 
In the two media problem (with more general boundary conditions) studied in Section \ref{SubsecBardos}, we observe (for $C^3$-boundaries) a similar (but more complicated) dependence of the first asymptotic term on the parallel volume and hence on some notion related to the Minkowski content, see e.g.~formulas~(\ref{EqNilocFint2C}) and~(\ref{EqInfNtreg}). 

In the fractal cases (3)-(5), the order of the first term  is $\frac {n-d}{2}$, where again $d$ is the (Minkowski) dimension of $\del\Omega$. But in these fractal cases it is not easy to understand how precisely the geometry of the boundary influences the first (and the further) asymptotic terms. From the rather general estimates in (3) one can see that the growth of the parallel volume $\mu(\del \Omega, r)$ plays some role. We conjecture that for fractal boundaries, curvature measures of parallel sets—potentially in the form of fractal curvatures, as explored in Sections \ref{sec:curv-meas}-\ref{sec:rel-fc}, or some related notion—will influence even the leading asymptotic term of the heat content. This conjecture seems natural to us, considering the second terms in equations (1) and (2) and the scaling behavior of curvature measures of parallel sets of fractals.

\section{Approximation of the interior Minkowski sausage}\label{SecMS}
\subsection{Domains with Minkowski measurable boundary and their convergence}\label{ssMinkMeas}

Recall that a bounded set $F\subset\R^n$ is \emph{$d$-dimensional Minkowski measurable} (for some $d\in[0,n]$) if its $d$-dimensional Minkowski content is positive and finite. In the sequel, we assume Minkowski measurability of the boundary $\del \Omega$ \emph{relative to $\Omega$}, i.e.,
we assume that the limit
 	\begin{equation}\label{EqSteffen}
 		\mathcal{M}^d(\del \Omega,\Omega):=\lim_{r\to 0} \frac{\mu(\del \Omega,r)}{r^{n-d}}
 	\end{equation}
exists as a positive and finite number, where $\mu(\del \Omega,r)$ is the interior parallel set as defined in \eqref{EqDefmu}. %
Observe that this is equivalent to assuming
\begin{equation}\label{EqMainAss}
	\mu(\del \Omega,r)=\mathcal{M}^d(\del \Omega,\Omega)r^{n-d}+o(r^{n-d}), \quad \text{ as } r\to0+.
\end{equation}
Note further that, since $\del \Omega$ is a boundary, the exponent $d$, which is also called the \emph{Minkowski dimension} of $\del\Omega$ (\emph{relative to} $\Omega$), must be at least $n-1$. We exclude the case $d=n$ and assume that $d\in[n-1,n)$. The case of a regular boundary is readily included: here one has $d=n-1$ and $$\mathcal{M}^{n-1}(\del \Omega,\Omega)=\sH^{n-1}(\del \Omega)=\int_{\del \Omega} 1 \,\sH^{n-1}(d \sigma).$$
%
If $\del \Omega$ is at least $C^3$, then
\eqref{EqMainAss} can be written explicitly in local coordinates: 
\begin{align}
\mu(\del \Omega,r )&=\int_{\del \Omega}\sH^{n-1}(d\sigma)\int_0^r ds	|J(s,\sigma)|\notag\\
&\label{EqAppJFliss} =\sH^{n-1}(\del \Omega)r +o(r ) \quad \text{ as } r \to 0+,
\end{align}
where $J(s,\sigma)$ denotes the Jacobian of the mapping $(s,\sigma)\mapsto \sigma+s n$ and $n$ is the unit outer normal of $\Omega$ at $\sigma$, see also \eqref{Jacobian}.
\begin{proposition}
	Let $z\in (0,\infty)$, $d\in[n-1,n)$ and let $\Omega$ be a bounded domain in $\R^n$ with a $d$-dimensional Minkowski measurable boundary $\del \Omega$.
 	Then
 	\begin{equation}\label{EqExpmu1}
		\mu(\del \Omega,zr)=z^{n-d}\mu(\del \Omega,r)+o(r^{n-d}) \quad \text{ as } r\to 0+.
	\end{equation}
	
\end{proposition}
\begin{proof}
	The assumed Minkowski measurability implies that
	$$\left|\frac{\mu(\del \Omega,zr)}{(zr)^{n-d}} -\frac{\mu(\del \Omega,r)}{r^{n-d}} \right|= o(1), \hbox{ as } r \to 0+.$$
	Multiplying by $(zr)^{n-d}$, yields the assertion~\eqref{EqExpmu1}.
\end{proof}
Let us recall the definition of an $(\eps,\infty)$-domain:
\begin{definition}
	\label{DefEDD}
	Let $\eps > 0$. A bounded domain $\Omega\subset \R^n$ is called an \emph{$(\eps,\infty)$-uniform domain} if for all $x, y \in \Omega$ there is a rectifiable curve $\gamma\subset \Omega$ with length $\ell(\gamma)$ joining $x$ to $y$ and satisfying
	\begin{enumerate}
		\item[(i)] $\ell(\gamma)\le \frac{|x-y|}{\eps}$, and
		\item[(ii)] $d(z,\del \Omega)\ge \eps |x-z|\frac{|y-z|}{|x-y|}$ for $z\in \gamma$.
	\end{enumerate}
\end{definition}
In the sequel we impose throughout the following hypothesis on a domain $\Omega$ and a sequence $(\Omega_i)_{i\in\N}$ of domains approximating $\Omega$ from inside. 
\begin{assumption}\label{Assump}
	$\Omega\subset \R^n$ is a fixed bounded $(\eps,\infty)$-domain, and there exists an increasing sequence of $(\eps,\infty)$-subdomains $(\Omega_i)_{i\in\N}$ of $\Omega$ (with the same fixed $\eps>0$) such that all $\Omega_i$ and $\Omega$ contain a fixed nontrivial ball $B_{r_0}(x)$ and such that
	$\Omega_i\nearrow \Omega$, $\Omega=\cup_{i\in \N} \Omega_i$.
	In addition, the boundaries $\del \Omega_i$ and $\del \Omega$ are assumed to be supports of finite Borel measures $\nu_i$ and $\nu$, respectively, such that $\nu$ is the weak limit of the measures $\nu_i$ as $i\to +\infty$, and there are constants $d\in[n-1,n[$, $c_{n-1}>0$ and $c_d>0$ such that for all $i\in \N$ the measures $\nu_i$ satisfy 
	\begin{equation}\label{EqIURegmeas}
		\nu_i(B_r(x))\le c_{n-1}r^{n-1}, \quad \nu_i(\overline{B_r(x)})\ge c_dr^d, \; x\in \del \Omega_i, \; 0<r\le 1.
	\end{equation}
\end{assumption}
This kind of assumption was previously used in~\cite{DEKKERS-2022}.
We recall now the definition of convergence in the Hausdorff metric for open sets, see e.g.~\cite[Definition 2.2.8]{HENROT-e}.
\begin{definition} \label{def:d_H}
    The \emph{Hausdorff distance} between two compact sets $K, L\subset \mathbb{R}^n$ is defined as 
\[d_H(K,L):=\inf\{\alpha>0: K\subset L_\alpha \text{ and } L\subset K_\alpha\}.\]  
A sequence $(K_m)_m$ of compact sets $K_m\subset \mathbb{R}^n$ is said to \emph{converge} to a compact set $K\subset \mathbb{R}^n$ \emph{in the Hausdorff sense} if $\lim_{m\to \infty} d_H(K_m,K)=0$. Let $D\subset \mathbb{R}^n$ be a bounded open set. A sequence $(\Omega_m)_m$ of open sets $\Omega_m\subset D$ is said to \emph{converge} to an open set $\Omega\subset D$ \emph{in the Hausdorff sense} if 
\[d_H(\overline{D}\setminus \Omega_m,\overline{D}\setminus \Omega)\to 0\quad \text{ as }\quad m\to \infty,\] 
Note that this convergence does not depend on the choice of $D$, see e.g.~\cite[Remark 2.2.11]{HENROT-e}.
\end{definition}

Let us also prove the following useful result:
\begin{proposition}\label{PropConvProp}
	Let $\Omega$ be a bounded domain in $\R^n$ and let $(\Omega_i)_{i\in\N}$ be an increasing sequence of subdomains (i.e. with $\Omega_i\subseteq\Omega_{i+1}\subseteq\Omega$) converging to $\Omega$ in the Hausdorff metric. 
Then there exists, for each $\delta>0$, some index $i_0=i_0(\delta)\in \N$ such that, for all $i\geq i_0$ and all $r>0$,
		\begin{equation}\label{EqConvMinkS}
		\left|\mu(\del\Omega_i,r)-\mu(\del\Omega,r)\right|\leq\delta.
	\end{equation}
Moreover, if $\Omega$ and $\Omega_i$ are $(\eps,\infty)$-domains satisfying Assumption~\ref{Assump}, then the measure $\nu$ satisfies \eqref{EqIURegmeas}.
	\end{proposition}
\begin{proof}
For $r>0$ let $\Omega_r:=\{y\in \Omega|\, d(y,\del \Omega)<r\}$ and $\Omega_{i,r}:=\{x\in \Omega_i|\, d(x,\del \Omega_i)<r\}$.
Let $\delta>0$. Observe that, for all $r>0$,
\begin{equation*}
		|\lambda_n(\Omega_{i,r})-\lambda_n(\Omega_r)|\leq \max\left\{\lambda_n(\Omega_{r}\setminus\Omega_{i,r}),\lambda_n(\Omega_{i,r}\setminus\Omega_r)\right\}.
	\end{equation*}
Since $\Omega_i\subset \Omega$, we have $\Omega_r\setminus \Omega_{i,r}\subset \Omega\setminus\Omega_i$, and $\Omega_i\nearrow\Omega$ (in the Hausdorff metric) now implies $\lambda_n(\Omega\setminus\Omega_i)\searrow 0$, as $i\to\infty$. Hence, one can find $i_1\in\N$ such that $\lambda_n(\Omega\setminus\Omega_i)\leq\delta$ and thus, in particular,
$$
\lambda_n(\Omega_r\setminus\Omega_{i,r})\leq\delta \quad \text{ for all } i\geq i_1 \text{ and } r>0.
$$
To bound the second volume $\lambda_n(\Omega_{i,r}\setminus\Omega_r)$ in the above maximum, observe first that it equals $0$ for all $r>\rho$, where $\rho:=\max\{\dist(x,\del \Omega): x\in\Omega\}$ denotes the inradius of $\Omega$.

Furthermore, for any $\eps>0$, we can find $i_2=i_2(\eps)\in\N$ such that for all $i\geq i_2$ and all $r\in(0,\rho]$
$$
\Omega_{i,r}\setminus\Omega_r\subset \Omega_{r+\eps}\setminus\Omega_r.
$$
Hence it is enough to show that for the given $\delta>0$ there is some $\eps>0$ 
such that
\begin{align} \label{eq:parbound}
\lambda_n(\Omega_{r+\eps}\setminus\Omega_r)\leq\delta \text{ for all } r\in(0,\rho].
\end{align}
(Then obviously the same holds for any $\eps'<\eps$.) For a proof of \eqref{eq:parbound}, choose $r_1=r_1(\delta)>0$ small enough that $\lambda_n(\Omega_{2r_1})\leq\delta$ (which is possible, since $\lambda_n(\Omega_{r})\searrow 0$ as $r\searrow 0$). Then we have for all $\eps,r\in(0,r_1]$
$$
\lambda_n(\Omega_{r+\eps}\setminus\Omega_r)\leq\lambda_n(\Omega_{2r_1})\leq\delta,
$$
showing the assertion for all $r<r_1$.
Moreover, by \cite[Cor.~4.2]{RW10}, there is some constant $S>0$ such that, for all $r\in[r_1,\rho]$, $\sH^{n-1}(\del(\Omega_r)\cap \Omega)\leq S$.
Recalling that the derivative of the volume function $t\mapsto\lambda_n(\Omega_t)$ at some $t=r$ equals the surface area $\sH^{n-1}(\del(\Omega_r)\cap\Omega)$ for a.e.\ $r>0$, see e.g.~\cite[Cor.~2.5]{RW10} (and that $\del(\Omega_r)\cap\Omega=\emptyset$ for $r>\rho$), we conclude that, for all $r\in[r_1,\rho]$ and any $\eps>0$,
$$
\lambda_n(\Omega_{r+\eps}\setminus\Omega_r)=\int_r^{r+\eps} \sH^{n-1}(\del(\Omega_t)\cap\Omega) \,d t\leq S\cdot\eps.
$$
Hence, for any $\eps\leq\eps_1:=S^{-1}\delta$, the desired bound
$
\lambda_n(\Omega_{r+\eps}\setminus\Omega_r)\leq \delta
$
holds for all $r\in[r_1,\rho]$.
We conclude that \eqref{eq:parbound} holds uniformly in $r\in(0,\rho]$ for any $\eps\leq\min\{\eps_1,r_1\}$.

The fact that $\nu$ satisfies \eqref{EqIURegmeas} is proven in~\cite[Proposition~3.3 (ii)]{HINZ-2021-1}. 
\end{proof}

	\begin{remark}\label{RemC3App}
    Let us recall three types of convergence for domains contained in a bounded domain $D\subset\R^n$. Let $\Omega$ and $\Omega_m$, $m\in\N$, be subdomains of $D$. 
    \begin{enumerate}
  \item[(i)] For the convergence of $\Omega_m\to\Omega$ in the Hausdorff sense see Definition~\ref{def:d_H}.
 \item[(ii)] $(\Omega_m)_{m\in \N}$ is said to converge to $\Omega$  \emph{in the sense of compacts}, if 
 for all compact subsets $K$ of $\Omega$ and all compact subsets $L$ of $D\setminus \overline{\Omega}$ it follows that $K\subset \Omega_m$ and $L\in D\setminus \overline{\Omega_m}$ for all sufficiently large $m$,
 \item[(iii)] $(\Omega_m)_{m\in\N}$ is said to converge to $\Omega$ \emph{in the sense of characteristic functions}, if
 $$\chi_{\Omega_m}\to \chi_{\Omega} \text{ as } m \to +\infty \quad \text{in  }L^p_{loc}(\R^n) \quad \text{for all } p\in [1,\infty[.$$
 Here $\chi_{A}(=\1_A)$ denotes the  characteristic function of a set $A$ taking value $1$ for $x\in A$ and $0$ otherwise.
 \end{enumerate}
 Once a sequence $(\Omega_i)$ of domains is monotonically increasing towards $\Omega$ (s.t. $\Omega=\cup_i \Omega_i$, see Assumption~\ref{Assump}), then this sequence converges to $\Omega$ by these tree types of domain convergences~\cite{CLARET-2024-1}.  
	The compactness corresponding to these three types of domain convergences,     
    and to the weak convergence of the boundary measures of the class of domains introduced in Proposition~\ref{PropConvProp} is proven in~\cite{HINZ-2021-1}. Using the notation of~\cite{HINZ-2021-1}, 
we will write 
$$
U_{ad}(D,B_{r_0}(x),\eps,d,n-1,c_{n-1},c_d)
$$ 
for the family of all domains $\Omega$ contained in the domain $D\subset\R^d$, containing the ball $B_{r_0}(x)$ with center $x\in D$ and radius $r_0$, and satisfying Assumption~\ref{Assump} for  the constants $\eps$ (from the condition to be an $(\eps,\infty)$-domain), $n$, $d$, $c_{n-1}$ and $c_d$ (from condition \eqref{EqIURegmeas} for the boundary measures). 
\end{remark}
\begin{proposition}\label{PropC3conv}
	For any bounded domain $\Omega$ in $\R^n$
    of the class 
 $$
 U_{ad}(D,B_{r_0}(x),\eps,d,n-1,c_{n-1},c_d)$$ as introduced above, 
 there exists an increasing sequence $(\Omega_i)_{i\in \N}$ of $C^3$-domains $\Omega_i\subset \Omega$ satisfying Assumption~\ref{Assump} such that
			\eqref{EqConvMinkS} holds uniformly in $\eta\in]0,\eta_0[$ for some $\eta_0>0$ (depending on the diameter of $\Omega$) and such that, in addition, for all $T>0$ the corresponding heat contents $N_i(t)=\lambda_n(\Omega_i)-\int_{\Omega_i}u_{i,+}(x,t)\dx$ converge uniformly in $t\in ]0,T[$ to $N(t)=\lambda_n(\Omega)-\int_{\Omega}u_{+}(x,t)\dx$ as $i\to +\infty$.
	\end{proposition}

\begin{proof}
	Without loss of the generality, let us consider a prefractal sequence of Lipschitz boundaries converging monotonously to a fractal or a $d$-set one, for an example, see~\cite{CAPITANELLI-2011,DEKKERS-2022}. The construction is performed in a way to keep all parameters $B_{r_0}(x)$, $\eps$, $d$, $c_d$, $c_{n-1}$ fixed. To construct such a sequence, it is always possible to make an interior dyadic grid of $\Omega$ with a $d$-set boundary first with the minimal length (between two boundary edges) equal to $\ell$, see~\cite{CLARET-2024-1}. Then we take the largest Lipschitz domain $\Omega_1$ belonging $U_{ad}(\Omega,B_{r_0}(x),\eps,d,n-1,c_{n-1},c_d)$ included in (or equal to) the largest interior $\ell$-grid-based subdomain. The next step is to take $\ell/2$ and to take the corresponding dyadic grid of $\Omega$. The largest subdomain from $U_{ad}(\Omega,B_{r_0}(x),\eps,d,n-1,c_{n-1},c_d)$ containing $\Omega_1$ is denoted by $\Omega_2$. Dividing by $2$ at each step the size of the squares, we obtain a monotonous increasing sequence of Lipschitz subdomains converging by three types of convergences to the initial domain $\Omega$ with the weak converge of the corresponding boundary volume measures. Then each Lipschitz boundary can be approximated by a regular boundary by~\cite[Thrm 3.42, p.147]{AMBROSIO-2000}. It means for each bounded domain with a $d$-set boundary, there exists a $C^3$ sequence of domains, which satisfies the hypothesis of Proposition~\ref{PropConvProp} and ensures~\eqref{EqConvMinkS}.
	
	For the constructed sequence of subdomains $(\Omega_i)_{i\in \N}$, by the $L^2$-Mosco convergence of the quadratic forms on $U_{ad}(\Omega,B_{r_0}(x),\eps,d,n-1,c_{n-1},c_d)$ (see~\cite{claret:hal-04316274}  for a detailed proof), and thanks to the positivity of $u_{i,+}$ as the heat solutions on $\Omega_i$, we obtain that
	$$
 \int_{\Omega_i} u_i(x,t) \dx=\|u_i(t)\|_{L^1(\Omega_i)}\to \|u(t)\|_{L^1(\Omega)},\quad \hbox{as } i\to+\infty,
 $$
	uniformly in $t\in ]0,T[$ for all $T>0$.	
\end{proof}
%

\begin{remark}\label{RemConvRests} 
	If there is a sequence of $C^3$-domains $(\Omega_i)_{i\in \N}$ (as in Proposition~\ref{PropC3conv}) converging to $\Omega$ with a $d$-set boundary such that, as $i\to +\infty$,
$\mu(\del \Omega_i,r)\to \mu(\del \Omega,r)$ converges  uniformly in $r$ for $r$ small enough, 
then their asymptotic decomposition on $r$, having the same general structure~\eqref{EqMainAss}, converges:
\begin{multline}\label{EqSeqVolumes}
	\mu(\del \Omega_i,r)=\mathcal{H}^{n-1}(\del \Omega_i)r^{n-(n-1)}+o_i(r^{n-(n-1)})\\
	\to \mu(\del \Omega,r)=\mathcal{M}^d(\del \Omega,\Omega)r^{n-d}+o(r^{n-d}).
\end{multline}
In particular, this convergence ensures that for $i\to +\infty$ the remainder terms $o_i(r^{n-(n-1)})$ converge to the remainder term for $\Omega$, $o(r^{n-d})$.
\end{remark}

\subsection{Minkowski nondegenerate boundaries}\label{SubNonMM}

Let $\Omega$ be a bounded domain of $\R^n$ with a $d$-set boundary, i.e. there exists a positive
Borel measure $\nu$ (a $d$-measure) with support $\del \Omega$ such that for some positive constants $c_1$, $c_2>0$,
\begin{equation*}
c_1r^d\le \nu(\del \Omega\cap B_r(x))\le c_2 r^d, \quad \hbox{ for } ~ x\in \del \Omega,\; 0<r\le 1,
\end{equation*}
where $B_r(x)\subset \R^n$ denotes the Euclidean ball centered at $x$
and of radius~$r$.
In general, $d$-sets  cannot be expected to be ($d$-dim.) Minkowski measurable, i.e., the limit in~\eqref{EqSteffen} might not exist. However, in the setting of domains with piecewise self-similar boundaries (as studied in Sections \ref{fractal bound}- \ref{sec:rel-fc} below), it is ensured that the Minkowski dimension of the boundary is $d$ and, moreover, that $\del \Omega$ is \emph{($d$-dimensional) Minkowski nondegenerate}, i.e., 
%
%
\begin{equation}\label{Eqliminfsup}
	0<\liminf_{r\to 0}\frac{\mu(\del \Omega,r)}{r^{n-d}} \le \limsup_{r\to 0}\frac{\mu(\del \Omega,r)}{r^{n-d}}<+\infty.
\end{equation}

%
In \cite[Theorem 1.1]{LapidusPearse2006} (see also \cite{Lapidus-P-2008,Lapidus-P-2010,Lapidus-P-W-2011,LapidusPearse}) it has been shown that the parallel volume 
the Koch snowflake domain $\Omega\subset \R^2$
has the form
\begin{equation}\label{eq-LP1}
\mu(\del \Omega,r )=G_1(r )r ^{2-d}+G_2(r )r ^{2}
\end{equation}
where $d=\frac{\log4}{\log3}$ is the Minkowski dimension of $\del\Omega$  and $G_1,G_2$ are bounded multiplicatively periodic functions of multiplicative period $3$. 
This formula is interesting because of the presence of oscillations in the main and in the second term and also because of the absence of a term of order $1$.

From what is known about the Minkowski content of self-similar sets (satisfying the open set condition), see \cite{Ga00}, one would expect a similar expansion of the parallel volume for any domain in $\R^n$ with piecewise self-similar boundary, possibly with finitely many oscillating terms instead of two and likely with some remainder term. Clearly, the exponent of the leading term will always be $n-d$, where $d$ is the Minkowski dimension of $\del \Omega$. It is well known that self-similar sets are Minkowski non-degenerate and this should carry over to piecewise self-similar boundaries. This would imply in particular, that
\begin{equation}\label{eq-LP1-general}
\mu(\del \Omega,r )=G(r )r^{n-d}+o(r^{n-d}),
\end{equation}
as $r\to 0+$, where $G$ is some positive and bounded periodic function. If all the self-similar sets forming the boundary are nonlattice, then the function $G$ would be constant, while otherwise one would expect $G$ to have nontrivial oscillations.
From \eqref{eq-LP1-general} we would obtain, for any $z>0$,
\begin{equation}\label{eq-LP2}
\mu(\del \Omega,zr )=G(zr )z^{n-d}r ^{n-d}+o(r^{n-d}).
\end{equation}
Since $G$ is positive and bounded, \eqref{eq-LP2} implies that 
\begin{equation}\label{eq-LP2-bound}
|\mu(\del \Omega,zr )-z^{n-d}\mu(\del\Omega, r)|\leq G_{\max} z^{n-d} r^{n-d}+o(r^{n-d}),
\end{equation}
where $G_{\max}=\sup_{r,s>0} |G(r)-G(s)|$. This should be compared with  \eqref{EqExpmu1} in the Minkowski measurable case. Moreover, since $G$ is multiplicative periodic with some period $h$, one even has \eqref{EqExpmu1} along the period, i.e., for  $z=h^k$, $k\in\N$.



Numerically we expect that $\sup_{r,s >0}|G(r )-G(s)|$
should be rather small and provide a small but noticeable influence on the heat kernel asymptotics; see \cite{StrVicsek,StrSig,ADT} and related articles.

\section{Asymptotics of the heat content for a two media problem using the volume of the interior Minkowski sausage}\label{SubsecBardos}

In this section, we examine the transmission problem between a hot, compact medium and an initially cold complement, where the interface may resist heat propagation. We explore the relationship between the short-time asymptotics of the heat content and the volume of the Minkowski sausage, thereby investigating the de Gennes hypothesis.
Our results are a refinement of those in~\cite{BARDOS-2016}, which we partially recall here.
We begin by introducing the model and briefly discussing its well-posedness (for further details, see~\cite{BARDOS-2016,claret:hal-04316274}).

In~\cite{BARDOS-2016} a bounded domain $\Omega \subset \R^n$ was considered with a boundary
$\del\Omega$ that splits $\R^n$ into a ``hot'' and a ``cold'' medium,
$\Omega_+ = \Omega$ and $\Omega_- = \R^n \setminus\overline{\Omega}$,
characterized by (distinct) heat diffusion coefficients $D_+$ and
$D_-$, see~\cite[Fig.~1]{BARDOS-2016}. On the boundary $\del\Omega$ a function $0\le\Lambda(x)\le \infty$ is
defined, which describes the
resistivity to heat exchange through the boundary.
%
The heat content propagation is associated with
the following transmission problem, which we present firstly in its strong (physical) form:
\begin{align}
  \del_t u_\pm- D_{\pm} \Delta u_\pm &=0 \quad x\in \Omega_\pm, \; t>0,\label{prb1}\\
  u_+|_{t=0}&= 1, \, x\in \Omega_+, \quad u_-|_{t=0}=0,  \, x\in \Omega_-\label{prb1i}\\
  D_- \frac{\del u_-}{\del n}|_{\del \Omega}&=\Lambda(x)(u_--u_+)|_{\del \Omega}, \; t>0,\label{prb1c} \\
  D_+ \frac{\del u_+}{\del n}|_{\del \Omega}&=D_- \frac{\del u_-}{\del n}|_{\del \Omega}, \; t>0,\label{prb1end}
\end{align}
where $\partial/\partial n$ is the normal derivative directed outside
the domain $\Omega$, well-defined for all $x\in \del \Omega$ for $C^1$-boundary.
Mathematically, especially for the case of non-Lipschitz interfaces $\del \Omega$, the problem~\eqref{prb1}--\eqref{prb1end} is understood in the variational (weak) sense as formulated in~\eqref{varpr} below, see~\cite{BARDOS-2016,claret:hal-04316274}. 

The boundary between the two media can have some resistance to
heat exchange, described by the function $\Lambda(x)\geq 0$ ($x\in
\del \Omega$) that may account for partial thermal isolation. Typically, there are three different types of boundary conditions considered corresponding to different values of $\Lambda$:
\begin{enumerate}
\item
If $\Lambda(x)=\Lambda$ for some strictly positive constant $0<\Lambda <\infty$ and all $x\in \del \Omega$, then $u$ is discontinuous at $\del \Omega$ and we have:
\begin{equation*}
\left(\Lambda u_--D_- \frac{\del u_-}{\del n}\right)|_{\del
\Omega}=\Lambda u_+|_{\del \Omega}, \quad D_+
\frac{\del u_+}{\del n}|_{\del \Omega}=D_- \frac{\del u_-}{\del
n}|_{\del \Omega}.
\end{equation*}

\item
If $\Lambda=+\infty$ for all $x\in \del \Omega$, then $u$ is continuous on
$\del \Omega$ due to the transmission condition and in this case
\begin{equation*}
u_+|_{\del \Omega}=u_-|_{\del \Omega},\quad D_+ \frac{\del u_+}{\del n}|_{\del
\Omega}=D_- \frac{\del u_-}{\del n}|_{\del \Omega}.
\end{equation*}

\item
If $\Lambda=0$ for all $x\in \del \Omega$, then we have the Neumann
boundary condition
\begin{equation*}
\frac{\del u_-}{\del n}|_{\del \Omega}=\frac{\del u_+}{\del n}|_{\del \Omega}=0
\end{equation*}
that models the complete thermal isolation of $\Omega$ and
implies the trivial solution given by $u_-(x,t) = 0$ and $u_+(x,t) =
1$ for all times $t\ge 0$.
\end{enumerate}
In problem~(\ref{prb1})--(\ref{prb1end}) it is assumed that the physical properties of the two media $\Omega_+$ and $\Omega_-$ are
different: $D_+ \ne D_-$. As mentioned in~\cite{BARDOS-2016}, this implies the
discontinuity of the metric on $\del \Omega$. 
The case of a continuous Riemannian
metric ($g_-|_{\del \Omega}=g_+|_{\del \Omega}$) on a smooth compact
$n$-dimensional Riemannian manifold with a smooth boundary $\del
\Omega$ was considered in~\cite{GILKEY-2003}. The case of continuous
transmission boundary conditions for the expansion of the heat kernel
on the diagonal was treated in~\cite{PIROZHENKO-2005}. In~\cite{VASSILEVICH-2003}, there is a survey of results on the asymptotic
expansion of the heat kernel for different boundary conditions.
In~\cite{BARDOS-2016} it was established that the problem~(\ref{prb1})--(\ref{prb1end}) is weakly well-posed whenever $\del \Omega$ is a $d$-set. However, thanks to the trace theorem~\cite[Theorem 5.1]{HINZ-2021-1}, it is possible to extend the well-posedness to upper $d$-regular boundaries by adding the condition that $\Omega$ and its complement $\Omega_-$ are extension domains (or, for short, that $\Omega$ is a two-sided extension domain), see~\cite[Theorem~2.7]{claret:hal-04316274}.
Recall that a set $\Omega\subset\R^n$ is an \textit{($H^1$-)~extension domain}~\cite{HAJLASZ-2008}, if and only if there exists a bounded linear extension operator $E:H^1(\Omega)\to H^1(\R^n)$, $\left.Eu\right|_\Omega=u$ such that
for all $u\in H^1(\Omega)$, $\|Eu\|_{H^1(\R^n)}\le C\|u\|_{H^1(\Omega)}$,
with $C>0$ depending only on $n$ and $\Omega$. Moreover, upper $d$-regularity (for some fixed $d\in [n-1,n[$) means the existence of a positive Borel measure
 $\nu$ with support equal to $\del \Omega$ satisfying the following condition: there exists $c_d>0$ such that, for all $x\in\del \Omega$ and all $r\in]0,1]$, 
\begin{equation}\label{Eqmu}
 \quad\nu(B_r(x))\le c_dr^d.
\end{equation}
Here $B_r(x)$ denotes the open ball in $\R^n$ with center $x$ and radius $r$. Such two-sided extension domains with a connected boundary $\del \Omega=\operatorname{supp} \nu$ and endowed with an upper $d$-regular measure $\nu$,
are shortly called \textit{admissible domains}.
In many cases $\nu$ can be chosen to be the $d$-dimensional Hausdorff measure, or, specifically, in case of a $1$-dimensional boundary, the boundary length. 
Condition~\eqref{Eqmu} implies that the Hausdorff dimension of the boundary is (locally) at least $d$. However, the condition is flexible enough to allow boundaries with locally varying dimension. For example, some part of the boundary may have dimension $d_1$, another one dimension $d_2$, with $d\le d_1<d_2< n$ and $n-1\le d<n$. Some parts may be Lipschitz, only if $d=n-1$. The weak well-posedness of problem ~\eqref{prb1}--\eqref{prb1end} for admissible domains $\Omega$ is discussed in detail in~\cite{claret:hal-04316274}. Here we just present the results.

For the weak well-posedness on admissible domains we introduce the space
\begin{equation*}
V:=\left\{f\in L^2(\R^n)|\; f_+=f|_{\Omega_+}\in H^1(\Omega_+),\;\hbox{and } f_-=f|_{\Omega_-}\in
H^1(\Omega_-)\right\}
\end{equation*}
of functions $f=f_+\mathds{1}_{\Omega_+}+f_-\mathds{1}_{\Omega_-}$
defined on $\Omega_+\cup\Omega_-$ such that their restrictions
$f_+=f|_{\Omega_+}$ and $f_-=f|_{\Omega_-}$ belong to $H^1$. We equip
$V$ with the norm
\begin{equation*}
 \|u\|_V^2=D_+\int_{\Omega_+} |\nabla u_+|^2\dx+D_-\int_{\Omega_-} |\nabla u_-|^2 \dx+\int_{\Omega_+\cup\Omega_-} |u|^2 \dx.
\end{equation*}
Note that $V$ is a Hilbert space, $V\subset L^2(\R^n)$, and $V$
is dense in $L^2(\R^n)$. In addition, $V\subset L^2(\R^n) \subset
V'$, where $V'$ is the dual space to $V$.
For further simplification of technical details, we choose to restrict here to the case that the boundary resistivity parameter $\Lambda$ has a constant value on all of $\del \Omega$: a constant $\Lambda\in]0, \infty[$, or $+\infty$. We refer to~\cite{BARDOS-2016,claret:hal-04316274} for a mixed case and the case when $\Lambda$ is a continuous function of $x$.
Thus, for $\Lambda\in]0, \infty[$ on $\del \Omega$, using in the usual way the continuous and coercive bi-linear form $\alpha_\Lambda$ on $V\times V$ with the notation $\nu$ for the measure on $\del \Omega$, defined by
\begin{multline}\label{VarForm1}
 a_\Lambda(u,v)=D_+\int_{\Omega_+} \nabla u_+ \nabla v_+dx+D_-\int_{\Omega_-} \nabla u_- \nabla v_-dx\\+
 \int_{\del \Omega}\Lambda\operatorname{Tr}(u_+-u_-)\operatorname{Tr}(v_+-v_-) d \nu,
\end{multline}
we obtain the weak well-posedness of problem~(\ref{prb1})--(\ref{prb1end})
in the following sense:
there exists a unique solution $u\in
C(\R^+_t,L^2(\R^n))\cap L^2(\R^+_t,V)$ of the variational problem
\begin{eqnarray}
 &&\forall v\in V\quad \frac{d}{dt}\langle u,v \rangle_{L^2(\R^n)} + a_\Lambda(u,v)=0,\quad u(x,0)=u_0\in L^2(\R^n).\label{varpr}
\end{eqnarray}
The trace operator $\operatorname{Tr}$ on the boundary of $u_\pm$ is understood as interior/exterior trace on $\del \Omega$~\cite{CLARET-2023,claret:hal-04316274} (i.e., equivalence classes of pointwise restrictions of quasi-continuous representatives $\tilde{u}_\pm$ on $\del \Omega$ of $u_\pm \in H^1 (\Omega_\pm)$, extended to $Eu_\pm\in H^1(\R^n)$).
If $\Lambda=+\infty$ on $\del \Omega$, then the boundary term in~(\ref{VarForm1}) disappears and hence $V=H^1(\R^n)$. Taking in \eqref{varpr} the bilinear form
\begin{equation}
	\label{VarFormInf}
 a_\infty(u,v)=D_+\int_{\Omega_+} \nabla u_+ \nabla v_+dx+D_-\int_{\Omega_-} \nabla u_- \nabla v_-dx,
\end{equation}
we obtain the same type of well-posedness result for the case $\Lambda=\infty$.
For the continuous dependence of $u$ on $\Lambda$ see \cite[Theorem~2.2]{BARDOS-2016}.

Once a unique solution $u_{\pm}$ of the
problem~(\ref{prb1})--(\ref{prb1end}) is established, the
heat content at time $t>0$ is defined by
\begin{equation}\label{Nt}
 N(t):=\int_{\R^n\setminus \Omega} u_-(x,t)\dx =\lambda_n(\Omega)-\int_\Omega u_+(x,t)\dx,
\end{equation}
where the last equality is due to the fact that, by definition, the Green's function $G(x,y,t)$ of the model satisfies $\int_\Omega G(x,y,t)\dy=u_+(x,t)$ as well as $\int_{\R^n} G(x,y,t)\dy =1$, so that
\begin{align*}
N(t)&=\int_{\R^n\setminus\Omega} u_-(x,t)\dx
=\int_{\R^n\setminus\Omega}\int_\Omega G(x,y,t)\dy \dx\\ &=\int_{\R^n}\int_\Omega G(x,y,t)\dy \dx-\int_\Omega \int_\Omega G(x,y,t)\dy \dx
= \lambda_n(\Omega)-\int_\Omega u_+(x,t) \dx.
\end{align*}
%

We are interested in asymptotic expansions of the heat content $N(t)$ as $t\to 0+$. Let us denote by
$\Omega_r$ and $\Omega_{-r}$ the open $r$-neighborhoods of
$\del \Omega$ in $\Omega$ and in $\R^n\setminus\Omega$, respectively.
According to~\cite[Lemma 3.1]{BARDOS-2016},
there exists some $\delta_0>0$ (a constant independent of time) and some $\eta(t)>0$ (more precisely, $\eta(t)>\sqrt{4D_+ t}$, $\eta(t)=O\left(\sqrt{t}t^{-\delta}\right)$ for some $\delta\in ]0,\frac{1}{2}[$), such that
 the heat content
$N(t)$ satisfies
\begin{equation}\label{ErrOm}
 N(t) 
 =\int_{\Omega_{\eta(t)}}\left(1-\int_{\Omega_{\eta(t)}} G(x,y,t)\dy\right)\dx
+ O\left(e^{-\frac{1}{t^{\delta_0}}}\right), 
\end{equation}
as $t\to 0+$, where $G(x,y,t)$ is the Green's function of our problem~(\ref{prb1})--(\ref{prb1end}). Note that formula~\eqref{ErrOm} holds for any bounded domain $\Omega$ in $\R^n$ for which~\eqref{varpr} is well-defined. %
In other words, for all bounded admissible domains $\Omega$, we have
\begin{equation}\label{EqNtNHB}
	N(t)-\mu(\del \Omega,\eta(t))+\int_{\Omega_{\eta(t)}}\int_{\Omega_{\eta(t)}} G(x,y,t)\dy\dx=O\left(e^{-\frac{1}{t^{\delta_0}}}\right).
\end{equation}
Clearly, the approximation of $N(t)$ depends on the volume of the Minkowski sausage $\mu(\del \Omega,\eta(t))=\int_{\Omega_{\eta(t)}}\dx$. In order to obtain a more explicit formula, we need to approximate the term $\int_{\Omega_{\eta(t)}}\int_{\Omega_{\eta(t)}} G(x,y,t)\dy\dx$. 

\subsection{The regular case of $C^3$ boundary}\label{SubsReg}
For the reader's convenience, this subsection revisits a result initially established in \cite[Theorem7.1]{BARDOS-2016} regarding the asymptotic expansion of the heat content $N(t)$ up to the third-order term for domains $\Omega$ with a regular boundary $\partial \Omega \in C^3$. It presents a sketch of the proof and illustrates the quality of the result with some numerical simulation, see Figure~\ref{Fig2asf}.
This section provides the necessary foundation for the new results presented in Section~\ref{SSpartic}.

\begin{theorem}~\cite[Theorem~7.1]{BARDOS-2016}\label{ThFinalr}
Let $\Omega$ be a bounded domain of $\R^n$ with a connected, $C^3$-regular boundary
$\del \Omega$. Then for $\Lambda=\infty$ we have
\begin{equation}\label{NtregLI}
N(t)=2 \frac{1-e^{-4}}{\sqrt{\pi}} ~ \frac{\sqrt{D_+ D_-}}{\sqrt{D_+}+\sqrt{D_-}} ~ \sH^{n-1}(\del \Omega)~ \sqrt{t} +O(t^{\frac{3}{2}}).
\end{equation}
In the case of $0<\Lambda<\infty$, we have
\begin{align}
 N(t)&=4 C_0t \int_{\del \Omega}\Lambda(\sigma) \mathcal{H}^{n-1}(d \sigma) -
\tfrac{2}{3}C_1t^\frac{3}{2}\biggl[2 \biggl(\tfrac{1}{\sqrt{D_+}}+\tfrac{1}{\sqrt{D_-}}\biggr) \int_{\del \Omega}\Lambda^2(\sigma)\mathcal{H}^{n-1}(d \sigma) \nonumber\\
 & \quad - \sqrt{D_+}(n-1)\int_{\del \Omega}\Lambda(\sigma) H(\sigma)\mathcal{H}^{n-1}(d \sigma)\biggr]+O(t^2),\label{NtregLF}
\end{align}
where $H$ is the mean curvature, and
\begin{align}
\label{eq:C0}
 C_0&=1+\frac{3}{2}\erf(1)-\frac{9}{4}\erf(2)+\frac{1}{\sqrt{\pi}}\left(\frac{1}{e}-\frac{1}{e^4} \right) \approx 0.2218,\\
 C_1&=\frac{1}{\sqrt{\pi}} -6+\frac{5e^{-4}-4e^{-1}}{\sqrt{\pi}}-5\erf(1)+11\erf(2) \approx 0.5207.
\end{align}
Here $\erf(x):=\frac 2{\sqrt{\pi}}\int_0^{x} \exp(-t^2)dt$ denotes the Gauss error function.
\end{theorem}
Fig.~\ref{Fig2asf} shows the efficiency of the approximation~\eqref{NtregLF} for $\Omega\subset\R^2$ being a ball.
\begin{figure}[!ht]
\begin{center}
 \psfrag{t}{{\small $t$}}
\psfrag{N}{{\small $N(t)$}}
\psfrag{data1}{{\small $N(t)$}}
\psfrag{data2}{{\small $1$ term}}
\psfrag{data3}{{\small $2$ terms}}
 \includegraphics[width=\linewidth]{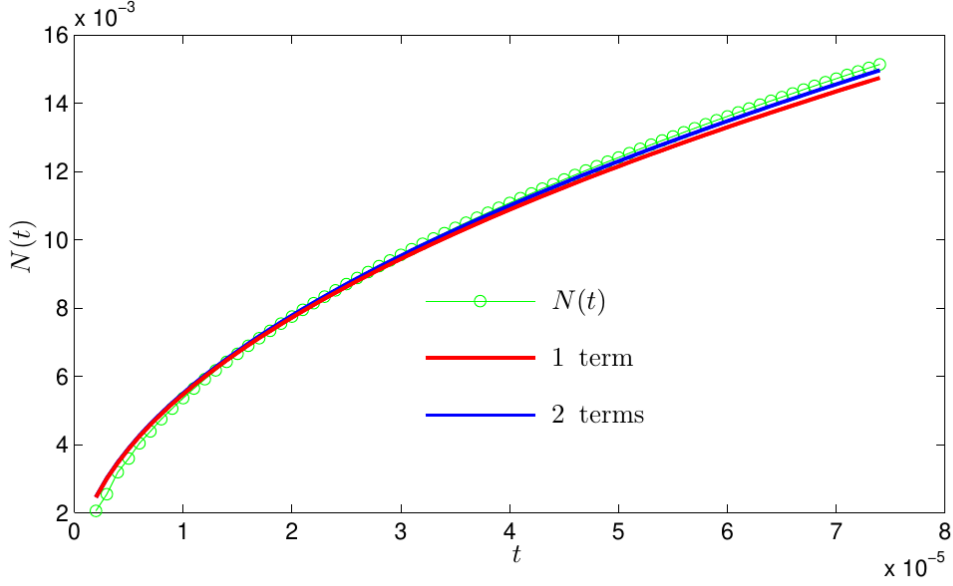} %
 \end{center}
\caption{\label{Fig2asf}
 Comparison between the asymptotic formula~(\ref{NFint2Ca}) with one term of the order $t$ (solid red
 line, $t^{-\delta}\approx 1$) and formula~\eqref{NtregLF} with two terms (blue line) and a FreeFem++ numerical solution of problem~(\ref{prb1})--(\ref{prb1end}) (green line with circles) for the circle boundary with $D_+=1/100$, $D_-=1$, and $\Lambda=17$. }
 \end{figure}
The depicted numerical solution was obtained in FreeFem++
by a finite element method with the implicit $\theta$-schema, also
known as Crank-Nicolson schema, for the time discretization with
$\theta=\frac{1}{2}$ and $\Delta t=10^{-6}$. The domain $\Omega$ is equal to a ball confined in another ball $B$ with the same center and two times bigger radius. The Neumann boundary condition was imposed on
the boundary of the external ball. According to the principle `not feeling
the boundary'~\cite{FLECKINGER-1995}, the
heat content propagation in $\R^2$ with a prescribed boundary $\del
\Omega$ can be very accurately approximated at short times by
the heat content propagation computed only in $B$. The accuracy of this
approximation can also be checked by changing the diameter of the
ball. In the case of
the discontinuous solution on the boundary (when $0<\Lambda<\infty$)
we apply the domain decomposition method and match the boundary values
of the respective solutions on $\del \Omega$ by a Picard fixed point
method. Therefore, we consider the numerical solution of heat
propagation for small times as a reference, to which asymptotic
formulas are compared. In particular, deviations between the
numerical solution and the asymptotic formulas observed at longer
times illustrate the range of validity of the short-time expansion.

The proof of the asymptotic formulas in Theorem~\ref{ThFinalr} is rather technical. As a first step, one needs to calculate explicitly Green's function of the constant coefficient problem in the half
space.
A key tool is formula~\eqref{ErrOm}, which is a variant of the principle `not feeling the
boundary'~\cite{FLECKINGER-1995}. Due to the continuity of $u$ in the parameter $\Lambda$, this
allows, to establish for a constant $\delta_0>0$ and $\eta(t)=O(\sqrt{t}t^{-\delta})$ (in~\cite{BARDOS-2016} $\eta(t)$ was taken as $\eps=O(\sqrt{t})$) that
\begin{equation*}
N(t)=\int_{\Omega} (1-u(x,t))\dx=\int_{\Omega_{\eta(t)}}(1-u(x,t))\dx+O(e^{-\frac{1}{t^{\delta_0}}})
\end{equation*}
can be found as a sum of two heat contents according to the finiteness or not of $\Lambda$ (including the value zero) in the boundary conditions ($i.e.$ for $\del \Omega=\Gamma_\infty\cup \Gamma_\Lambda\cup \Gamma_0$):
\begin{equation*}
N(t)=\int_{\Omega_{\eta(t)}^{\Gamma_\Lambda}}(1-u(x,t))\dx+ \int_{\Omega_{\eta(t)}^{\Gamma_\infty}}(1-u(x,t))\dx+O(e^{-\frac{1}{t^{\delta_0}}}).
\end{equation*}
Considering a regular $\del
\Omega$ (at least in $C^3$) and using the localization properties of the heat propagation, we rewrite the formula
for $N(t)$ in terms of local coordinates.

 In order to explain this in more detail, we recall from~\cite{BARDOS-2016} the approximating framework.
Assume $\del \Omega\in C^\infty$ (or at least $C^3$). 
Then, due to the assumed compactness, the regularity implies that $\del \Omega$ has \emph{positive reach}, see e.g.~\cite[p.56]{RZ19} (in fact, $C^2$ suffices for this). Recall that a closed set $A$ has \emph{reach} $r>0$, if each point $x$ with $\dist(x,A)<r$ has a unique nearest point in $A$. We denote the reach of $\del \Omega$ by $\eta_0$ and assume without further mention that all radii $\eta(t)$ considered below satisfy $0<\eta(t)<\eta_0$. We consider a disjoint
 decomposition $\Omega_{\eta(t)}\cup\del
\Omega\cup\Omega_{-\eta(t)}=\bigcup_{m=1}^M B_{m,\eta(t)}$ into a finite number of sets $B_{m,\eta(t)}$ for which it is possible
to introduce local coordinates. %
In addition, we assume that for
all $m=1,\ldots,M$ there exists $\sigma_m\in \del \Omega \cap
B_{m,\eta(t)}$ such that $B_{m,\eta(t)}$ is included in the Euclidean ball with centre $\sigma_m$ and radius $2\eta(t)$. %
Due to \cite[Proposition~3.1]{BARDOS-2016}, the last assumption ensures that it is possible to consider problem~(\ref{prb1})--(\ref{prb1end}) locally, only on $B_{m,\eta(t)}$, up to an exponentially small error.
Denoting $\Omega_{m,+\eta(t)}:=B_{m,\eta(t)}\cap\Omega_{\eta(t)}$ and
 $\Omega_{m,-\eta(t)}:=B_{m,\eta(t)}\cap\Omega_{-\eta(t)}$, respectively, equation (\ref{ErrOm}) becomes
\begin{align}
 N(t)&=\sum_{m=1}^M \int_{\Omega_{m,+\eta(t)}} (1-u(x,t))\dx+O(e^{-\frac{1}{t^{\delta_0}}}). \label{nloc}
\end{align}

For each $m$, we perform the change of the space variables
$(x_1,\ldots,x_n)\in B_{m,\eta(t)}$ to the local
coordinates $(\theta_1,\ldots,\theta_{n-1},s)$ by the formula
\begin{equation*}
 x=\hat{x}(\theta_1,\ldots,\theta_{n-1})-sn(\theta_1,\ldots,\theta_{n-1}), 
 \left\{\begin{array}{l}
 0<s<\eta(t)\;\hbox{for } x\in B_{m,\eta(t)}\cap\Omega_{\eta(t)},\\
 -\eta(t)<s<0 \; \hbox{for } x\in B_{m,\eta(t)}\cap\Omega_{-\eta(t)},
 \end{array}\right.
\end{equation*}
where $\hat{x}(\theta_1,\ldots,\theta_{n-1})\in \del \Omega$ and $x$,
$\hat{x}$ and $n$ are the vectors in $\R^n$ such that
\begin{equation*}
\left\{\frac{\del\hat{x} }{\del\theta_1 }, \ldots,\frac{\del\hat{x} }{\del\theta_{n-1} },n\right\}
\end{equation*}
is an orthonormal basis in $\R^n$.
 In each of two regions,
$\Omega_{m,+\eta(t)}$ and $\Omega_{m,-\eta(t)}$, the change of variables
$(x_1,\ldots,x_n)\mapsto (\theta_1,\ldots,\theta_{n-1},s)$ is a local
$C^1$-diffeomorphism. We notice that
 $\del \Omega$ is described by $s=0$.
Denoting $\theta=(\theta_1,\ldots,\theta_{n-1})$, the integration
domain $\Omega_{m,+\eta(t)}$ in~(\ref{nloc}) becomes a parallelepiped
\begin{equation*}
\hat{\Omega}_{m,+\eta(t)}=\{0<s<\eta(t), \quad \theta \hbox{ s.t. } \hat{x}(\theta)\in \del \Omega\cap\overline{\Omega}_{m,+\eta(t)}\}.
\end{equation*}

Let us fix $m$.
Thanks to the local change of variables, as explained in~\cite{BARDOS-2016},
using twice the integration by parts and the notations
\begin{equation}\label{Jacobian}
 |J(s,\theta)|=\prod_{j=1}^{n-1}(1-sk_{j})
\end{equation}
for the Jacobian and $k_{j}=k_{j}(\theta_1,\ldots,\theta_{n-1})$ for the
principal curvatures of $\del \Omega$ curving away the outward normal
$n$ to $\del \Omega$ like in the case of the sphere, we find that for
all test functions $\phi=(\phi_+,\phi_-)\in V|_{B_{m,\eta(t)}}$
\begin{align*}
 &\int_{\hat{\Omega}_{m,+\eta(t)}\cup \hat{\Omega}_{m,-\eta(t)}}\del_t u ~|J(s,\theta)|~\phi ~\ds d \theta_1\cdots d \theta_{n-1}
\\
&-\int_{\hat{\Omega}_{m,+\eta(t)}} \hspace*{-1mm} \left(\frac{\del }{\del s}\left(D_+|J(s,\theta)|\frac{\del u_+ }{\del s} \right)
+\sum_{j=1}^{n-1}\frac{\del }{\del \theta_j}\left(\tfrac{D_+|J(s,\theta)|}{(1-sk_j)^2}
\frac{\del u_+}{\del \theta_j}\right)\right)\phi_+ \ds d \theta_1\cdots d \theta_{n-1}\\
&-\int_{\hat{\Omega}_{m,-\eta(t)}} \hspace*{-1mm} \left(\frac{\del }{\del s}\left(D_-|J(s,\theta)|\frac{\del u_- }{\del s} \right)+
\sum_{j=1}^{n-1}\frac{\del }{\del \theta_j}\left(\tfrac{D_-|J(s,\theta)|}{(1-sk_j)^2}
\frac{\del u_-}{\del \theta_j}\right)\right)\phi_- \ds d \theta_1\cdots d \theta_{n-1}\\
&+ \int_{s=0}\Lambda(u_+-u_-)(\phi_+-\phi_-) d\theta=0.
\end{align*}
The regularity of the boundary ensures that the principal curvatures
$k_j(\theta)$ are at least in $C^1(\del \Omega\cap B_{m,\eta(t)}).$
Therefore, problem~(\ref{prb1})--(\ref{prb1end}) locally becomes
\begin{align}
 &\frac{\del}{\del t}u_+-D_+\left(\frac{\del^2}{\del s^2}+\sum_{j=1}^{n-1}\frac{\del^2}{\del \theta^2_j}\right) u_+=
 D_+\sum_{j=1}^{n-1} \frac{sk_j(\theta)}{1-sk_j(\theta)}
\left( 1+\frac{1}{1-sk_j(\theta)}\right)\frac{\del^2u_+}{\del \theta^2}\nonumber\\
 & -D_+\left(\sum_{j=1}^{n-1}k_j(\theta)+s\sum_{j=1}^{n-1} \frac{k_j^2(\theta)}{1-sk_j(\theta)} \right)
\frac{\del u_+}{\del s}\nonumber\\
 &+\frac{D_+}{|J(s,\theta)|}\sum_{j=1}^{n-1}\frac{\del}{\del \theta_j}
\left(\frac{|J(s,\theta)|}{(1-sk_j(\theta))^2} \right) \frac{\del u_+}{\del \theta_j},\;(s,\theta)\in \overline{\hat{\Omega}}_{m,+\eta},\label{eq47}
\end{align}
\begin{align}
 \frac{\del}{\del t}u_--D_-\left(\frac{\del^2}{\del s^2}+\sum_{j=1}^{n-1}\frac{\del^2}{\del \theta^2_j}\right) u_-&=
 D_-\sum_{j=1}^{n-1} \frac{sk_j(\theta)}{1-sk_j(\theta)}\left( 1+\frac{1}{1-sk_j(\theta)}\right)
\frac{\del^2u_-}{\del \theta^2} \nonumber\\
&-D_-\left(\sum_{j=1}^{n-1}k_j(\theta)+s\sum_{j=1}^{n-1} \frac{k_j^2(\theta)}{1-sk_j(\theta)} \right)
\frac{\del u_-}{\del s}\\
&+\frac{D_-}{|J(s,\theta)|}\sum_{j=1}^{n-1}\frac{\del}{\del \theta_j}\left(\frac{
|J(s,\theta)|}{(1-sk_j(\theta))^2} \right) \frac{\del u_-}{\del \theta_j},\nonumber\\
&(s,\theta)\in \overline{\hat{\Omega}}_{m,-\eta(t)},\label{eq48}
\end{align}
\begin{align}
&u_+|_{t=0}=1,\; u_-|_{t=0}=0,\label{ini}\\
 & D_-\frac{\del u_-}{\del s}|_{s=-0}=\Lambda(\theta) (u_--u_+)|_{s=0}, \label{boundar}\\
& D_+ \frac{\del u_+}{\del s}|_{s=+0}=D_-\frac{\del u_-}{\del s}|_{s=-0}.\label{endsys}
\end{align}
We emphasize that problem~(\ref{eq47})--(\ref{endsys}) should be
considered as the trace of Eqs.~(\ref{prb1})--(\ref{prb1end}) on
$B_{m,\eta(t)}$ in the sense of \cite[Proposition~3.1]{BARDOS-2016} with the mollifier $\phi_{\theta_m}\equiv 1$ on $B_{m,\eta(t)}.$

Therefore, we can rewrite~(\ref{nloc}) in new coordinates and use the
parallelepiped property of $\Omega_{m,+\eta(t)}$ in the space of
variables $(s,\theta,t)$:
\begin{eqnarray*}
N(t)&=&\sum_{m=1}^M\int_{\hat{\Omega}_{m,+\eta(t)}}(1-u(s,\theta,t))|J(s,\theta)|dsd\theta+O(e^{-\frac{1}{t^{\delta_0}}})\\
 &=&\sum_{m=1}^M\int_{\hat{x}(\theta)\in\del \Omega\cap\overline{\Omega}_{m,+\eta(t)}}d\theta \int_{]0,\eta(t)[}ds (1-u(s,\theta,t))|J(s,\theta)|+O(e^{-\frac{1}{t^{\delta_0}}}).\\
\end{eqnarray*}
Since this local representation holds for all $m$ (the form of the
problem~(\ref{eq47})--(\ref{endsys}) is the same for all $m$) and
$\sum_{m=1}^M\int_{\hat{x}(\theta)\in\del
\Omega\cap\overline{\Omega}_{m,+\eta}}d\theta=\int_{\del \Omega} \sH^{n-1}( d 
\sigma)=\sH^{n-1}(\del \Omega)$, we can formally write
\begin{align}
N(t)&=\int_{\del \Omega} \sH^{n-1}(d \sigma) \int_{]0,\eta(t)[}ds (1-u(s,\sigma,t))|J(s,\sigma)|+O(e^{-\frac{1}{t^{\delta_0}}}),\label{nnloc}
\end{align}
where $u$ is the solution of~(\ref{eq47})--(\ref{endsys}) in
$]-\eta(t),\eta(t)[\times \del \Omega$ in the local sense, as explained
previously.

First, we give the approximation of the heat content associated to the solution of the system~\eqref{prb1}--\eqref{prb1end} with a constant $\Lambda$ by the solution of the one-dimensional constant-coefficient
problem.
The key point is that, according to
\cite[p.~48-49]{MCKEAN-1967}, due to Varadhan's bound property of
the Green's function, locally, the difference between the
Green's function of the problem in the local coordinates with ``frozen'' coefficients in one boundary point and the analogous Green's function of the constant coefficient problem in the half
space in $\R^n$ is exponentially small, see~\cite[p.80]{BARDOS-2016}. Then we also approximate it by the Green function of the simplified operator $D_\pm (\frac{\del^2}{\del s^2}+\sum_{i=1}^{n-1}\frac{\del^2}{\del \theta_i^2})+D_\pm k\frac{\del}{\del s}$ with constant coefficients. %
Then, we use Duhamel's formula to construct a parametrix~\cite{MCKEAN-1967} by the explicitly known Green's function in the half-space for this simplified operator~\cite{BARDOS-2016}. %
Finally, this allows to prove Theorem~\ref{ThFinalr}.

\subsection{Heat content for the particular case $D_+=D_{-}=\const$}\label{SSpartic}
In this subsection, we consider a simplified case where $D_+=D_{-}$. We refine the results from~\cite{BARDOS-2016}, providing more precise estimates of the Minkowski sausage $\Omega_{\eta}$.
We then examine the principle of `not feeling the boundary' and perform an asymptotic estimation of $N(t)$ for arbitrary domains. 
The asymptotics become more precise for a $d$-Minkowski measurable boundary and subsequently for a regular $C^3$-boundary. 
All these results are novel. 
Additionally, we establish an initial result for an admissible domain, defined as the limit of a sequence of regular domains. 
All the obtained asymptotic expansions of $N(t)$ explicitly reveal the relationship with the volume of the Minkowski sausage of $\partial \Omega$ and clarify thus the de Gennes hypothesis discussed in the introduction.

Let us simplify the system~\eqref{prb1}--\eqref{prb1end} and consider instead 
the following problem with a constant $D>0$:
\begin{align}
  \del_t u-D \Delta u&=0 \quad x\in \R^n, \; t>0,\label{pch}\\
 u|_{t=0}&=\mathds{1}_{\Omega }.\label{ps}
\end{align}
Then, we follow~\cite{BARDOS-2016} by updating it. We have
\begin{lemma}
Let $\Omega \subset\R^n$ be an arbitrary connected bounded open set and let
\begin{equation}\label{EqSolSimple}
 u(x,t)= \int_{\R^n}(4D\pi t)^{-\frac{n}{2}} \exp\left(-\frac{|x-y|^2}{4Dt}\right)\mathds{1}_{\Omega }(y)\dy 	
 \end{equation}
be the solution of system~\eqref{pch}--\eqref{ps}. 
Then for all $p\in ]0,\frac{1}{2}[$ and $\eta(t)=O(t^p)$ (or equivalently, $\eta(t)=O(t^{\frac{1}{2}-\delta})$ for some $\delta\in ]0,\frac{1}{2}[$) there exists a constant $\delta_0>0$ (depending on $p$ and $n$) such that the corresponding heat content
$N(t)$ can be calculated for $t\to0+$ by
\begin{equation}\label{EqExp1base}
	N(t)=\int_{\Omega_{\eta(t)}}\left(1-u(x,t)\mathds{1}_\Omega(x)\right)\dx+O(e^{- 1/t^{\delta_0}})
\end{equation}
and also by
\begin{equation}\label{EqExp2base}
	N(t)=\int_{\R^n}\frac{1}{\pi^\frac{n}{2}}e^{-|v|^2} \left[\int_{\Omega_{\eta(t)}}\left(\mathds{1}_{\Omega_{\eta(t)}}(x)-\mathds{1}_{\Omega -2\sqrt{Dt}v}(x)\right)\dx\right]\dv+O(e^{- 1/t^{\delta_0}}).
\end{equation}
Here $\mathds{1}_{\Omega-2\sqrt{Dt}v }(x)=\mathds{1}_{\Omega
}(x+2\sqrt{Dt}v)$ and the notation $\Omega-2\sqrt{Dt}v$ means that
$\Omega$ is shifted by the vector $-2\sqrt{Dt}v\in \R^n$.
\end{lemma}
\begin{proof}
	In the proof of expansion~\eqref{EqExp1base}, we use the principle ``not feel the boundary'' as in the general case. For the reader's convenience, we give the details of the proof.
	
	The solution of system~\eqref{pch}--\eqref{ps} $u(x,t)$, which is obviously $0\le u(x,t)\le 1$ for all $x\in \R^n$ and $t>0$, can be transformed to the solution
 $\hat{u}=1-u$ of the following system
 \begin{align}
 \left(\del_t -D\Delta\right) \hat{u}&=0,\quad x\in \R^n,\\
 \hat{u}|_{t=0}&=\mathds{1}_{\overline{\Omega^c}}.
\end{align}
 Once again, $\hat{u}$ takes values in $[0,1]$ and decreasing in time on $\del \Omega$. Therefore, considering its values on $\Omega$, we find that
 $\hat{u}\le w$ on $\Omega$ for $t\ge 0$ for $w$, the solution of the problem
 \begin{align}
 \left(\del_t -D\Delta\right) w&=0,\quad x\in \Omega,\\
w|_{t=0}&=0, \\
w|_{\del \Omega}&=1.
\end{align}
Following the approach of~\cite{FALCONER-1997} (p.231 Lemma 12.7), we find that for the domain $\Omega$ equal to a ball $\Omega=B_r(z)$
centered at $z\in \R^n$ and of radius $r>0$, the solution satisfies as $t\to +0$
\begin{equation*}
w(z,t)\le C\left(\frac{r}{\sqrt{4D_+t}}\right)^{n-2} \exp\left(-\frac{r^2}{4D_+t}\right) ,
\end{equation*}
with a constant $C>0$ depending only on $n$ ($C$ can be explicitly
obtained by the integration by parts in the generalized spherical
coordinates in $\R^n$, where the coefficient
$\left(\frac{r}{\sqrt{4D_+t}}\right)^{n-2}$ corresponding to the leading
term as $t\to +0$, appears from the integral
$\int_{\frac{r}{\sqrt{4Dt}}}^{+\infty}e^{-p^2}p^{n-1}dp$).
Thus, see~\cite{BARDOS-2016} Lemma~3.1 which uses~\cite{FALCONER-1997} Corollary 12.8 p.232, for
$z\in \textrm{int} \{ \Omega\}$ and $t\to +0$ we find
\begin{equation}\label{vest}
w(z,t)\le C \left(\frac{\dist(z,\del \Omega)}{\sqrt{4D_+t}}\right)^{n-2}\exp\left(-\frac{\dist(z,\del \Omega)^2}{4D_+t}\right).
\end{equation}
This means that if
$F\subset \Omega$ is a nonempty open bounded set in $\R^n$, such
that $\operatorname{dist}(F,\del \Omega)=r>0$, then we have as $t\to +0$
\begin{equation}\label{errF}
\begin{split}
 \int_F(1-u(x,t)\mathds{1}_\Omega)\dx & =\int_F \left(1- \int_\Omega G(x,y,t)\dy\right) \dx \\
& =O\left(\left(\frac{r}{\sqrt{4D_+t}} \right)^{n-2}e^{- r^2/(4D_+t)}\right). \\
\end{split}
\end{equation}

Let us consider $r$ wich ensures the exponential decay in Eq.~(\ref{errF}), $i.e.$ the existence of a positive constant $\delta_0>0$ such that %
\begin{equation}\label{EqDelta0}
O\left(\left(\frac{r}{\sqrt{4Dt}}
\right)^{n-2}e^{- r^2/(4Dt)}\right)=O(e^{- 1/t^{\delta_0}}).
\end{equation}
We see that if the distance $\operatorname{dist}(F,\del \Omega)=r(t)=O(t^p)$ with $p<\frac{1}{2}$ then there exists $\delta_0(p,n)>0$ depending on $p$ and the dimension $n$ such than
\begin{equation}\label{EqFu}
	\int_F(1-u(x,t)\mathds{1}_\Omega)\dx=O(e^{- 1/t^{\delta_0(p)}}).
\end{equation}
Thus, we consider $\eta(t)=O(t^p)$ and decompose $\Omega=\Omega_{\eta(t)}\cup (\Omega\setminus \overline{\Omega_{\eta(t)}})$ in a way that for a $\delta_0>0$ sufficiently small
\begin{align*}
	\notag N(t)&= \int_{\Omega_{\eta(t)}}(1-u(x,t)\mathds{1}_\Omega)\dx+\int_{\Omega\setminus \overline{\Omega_{\eta(t)}}}(1-u(x,t)\mathds{1}_\Omega)\dx\\
	&=\int_{\Omega_{\eta(t)}}(1-u(x,t)\mathds{1}_\Omega)\dx+O(e^{- 1/t^{\delta_0}}).
\end{align*}
We have just noticed that for all $F\subset \Omega\setminus \overline{\Omega_{\eta(t)}}$ we have $\operatorname{dist}(F,\del \Omega)\ge O(t^p)>O(\sqrt{t})$ which is equivalent to the existence of a $p< \frac{1}{2}$ such that $\operatorname{dist}(F,\del \Omega)=O(t^p)$ and~\eqref{EqFu} holds. Thus, we found expansion~\eqref{EqExp1base}.

To obtain~\eqref{EqExp2base} we proceed in the following way:
By definition,
\begin{align}
 N (t)&=\int_{\R^n\setminus\overline{\Omega}}\int_{\R^n}G(x,y,t)\mathds{1}_{\Omega }(x)\dx\dy,
\end{align}
where this time $G$ is the heat kernel in $\R^n$
\begin{equation*}
G(x,y,t)= (4D\pi t)^{-\frac{n}{2}} \exp\left(-\frac{|x-y|^2}{4Dt}\right).
\end{equation*}
Therefore, we have
 \begin{align*}
 	N(t)&= \int_{\Omega^c}\int_\Omega G(x,y,t)\dy \dx 
    \\
 	&=\int_{\R^n}\int_\Omega G(x,y,t)\dy \dx-\int_\Omega \int_\Omega G(x,y,t)\dy \dx\\
&= \operatorname{Vol}(\Omega)-\int_\Omega \int_\Omega G(x,y,t)\dy \dx\\
&=\mathrm{Vol}(\Omega)-\int_{\R^n}\frac{1}{\pi^\frac{n}{2}}e^{-|v|^2}
 \left( \int_{\R^n}\mathds{1}_{\Omega }(x)\mathds{1}_{\Omega }(x+2\sqrt{Dt}v)\dx\right)\dv\\
&=\int_{\R^n}\frac{1}{\pi^\frac{n}{2}}e^{-|v|^2} \left[\int_{\Omega }\left(\mathds{1}_{\Omega }(x)-\mathds{1}_{\Omega -2\sqrt{Dt}v}(x)\right)\dx\right]\dv.
 \end{align*}
Then we combine the last formula with~\eqref{EqExp1base} and obtain~\eqref{EqExp2base}.
\end{proof}
This localization around the boundary can be precised more if we consider smaller distances from the boundary with the radius depending on the diffusion length, going closer to de Gennes hypothesis:
\begin{proposition}\label{PropArbSC}
	Let $\Omega \subset\R^n$ be a connected bounded arbitrary open set and $u$ is the solution of system~\eqref{pch}--\eqref{ps} defined in \eqref{EqSolSimple}. 
	Let %
	$\delta\in ]0,\frac{1}{2}[$ be chosen. %

Then  %
there exists $\delta_0(\delta,n)>0$ independent of $t$, such that, as $t\to 0+$, 
 \begin{equation}\label{EqResPol1}
 	N(t)\le 	\int_0^{t^{-\delta}}\frac{r^{n-1}}{\sqrt{\pi}}e^{-r^2} \mu(\del \Omega,r \sqrt{4Dt})dr+O\left(e^{-\frac{1}{t^{\delta_0}}}\right),
 \end{equation}
 and
 \begin{multline}\label{EqOcenkaSnizu}
 	N(t)\ge \int_{\R^n}\frac{1}{\pi^\frac{n}{2}}e^{-|v|^2} \left[\int_{\Omega_{\sqrt{4Dt}}}\left(\mathds{1}_{\Omega_{\sqrt{4Dt}} }(x)-
 \mathds{1}_{\Omega_{\sqrt{4Dt}} -2\sqrt{Dt}v}(x)\right)\dx\right]\dv\\=O\left(\mu(\del \Omega, \sqrt{4Dt})\right). 
 \end{multline}
\end{proposition}
\begin{proof}
	We start with formula~\eqref{EqExp2base}. The class $O(t^p)$ with $p<\frac{1}{2}$ ($p>0$) can be also expressed by $p=\frac{1}{2}-\delta$ with some $\delta\in ]0,\frac{1}{2}[$. Hence, for a fixed such $\delta$ we take instead of a constant $\eta$ in \eqref{errF}, a function depending on $t$ of the form
\begin{equation}\label{EqEta}
	\eta(t)=\sqrt{4D}t^{\frac{1}{2}-\delta}
\end{equation}
which leads to the decomposition of $\Omega$ in
$$\Omega=\Omega_{\eta(t)}\cup (\Omega\setminus \overline{\Omega_{\eta(t)}}).$$
For all $F\subset \Omega\setminus \overline{\Omega_{\eta(t)}}$, $\operatorname{dist}(F,\del \Omega)>\sqrt{4D}t^{\frac{1}{2}-\delta}$ and thus there exists (a minimal uniform on $t$ rate depending on the space dimension $n$ by formula~\eqref{EqDelta0}) $\delta_0(\delta,n)>0$ such that
the error decreases exponentially
\begin{equation*}
 \int_F\left(1-u(x,t)\mathds{1}_\Omega(x)\right)\dx\le O(e^{- 1/t^{\delta_0}}).
\end{equation*} 
\begin{equation*}
 	N(t)= \int_{\R^n}\frac{1}{\pi^\frac{n}{2}}e^{-|v|^2} \left[\int_{\Omega_{\eta(t)} }\left(\mathds{1}_{\Omega }(x)-
 \mathds{1}_{\Omega -2\sqrt{Dt}v}(x)\right)\dx\right]\dv+O\left(e^{-\frac{1}{t^{\delta_0}}}\right).
 \end{equation*}
 Let us fix $t$ and $v$. Then, for all $x\in \Omega_{\eta(t)}$ such that $\operatorname{dist}(x,\del \Omega)>\sqrt{4Dt}|v|$, 
it holds $(x+2\sqrt{Dt}v)\in\Omega $.
Thus, it follows that %
\begin{equation}\label{EqChivt}
\chi_{(v,t)}(x)=\mathds{1}_{\Omega_{\eta(t)} }(x)\left(\mathds{1}_{\Omega }(x)-\mathds{1}_{\Omega -2\sqrt{Dt}v}(x)\right)=0 \hbox{ for all } x\in\Omega_{\eta(t)} \setminus\Omega_{\sqrt{4Dt}|v|},
\end{equation}
where, as previously, the notation $\Omega-2\sqrt{Dt}v$ means that
$\Omega$ is shifted by the vector $-2\sqrt{Dt}v\in \R^n$: $\mathds{1}_{\Omega-2\sqrt{Dt}v }(x)=\mathds{1}_{\Omega
}(x+2\sqrt{Dt}v)$.

The set $\Omega_{\eta(t)} \setminus\Omega_{\sqrt{4Dt}|v|}$ is not exact set when $\chi_{(v,t)}(x)=0$ as soon as there are directions of $v$ which move $x\in \Omega_{\sqrt{4Dt}|v|}$ inside of $\Omega$ with $(x+2\sqrt{Dt}v)\in\Omega_{\eta(t)} $. If $\del \Omega$ is not regular, the question of how to separate the inside directions of $v$ from the outside directions is an open problem.

We notice that we need only $0<|v|<t^{-\delta}$ to ensure $x\in \Omega_{\eta(t)}$.

 We also have (see~\eqref{EqChivt})
\begin{equation*}
	0\le \int_{\Omega_{\eta(t)}}\left(\mathds{1}_{\Omega_{\eta(t)} }(x)-
 \mathds{1}_{\Omega_{\eta(t)} -2\sqrt{Dt}v}(x)\right)\dx\le \mu(\del \Omega, |v|\sqrt{4Dt}).
\end{equation*}
 Hence, we estimate $N(t)$ by 
 \begin{equation}\label{EqRes1}
 	N(t)\le \int_{|v|\le t^{-\delta}}\frac{1}{\pi^\frac{n}{2}}e^{-|v|^2} \mu(\del \Omega,|v| \sqrt{4Dt})dv +O\left(e^{-\frac{1}{t^{\delta_0}}}\right).
 \end{equation}
 Passing to polar coordinates in \eqref{EqRes1}, we find \eqref{EqResPol1}.
 
 At the same time we notice that, since $\eta(t)>\sqrt{4Dt}$,  
 \begin{equation*}
 	\int_{\Omega_{\eta(t)}\setminus \overline{\Omega_{\sqrt{4Dt}}}}\left(\mathds{1}_{\Omega_{\eta(t)} }(x)-
 \mathds{1}_{(\Omega_{\eta(t)}\setminus \overline{\Omega_{\sqrt{4Dt}}}) -2\sqrt{Dt}v}(x)\right)\dx\ge 0,
 \end{equation*}
and so \eqref{EqOcenkaSnizu} follows.
\end{proof}

Let us now take a bounded domain $\Omega$ in $\R^n$ with a
 $d$-dim.\ Minkowski measurable boundary $\del \Omega$. 
The following proposition is implied by \eqref{EqResPol1}, \eqref{EqExpmu1} and \eqref{EqMainAss}.

\begin{proposition}\label{PropPCMink}
	Let $\Omega \subset \mathbb{R}^n$ be a connected bounded domain with a $d$-dimensional Minkowski measurable boundary $\partial \Omega$, and let $u$ be the solution of system~\eqref{pch}--\eqref{ps} defined in~\eqref{EqSolSimple}. 

	For a fixed $\delta \in ]0, \frac{1}{2}[$, we have, as $t \to 0^+$,
\begin{equation}\label{EqNtSverhuExp} 	
		N(t) \le \left(\int_0^{t^{-\delta}} \frac{r^{2n-1-d}}{\sqrt{\pi}} e^{-r^2} \, dr\right) \mu(\partial \Omega, \sqrt{4Dt}) + o(t^{\frac{n-d}{2}}),
	\end{equation}
	or, equivalently,
\begin{equation}\label{EqNtSverhuExpP} 	
		N(t) \le \left(\int_0^{t^{-\delta}} \frac{r^{2n-1-d}}{\sqrt{\pi}} e^{-r^2} \, dr\right) \mathcal{M}^d(\partial \Omega, \Omega) (4Dt)^{\frac{n-d}{2}} + o(t^{\frac{n-d}{2}}).
	\end{equation}

	Moreover, integrating over $r$ up to infinity in~\eqref{EqNtSverhuExp} and~\eqref{EqNtSverhuExpP} does not affect the remainder terms.
\end{proposition}

Following~\cite{BARDOS-2016}, we also provide the analogous results for regular boundaries:
 \begin{lemma}\label{LemRegDD}
Let $\Omega \subset\R^n$ be a connected bounded $C^3$-domain (implying that its boundary $\del \Omega$ is
 $(n-1)$-dim.\ Minkowski measurable), and let
$u$ be the solution of system~\eqref{pch}--\eqref{ps} given in \eqref{EqSolSimple}. 
Then for a fixed $\delta\in]0,\frac{1}{2}[$, and for $t\to 0+$, 
 it holds
\begin{equation}
  N (t)=\int_{0}^{+\infty}\frac{e^{-z^2}}{\sqrt{\pi}}\mu(\del\Omega,2\sqrt{D t}z)\dz + O\left(\sqrt{t}\mu(\del\Omega,\eta(t))\right)\label{resV}.
 \end{equation}
 In addition,  the heat content is explicitly given by
\begin{equation}\label{eqNtregSpecC}   
		N(t)= \mathcal{M}^{n-1}(\del \Omega,\Omega)\frac{\sqrt{Dt}}{\sqrt{\pi}}+O(t^{\frac{3}{2}-2\delta}).
	\end{equation}
\end{lemma}
The proof of Lemma~\ref{LemRegDD} follows closely~\cite{BARDOS-2016}- taking into account this time that the local variable $s$ takes its positive values in $[0, \eta(t)]$ with $\eta(t)=\sqrt{t}t^{-\delta}$ instead of $[0,\eps]$ with $\eps=O(\sqrt{t})$, considered in~\cite{BARDOS-2016}. 

As shown in~\cite{CLARET-2024-1}, any open set in $\mathbb{R}^n$ can be approximated by a monotone sequence of regular domains that satisfy Assumption~\ref{Assump}. In the case where there exists a sequence of regular domains (satisfying Assumption~\ref{Assump}) that converges to a non-regular domain and satisfies~\eqref{EqConvMinkS} uniformly in $r$, we can refine Propositions~\ref{PropArbSC} and~\ref{PropPCMink} as follows:

\begin{proposition}\label{PropLimitCase}
	 Let $\Omega$ be a bounded admissible domain of $\R^n$ and $(\Omega_i)_{i\in \N}$ be a sequence of regular domains converging to $\Omega$ and satisfying Proposition~\ref{PropConvProp}, such that \eqref{EqConvMinkS} holds
 uniformly in $r$ on $]0,\eta(t)]$ ($\delta\in]0,\frac{1}{2}[$ is supposed to be fixed). Then for the limit domain $\Omega$ there exists $\beta\ge0$ such that  it holds the asymptotic expansion for $t\to 0+$
 \begin{equation}
	N (t)
=\int_{0}^{+\infty}\frac{e^{-z^2}}{\sqrt{\pi}} \mu(\del\Omega,\sqrt{4Dt}z)dz+
o(t^\beta).\quad \label{EqResVM}
\end{equation}

 If, in addition, the domain $\Omega$ has 
 a $d$-Minkowski measurable boundary $\del \Omega$, 
 then~\eqref{resV} takes the following form
\begin{equation}
	N (t)=\mathcal{M}^{d}(\del \Omega,\Omega)(4Dt)^\frac{n-d}{2}\int_{0}^{+\infty}z^{d}\frac{e^{-z^2}}{\sqrt{\pi}}\dz + o\left(\max(t^\beta,t^\frac{n-d}{2})\right).\label{resVExp2}
\end{equation}
\end{proposition}
\begin{proof}
For all $i\in \N$ we have (see formula~\eqref{resV})
\begin{equation}
	N_i (t)
=\int_{0}^{+\infty}\frac{e^{-z^2}}{\sqrt{\pi}} \mu(\del\Omega_i,\sqrt{4Dt}z)dz+
o_i(\mu(\del\Omega_i,\sqrt{t})).\quad \label{EqNtregDconti}
\end{equation}

 Thus, thanks to the assumptions,  we pass to the limit as $i\to +\infty$ in \eqref{EqNtregDconti}. By Proposition~\ref{PropC3conv}), $N_i (t)$ converges uniformly in $t$ to $N(t)$, the heat content associated with the limit domain $\Omega$. At the same time $\int_{0}^{+\infty}\frac{e^{-z^2}}{\sqrt{\pi}} \mu(\del\Omega_i,\sqrt{4Dt}z)dz$ also converges uniformly in $t$ to 
 $\int_{0}^{+\infty}\frac{e^{-z^2}}{\sqrt{\pi}} \mu(\del\Omega,\sqrt{4Dt}z)dz$. 

  The remainder term $o_i(\mu(\del\Omega_i,\sqrt{t}))$ is a continuous function of $t$ for all $i$ as the sum of continuous on $t$ functions:
  \begin{equation}\label{EqOi}
 	o_i(\mu(\del\Omega_i,\sqrt{t}))=N_i (t)-\int_{0}^{+\infty}\frac{e^{-z^2}}{\sqrt{\pi}} \mu(\del\Omega_i,\sqrt{4Dt}z)dz\to 0 \quad \hbox{for }t\to+\infty.
 \end{equation}
  Hence, for $i\to +\infty$ the limit function is also continuous on time function converging to $0$ for $t\to 0$, $i.e.$ in other words, it holds~\eqref{EqResVM} for a constant $\beta\ge 0$.  
 Thus the asymptotic expansion~\eqref{resV} holds also for the limit, possibly non-Lipschitz, domain $\Omega$. To obtain~\eqref{resVExp2} we use \eqref{EqExpmu1} and~\eqref{EqMainAss}.
\end{proof}
\begin{remark}
    For domains $\Omega\subset\R^n$ for which the inner parallel volume behaves as in \eqref{eq-LP1-general} (as expected for domains with piecewise self-similar boundaries), one obtains from \eqref{eq-LP2} that 
\begin{equation}\label{EqMuVonKoch}
	\mu(\del \Omega,\sqrt{4Dt}z)=G(\sqrt{4Dt}z)(4Dt)^{\frac{n-d}{2}}z^{n-d}+o(t^{\frac{n-d}{2}}) 
\end{equation}
as $t\to 0+$. Putting it in~\eqref{EqResVM} we find
\begin{equation*}
    N (t)
=(4Dt)^{\frac{n-d}{2}}\int_{0}^{+\infty}\frac{z^{n-d} e^{-z^2}}{\sqrt{\pi}} G(\sqrt{4Dt}z)dz+
o(\max(t^\beta,t^\frac{n-d}{2}).
\end{equation*}
\end{remark}
\begin{remark}
	By~\cite{CLARET-2024-1} any fractal boundary domain $\Omega$ can be approximated by a regular sequence of domains $(\Omega_i)_{i\in \N}$ in the sense of Assumption~\ref{Assump}.
	As the shape of the boundary of $\Omega$ in Proposition~\ref{PropLimitCase} is not explicitly known, we are not able to define the parameter $\beta$ of the remainder term in~\eqref{EqResVM}. From numerical experiences we could expect that it is actually $o(\mu(\del \Omega,\sqrt{t}))$. 
\end{remark}

%


\subsection{Heat content for the general case of $\Lambda=+\infty$ or a constant $\Lambda\in ]0,+\infty[$}\label{SecGeneralFrac}

In what follows, we consider the heat content associated with the solution of~\eqref{prb1}--\eqref{prb1end} with a constant $\Lambda$ (finite or not) for some irregular domain $\Omega$. We use~\eqref{EqEta} to define $\eta(t)$.
The method follows this approach: the expansion of heat content in terms of parallel volumes, known for the regular case, is used to approximate the heat content in the irregular case. 
More precisely, we consider a sequence of $C^3$-domains $(\Omega_i)_{i \in \mathbb{N}}$ that approximate the irregular domain $\Omega$, ensuring that the parallel volumes of $\Omega_i$ converge to those of the limit domain $\Omega$. 
We employ this convergence to pass to the limit in the heat content estimates, obtaining the corresponding estimate for the limit domain. 
We then refine the analysis by distinguishing between Minkowski measurable and non-measurable sets. 
Since parallel sets exhibit distinct asymptotic expansions, this leads to different asymptotic expansions of $N(t)$, up to the second term.
All findings refine and extend the results initially established in~\cite{BARDOS-2016}. First, we update the regular case, and then we address admissible domains, which were not considered in~\cite{BARDOS-2016}.

\subsubsection{Regular case}
We first
prove the asymptotic extension for the case of domains with regular boundary $\del \Omega\in
C^\infty$ or at least in $C^3$. We fix a regular shape and obtain a uniform estimate of $N(t)$ expressed on the volume of Minkowski sausage for $t\to 0+$.

We start to precise the results of~\cite{BARDOS-2016} by taking into account that the exponentially small error in the asymptotic expansion of the heat content is ensured outside of $\eta(t)$-neighborhood of $\del \Omega$ with $\eta(t)$ defined in \eqref{EqEta}. The difference to~\cite{BARDOS-2016}, that we don't clame that $t^{-\delta}\approx O(1)$ for sufficiently small $\delta>0$ and thus distingish the classes $O(\sqrt{t})$ and $O(\sqrt{t} t^{-\delta})$. 
Following~\cite{BARDOS-2016}, we fistly prove the following approximation of $N(t)$ by the the homogeneous solution of the one-dimensional problem:

\begin{theorem}\label{TH1}
Let $\Omega\subset \R^n$ be a bounded domain with a $C^3$-regular boundary.
Let 
\begin{equation*}
\hat{u}=\left\{\begin{array}{ll}
 \hat{u}_+,& 0<s<\eta(t),\\
\hat{u}_-,& -\eta(t)<s<0,
 \end{array}
 \right. 
\end{equation*}
be the solution of the one-dimensional problem
\begin{align}
 \frac{\del}{\del t}\hat{u}-D_\pm\frac{\del^2}{\del s^2} \hat{u}&=\mathcal{R}_s(s,\theta_0)\hat{u}, \quad -\eta(t)<s<\eta(t), \;\theta\equiv\theta_0,\label{opereqCH1d}\\
\hat{u}|_{t=0}&=\mathds{1}_{0<s<\eta(t)}(s),\nonumber\\
 D_- \frac{\del \hat{u}_-}{\del s}|_{s=-0}&=\Lambda(\hat{u}_--\hat{u}_+)|_{s=0}, \\
 D_+\frac{\del \hat{u}_+}{\del s}|_{s=+0}&=D_-\frac{\del \hat{u}_-}{\del s}|_{s=-0},\label{opereqCH1dN}
\end{align}
obtained from~(\ref{eq47})--(\ref{endsys}) by setting
$\theta\equiv\theta_0$. Here $\mathcal{R}_s(s,\theta_0)$ is given
by
\begin{equation}
	\mathcal{R}_s(s,\theta_0)= R(s,\theta_0)\frac{\del }{\del s}=-D_+
\left(\sum_{i=1}^{n-1}k_i(\theta_0)+s\sum_{i=1}^{n-1} \frac{k_i^2(\theta_0)}{1-sk_i(\theta_0)} \right)\frac{\del }{\del s}.\label{Restins}
\end{equation}

Then the heat content $N(t)$ of problem~\eqref{prb1}--\eqref{prb1end}, approximated
in~(\ref{nnloc}) by the solution of~(\ref{eq47})--(\ref{endsys}) in
$]-\eta(t),\eta(t)[\times \del \Omega$ in the local sense, satisfies
\begin{align}\label{Nuhat}
 N(t)-\int_{\del \Omega} \sH^{n-1}(d\sigma)&\int_{]0,\eta(t)[}\ds\;(1- \hat{u}(s,\sigma,t))|J(s,\sigma)|\\
 &=\left\{\begin{array}{ll}
 O(t^2\mu(\del \Omega,\eta(t))), & 0<\Lambda<\infty,\\ 
 O(t^\frac{3}{2}\mu(\del \Omega,\eta(t))), & \Lambda=\infty. \end{array}
 \right.
\end{align}
If all principal curvatures of $\del \Omega$ are constant, then
$$N(t)=\int_{\del \Omega}\sH^{n-1}(d
\sigma)\int_{]0,\eta(t)[}\ds\;
(1-\hat{u}(s,\sigma,t))|J(s,\sigma)|+O(e^{-\frac{1}{t^{\delta_0}}}).$$
 
Moreover, if $\hat{u}^{hom}(s,t)$ is the solution of the homogeneous
constant coefficients problem
\begin{align}
\del_t \hat{u}^{hom}-D_\pm\frac{\del^2}{\del s^2}\hat{u}^{hom}&=0, \quad -\eta(t)<s<\eta(t), 
\label{homp01}\\
\hat{u}^{hom}|_{t=0}&=\mathds{1}_{0<s<\eta(t)}(s),\nonumber\\
 D_- \frac{\del \hat{u}_{-}^{hom}}{\del s}|_{s=-0}&=\Lambda (\hat{u}_-^{hom}-\hat{u}_+^{hom})|_{s=0}, \\
 D_+\frac{\del \hat{u}_+^{hom}}{\del s}|_{s=+0}&=D_-\frac{\del \hat{u}_-^{hom}}{\del s}|_{s=-0},\label{homp0e}
\end{align} 
then
\begin{align} 
 \notag N(t)-\int_{\del \Omega} \sH^{n-1}(d\sigma)&\int_{]0,\eta(t)[}\ds\; 
(1-\hat{u}^{hom}(s,t))|J(s,\sigma)|\\ &=\left\{\begin{array}{ll}
 O(t \mu(\del \Omega,\eta(t))), & 0<\Lambda<\infty,\\ 
 O(\sqrt{t}\mu(\del \Omega,\eta(t))), & \Lambda=\infty. 
\end{array}
 \right.\label{NOt2h}
\end{align}
\end{theorem}
\begin{proof}
	The proof follows closely the initial proof of~\cite[Theorem~5.1]{BARDOS-2016}, denoting $\eps$ by $\eta(t)$. This time in our framework (to compare to (5.35) in~\cite{BARDOS-2016}), 
	we have that for all $j\ge 1$ there are some constants $C_j>0$ 
	\begin{equation}\label{NNj}
 |NN^j(t)|\le C_j\left\{\begin{array}{ll}
 t (\eta(t))^{j-1}\mu(\del \Omega,\eta(t)), & \quad 0<\Lambda<\infty,\\ 
 t^{\frac{1}{2}}(\eta(t))^{j-1}\mu(\del \Omega,\eta(t)), & \quad \Lambda=\infty, \end{array}
 \right.
\end{equation}
where $NN^j(t)$ denotes the remainder term number $j$ of the corresponding Green function parametrix expansion under the operator $\mathcal{R}_s$ (see Eq.~(5.34) in~\cite{BARDOS-2016}).
		All proof details are straightforward and omitted. 
\end{proof}

\begin{theorem}
Let $\Omega \subset\R^n$ be a connected bounded $C^3$-domain 
 and $N(t)$ be the heat content associated to the solution
$u$ of system~(\ref{prb1})--(\ref{prb1end}). Let $\delta\in]0,\frac{1}{2}[$ be fixed.
Then for $t\to 0+$, and $\eta(t)=2\sqrt{D_+}t^{\frac{1}{2}-\delta}$, the heat content has the following asymptotic form:
\begin{enumerate}
 \item for $\Lambda<\infty$ on $\del \Omega$:
\begin{eqnarray}
 N(t)&=&\frac{2\sqrt{t}\Lambda}{\sqrt{D_+}} 
\left[\mu(\del \Omega,\eta(t))\int_{t^{-\delta}}^{2t^{-\delta}}\dz f(z,t) \right.\nonumber\\
 &&\left. 
 - \int_{t^{-\delta}}^{2t^{-\delta}} \dz \mu\left(\del \Omega,\sqrt{4D_+t}(z-t^{-\delta})\right) f(z,t)
 \right.\nonumber\\
&&\left.-
 \int_0^{t^{-\delta}} \dz \mu\left(\del \Omega,\sqrt{4D_+t}z\right) f(z,t) \right]+O(t\mu(\del \Omega,\eta(t))),\label{EqNilocFint2C}
\end{eqnarray}
where \begin{equation}\label{EqOo}
O(t\mu(\del \Omega,\eta(t)))=o(t),
 	 \end{equation}
 $\alpha=\frac{1}{\sqrt{D_-}}+\frac{1}{\sqrt{D_+}}$ and
\begin{equation}\label{eqf}
f(z,t)= \exp\left(2\Lambda \alpha \sqrt{t}z+\Lambda^2\alpha^2t\right) 
\operatorname{erfc}(z+\Lambda\alpha\sqrt{t}) 
\end{equation}where $\operatorname{erfc}=1-
\operatorname{erf}$ is the Gauss complementary error function;
\item for $\Lambda=\infty$ on $\del \Omega$, 
\begin{equation}
 N(t)=\tfrac{2\sqrt{D_-}}{\sqrt{D_-}+\sqrt{D_+}}\left(\int_0^{+\infty}\frac{e^{-z^2}}{\sqrt{\pi}}\mu(\del \Omega,\sqrt{4D_+t} z) \dz \right)
+O(\sqrt{t}~\mu(\del \Omega,\eta(t))),\label{EqInfNtreg}
 \end{equation}
 where $O(\sqrt{t}~\mu(\del \Omega,\eta(t)))=o(t^\frac{1}{2})$.
\end{enumerate}
 Formulas~(\ref{EqNilocFint2C}) and~(\ref{EqInfNtreg}) can be
 approximated
 \begin{enumerate}
 \item for $\Lambda<\infty$ on $\del \Omega$ by
\begin{equation}\label{NFint2Ca}
	N(t)=4t\Lambda \mathcal{H}^{n-1}(\del \Omega)
\left(2 t^{-\delta}\int_{t^{-\delta}}^{2t^{-\delta}} \hspace{-3mm}f(z,t) \dz - \int_{0}^{2t^{-\delta}} \hspace{-3mm}zf(z,t) \dz\right)+o(t),
\end{equation}
\item and, for $\Lambda=\infty$ on $\del \Omega$, by
\begin{equation}
 N(t)=\frac{2\sqrt{D_-D_+t}~ }{\sqrt{\pi}(\sqrt{D_-}+\sqrt{D_+})} \left[\mathcal{H}^{n-1}(\del \Omega) (1-e^{-t^{-2\delta}})\right]+o(\sqrt{t}).\label{NFint2InfCa}
\end{equation}
 \end{enumerate}
\end{theorem}
\begin{proof}
Thanks to~\eqref{NOt2h},	$N(t)$ becomes 
\begin{align*}
 N(t)-\mu(\del \Omega,\eta(t)) &+Nh_{\eta(t)}(t)+Nf_{\eta(t)}(t)\\
 &=\left\{\begin{array}{ll}
 O(t~\mu(\del \Omega,\eta(t))), & \quad 0<\Lambda<\infty,\\ 
 O(\sqrt{t}~\mu(\del \Omega,\eta(t))), & \quad \Lambda=\infty. \end{array}
 \right.
\end{align*}
where
\begin{equation}
Nh_{\eta(t)}(t)= \int_{]0,\eta(t)[^2}d s_1 \ds \; h_+(s,s_1,t)\int_{\del \Omega} \sH^{n-1}(d\sigma) |J(s,\sigma)|
\end{equation}
and 
\begin{equation}
Nf_{m,\eta(t)}(t)=- \int_{]0,\eta(t)[^2}d s_1 \ds \; f_+(s,s_1,t)\int_{\del \Omega} \sH^{n-1}(d\sigma) |J(s,\sigma)|.
\end{equation}
 Here $h_+$ and $f_+$ are the following functions involved in the definition of the one-dimensional Green function on the positive half space $s_1>0$ and $s_2>0$ (see Appendix B on~\cite{BARDOS-2016}):
\begin{align}
h_+(s_1,s_2,t)&= \tfrac{1}{\sqrt{4\pi D_+t}}\left(\exp\left(-\tfrac{(s_1-s_2)^2}{4D_+t}\right) 
+a(\Lambda) \exp\left(-\tfrac{(s_1+s_2)^2}{4D_+t}\right) \right),\label{h}\\
f_+(s_1,s_2,t)&=b(\Lambda)\tfrac{\Lambda}{D_+}
\exp\left(\tfrac{\Lambda \alpha}{\sqrt{D_+}}(s_1+s_2)+ 
\Lambda^2 \alpha^2 t\right)
\cdot \operatorname{erfc}\left(\tfrac{s_1+s_2}{2\sqrt{D_+t}}+\Lambda \alpha\sqrt{ t} \right),\label{fp}
\end{align}
where 
\begin{align*}
 a(\Lambda)&=\left\{\begin{array}{ll}
 1,&\Lambda<\infty \\
 \frac{\sqrt{D_+}-\sqrt{D_-}}{\sqrt{D_+}+\sqrt{D_-} },&\Lambda=\infty 
 \end{array}
 \right.
 &\text{ and } &&
 b(\Lambda)&=\left\{\begin{array}{ll}
 1,& \quad \Lambda<\infty\\
 0,& \quad \Lambda=\infty
 \end{array}
 \right..
\end{align*}
Let us calculate it explicitly. We notice that $Nf_{\eta(t)}(t)$ is equal to zero for $\Lambda=\infty$. We start with the part $Nh_{\eta(t)}(t)$. As it was detailed in~\cite[Theorem~6.1]{BARDOS-2016} (just by taking $\eps=\eta(t)$) we result in
\begin{equation*}
 Nh_{\eta(t)}(t)=\mu(\del \Omega,\eta(t))-(1-a(\Lambda))\left[\int_{0}^{t^{-\delta}}\frac{e^{-z^2}}{\sqrt{\pi}} \mu(\del \Omega,\sqrt{4D_+t}z)\dz\right].
\end{equation*}
We notice that the integral over $[0,t^{-\delta}]$ in the previous expression can be replaced by the integral over $[0,+\infty[$ with  exponentially small error.

Thus, for $\Lambda<\infty$, $Nh_{\eta(t)}(t)=\mu(\del \Omega,\eta(t))$, since $a(\Lambda)=1$. 
 For the function $f$
from equation (\ref{eqf}), we find that
\begin{align*}
Nf_{\eta(t)}(t)&=- \frac{2\Lambda\sqrt{t}}{\sqrt{D_+}}\int_{\R^2}\ds \dz 
\mathds{1}_{]0,\eta(t)[}(s)\mathds{1}_{]0,\eta(t)[}(-s+2\sqrt{D_+t}z)
 f(z,t)\\ 
 &\qquad\cdot\int_{\del \Omega}\sH^{n-1}(d\sigma) |J(s,\sigma)|\\
&=-\frac{2\Lambda\sqrt{t}}{\sqrt{D_+}}\left[\int_{t^{-\delta}}^{2t^{-\delta}} \dz f(z,t) 
\int_{\del \Omega}\int_{(z-t^{-\delta})\sqrt{4D_+t}}^{\sqrt{4D_+t}t^{-\delta}}|J(s,\sigma)|\ds \sH^{n-1}(d\sigma) \right.\\
&\hspace{1.8cm}\left.+ \int_0^{t^{-\delta}} \dz f(z,t) \int_{\del \Omega}\int_{0}^{\sqrt{4D_+t}z}|J(s,\sigma)|\ds \sH^{n-1}(d\sigma) \right] 
\end{align*}
and  hence 
\begin{align*}
&Nf_{\eta(t)}(t)=-\frac{2\Lambda\sqrt{t}}{\sqrt{D_+}}\left[ \mu(\del \Omega,\eta(t))\int_{t^{-\delta}}^{2t^{-\delta}} 
f(z,t)\dz \right. \\
& \left. -\int_{t^{-\delta}}^{2t^{-\delta}} f(z,t) \mu(\del \Omega,\sqrt{4D_+t}(z-t^{-\delta}))\dz 
+ \int_{0}^{t^{-\delta}} f(z,t) \mu(\del \Omega,\sqrt{4D_+t}z)\dz \right] .
\end{align*}
Here, we have used~\cite[Eq.~(5.38)]{BARDOS-2016} with $v=2\sqrt{D_+t}z$ and 
\begin{equation}
	\mathds{1}_{]0,\eta(t)[}(s)\mathds{1}_{]0,\eta(t)[}(-s+v)\ne 0 \Leftrightarrow  \left\{\begin{array}{ll}
 0<v<\eta(t),&\;0<s<v,\\
 \eta(t)<v<2\eta(t),&v-\eta(t)<s<\eta(t).
 \end{array} \right.\label{zeta}
 \end{equation}

Putting two results together, 
 we obtain formulas~\eqref{EqNilocFint2C} and~\eqref{EqInfNtreg}.
We also notice that it holds~\eqref{EqOo} by the definitions of $\eta(t)$ and $\delta$.
Finally, we use approximation of the volume of the parallel set~\eqref{EqAppJFliss} to obtain from~\eqref{EqNilocFint2C} and~\eqref{EqInfNtreg} formulas~\eqref{NFint2Ca} and~\eqref{NFint2InfCa} respectively.
\end{proof}

\subsubsection{On admissible domains}\label{sssArbC}
 
\begin{proposition}\label{PropAbrCG}
	Let $\delta\in]0,\frac{1}{2}[$ be fixed. Let $\Omega$ be a bounded admissible domain of $\R^n$ and $(\Omega_i)_{i\in \N}$ be a sequence of regular domains converging to $\Omega$ and satisfying Proposition~\ref{PropConvProp}, such that \eqref{EqConvMinkS} holds %
 uniformly in $r$ on $]0,\eta(t)]$. Then for the limit domain $\Omega$ there exist $\beta_1\ge0$ and $\beta_2\ge 0$ such that  for $t\to 0+$ it hold the asymptotic expansions~\eqref{EqNilocFint2C} and~\eqref{EqInfNtreg}, for $\Lambda$ finite 
 and infinite, with the remainder terms replaced by $o(t^{\beta_1})$ and $o(t^{\beta_2})$ respectively.
 \end{proposition}
 The proof of Proposition~\ref{PropAbrCG} follows the same method as the proof of Proposition~\ref{PropLimitCase} based on the general form of formulas~\eqref{EqNilocFint2C} and~\eqref{EqInfNtreg}, and hence the details are omitted. All parameters $\Lambda$, $D_+$, $D_-$, and then $\alpha$ and $f$ don't depend on $\Omega_i$. Another main gradient, by Proposition~\ref{PropC3conv}, is the continuous dependence of $N(t, \del \Omega)$ on $\del \Omega$ (once $\Omega_i\to \Omega$ in the sense of Proposition~\ref{PropConvProp}, this implies the $L^2$-Mosco convergence of the quadratic forms, the convergence of the weak formulations), which %
 implies the convergence of $N_i(t) \to N(t)$ for $i\to +\infty$ uniformly in time. 
Without any exact precision on the boundary shape we are just able to establish that the remainder terms are converging to $0$ for $t\to0+$ continuous functions on time. We could expect that they have the form $o(\sqrt{t}\mu(\del \Omega,\eta(t)))$ for $\Lambda<\infty$ and $o(\mu(\del \Omega,\eta(t)))$ for $\Lambda=+\infty$ respectively.

 If, in addition, the domain $\Omega$ has 
 a $d$-Minkowski measurable boundary $\del \Omega$, 
 then we can use the approximation \eqref{EqMainAss} for the volume of the parallel set.
 
 Let us consider the approximation question for non-Minkowski measurable $d$-sets, presented in Subsection~\ref{SubNonMM}. Using as previously a ``nice'' approximation sequence of regular domains, we have also~\eqref{EqNilocFint2C} and~\eqref{EqInfNtreg}. 
Therefore, it is sufficient to use~\eqref{eq-LP2} in formulas~\eqref{EqNilocFint2C} and~\eqref{EqInfNtreg} to obtain the corresponding heat content asymptotic behavior when $\Omega\subset \R^2$ is a bounded simply connected admissible domain with a compact ``lattice case'' self-similar connected boundary of dimension $d\in ]1,2[$.

We illustrate the pre-fractal boundary case by Fig.~\ref{FigF3}.
\begin{figure}[!ht]
	 \begin{center}
	 \psfrag{t}{{\small $t$}}
\psfrag{N}{{\small $N(t)$}}
\psfrag{data1}{{\small numerical}}
\psfrag{data2}{{\small fractal}}
\psfrag{data3}{{\small asymptotic}}
 \includegraphics[width=\linewidth]{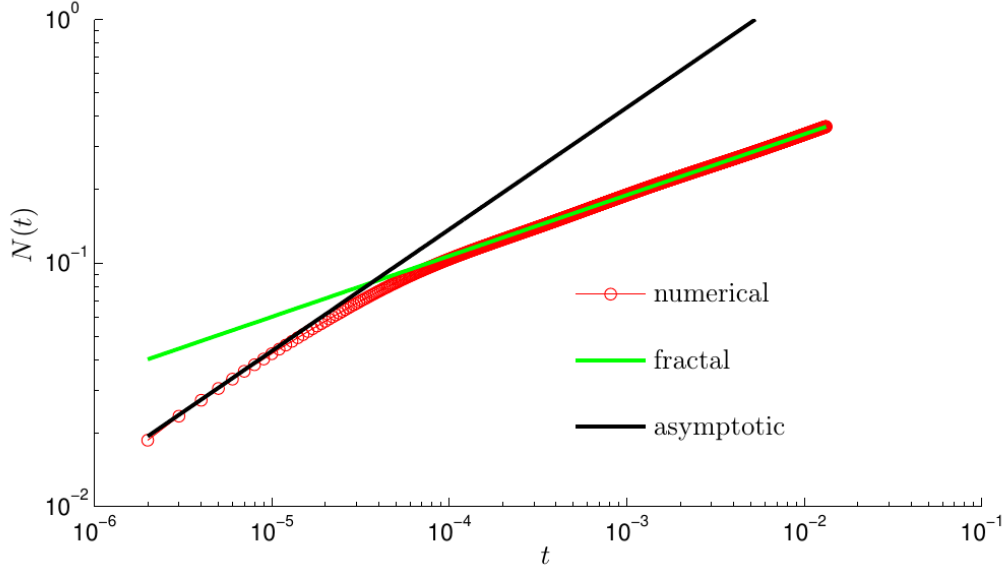}
 \end{center}
 \caption{\label{FigF3}Comparison between the asymptotic formula~\eqref{EqInfNtreg}(with $t^{-\delta}\approx2$) and a FreeFem++ numerical solution of
problem~(\ref{prb1})--(\ref{prb1end}) (red circles) for the
third generation of the Minkowski fractal ($\operatorname{Vol}(\del
\Omega)=2^3\cdot4$), with $D_-/D_+=0.4$, and $\Lambda=+\infty$. The black solid
line is the approximation by~\eqref{EqInfNtreg} considering the boundary as a Lipschitz prefractal one-dimensional curve. The green line shows the fractal
asymptotic (that would be exact for the infinite generation of the
fractal) with de Gennes approximation of $\mu\left(\del
\Omega,\sqrt{4D_+t}\right)$ in~\eqref{EqInfNtreg} by
$(4D_+t)^\frac{1}{4}$. This approximation is valid for intermediate
times.}
\end{figure}

As expected, the resistivity of the boundary to heat transfer makes
heat diffusion {\it slower} due to the presence of the coefficient
$\sqrt{t}$ (see formula~\eqref{EqNilocFint2C}).

\section{Curvature measures} \label{sec:curv-meas}
We recall some geometric background regarding curvature measures and review the question whether parallel sets of arbitrary compact sets in $\R^n$ admit curvature measures.

First order quantities for differentiable submanifolds of $\R^n$, in particular their surface area measures and related dimensions, are determined by the first order derivatives in local parametrization. This has been extended in geometric measure theory to Hausdorff rectifiable sets. Hausdorff measures $\H ^s$, Minkowski contents $\sM^s$ and associated non-integer dimensions $s$ may be considered as certain fractal counterparts. There exists a large literature on these topics.

An important subject of second order differential geometry are curvature properties of manifolds. We will concentrate here on the so-called \emph{Lipschitz-Killing curvature measures} which can be defined for various classes of singular subsets of $\rn$ using tools from geometric measure theory and algebraic geometry. Starting with the classical setting of an $n$-dimensional submanifold $M_n$ of $\R^n$ with $C^2$-smooth boundary $\partial M_n$, curvature measures of $M_n$ may be defined 
in terms of integrals over the boundary of symmetric functions of the principal curvatures $\kappa_i$:
$$
C_k(M_n,\cdot):=\frac{1}{(n-k)\omega_{n-k}}\int_{\partial M_n} \hspace{-3mm}\1_{(\cdot)}(x)\hspace{-3mm}\sum_{1\leq i_1<\cdots <i_{n-1-k}\leq n-1}\hspace{-5mm}\kappa_{i_1}(x)\cdots\kappa_{i_{n-1-k}}(x)\, \H^{n-1}(dx),
$$
$k=0,1,\ldots,n-1$, where $\omega_{j}$ denotes the volume of the $j$-dimensional unit ball and $\H^{n-1}$ the $(n-1)$-dimensional Hausdorff measure. (Note that $2\,C_{n-1}(M_n,\cdot)$ equals the surface area measure $\H^{n-1}$ on $\partial M_n$, which does not reflect curvature properties.)
In the compact case, the \emph{total curvatures}
$$
C_k(M_n):=C_k(M_n,\rn),
$$
i.e., the total masses of these (signed) measures, are known as (higher order) \emph{integrals of mean curvature}.

The $\kappa_i$ are the eigenvalues of the Weingarten mapping (defined at points of $\partial M_n$ as minus the differential of the outer unit normal). For even $n-1-k$ the curvature measures are intrinsically determined in terms of the traces of powers of the Riemannian curvature tensor. For odd $n-1-k$ the curvature measures are extrinsic. In particular, the \emph{integral of scalar curvature} arises in the special case $k=n-3$, and for $k=n-2$ the integrand reduces to the usual \emph{mean curvature} $H=\sum_{i=1}^{n-1}\kappa_i$ and thus $C_{n-2}(M_n)$ is the (first order) \emph{integral of mean curvature}, 
cf.~\cite[Chapter 3]{RZ19}.

In convex geometry, the analogues of the total curvatures $C_k(K)$ for compact, convex sets $K\subset\R^n$ are well-known as \emph{quermassintegrals} or \emph{intrinsic volumes}, which also possess localizations as curvature measures, cf.~\cite[Chapter 2]{RZ19}.

An important extension of both the convex geometric and differential geometric setting are sets with \emph{positive reach}. Recall that $\rea(X)$, the \emph{reach} of a set $X\subset\R^n$, is the largest radius $r$ such that each point with distance less than $r$ to $X$ has a unique nearest point in $X$. Studying the local volume of the parallel sets $X(\ep)$ (see \eqref{eq:parset} for the definition) of such sets $X$ for small enough $\ep>0$, Federer \cite{Fe59} introduced the curvature measures of $X$ as the coefficients in the polynomial expansion in $\ep$ of this local parallel volume. This generalized the Steiner formula from convex geometry as well as Weyl's tube formula from differential geometry. An explicit representation by integration of symmetric functions of generalized principal curvatures $\kappa_1,\ldots,\kappa_{n-1}$ or, equivalently, of Lipschitz-Killing differential forms, 
was given in \cite{Za86b}:
\begin{align}\label{curvposreach}
C_k(X,\cdot):=\frac{1}{(n-k)\omega_{n-k}}&\int_{\nor X}\1_{(\cdot)\times \bbS^{n-1}}(x,\nu) \prod_{j=1}^{n-1}(1+\kappa_{m,\eta(t)}(x,\nu)^2)^{-1/2}\\
\nonumber
&\sum_{i_1<\cdots <i_{n-1-k}}\kappa_{i_1}(x,\nu)\cdots\kappa_{i_{n-1-k}}(x,\nu)\H^{n-1}(d(x,\nu))\,
\end{align}
for $k=0,1,\ldots,n-1$, where $\bbS^{n-1}$ denotes the unit sphere in $\rn$. Note that the
integration here is not over the boundary $\bd X$ but over the \emph{unit normal bundle} of $X$, defined by
\begin{align} \label{eq:norX}
\nor X:=\{(x,\nu)\in \R^n\times \bbS^{n-1}: x\in \bd X ,\, \nu\in \operatorname{Nor}(X,x)\}
\end{align}
where $\operatorname{Nor}(X,x)$ is the set of unit outer normals of $X$ at $x$. 
The resulting curvature measures are signed measures on the Borel $\sigma$-algebra of $\rn$ with locally finite variation measures.
Further extensions of curvature measures may be found in \cite{RZ03}, \cite{RZ05} and the references therein. For details we refer to the monograph \cite{RZ19}. Recently, curvature measures have been extended even further to the class of WDC sets \cite{PR2013,Po2023}

The significance of curvature measures and in particular total curvatures is underlined by their unique position among geometric invariants. It is well-known from convex geometry that every set-additive, continuous and motion invariant functional on the space of convex bodies is a linear combination of quermassintegrals (Hadwiger's characterization theorem). Moreover, if the functional is also homogeneous of degree $k$, then it coincides up to a constant with $C_k$. A localization to curvature measures of convex sets was given in \cite{Schn78}. In \cite[Chapter 8]{RZ19} this is used together with an appropriate notion of continuity as an approximation tool for generalizing such characterizations to a large class of singular sets.

Despite all the progress of the curvature theory for singular sets, for most fractal sets, such curvature measures are not defined and they cannot be expected to exist. So different approaches are needed to study the curvature properties of fractals.
The main idea for the definition of \emph{fractal curvatures} is borrowed from the notion of Minkowski content, namely the idea of approximation of the set in question by its parallel sets. Moreover, from the study of Minkowski contents of self-similar (random) fractals in \cite{Ga00}, the classical Renewal theorem is known to be a powerful tool in this context. Fractal curvatures (to be defined below) were first introduced and investigated in \cite{Wi08} for some class of self-similar sets. More general self-similar sets (including some random self-similar sets) are discussed in \cite{Za11}, \cite{Wi11}, \cite{WZ13}, \cite{Za23}, \cite{RWZ23}.
A local approach by means of ergodic theorems for related dynamical systems can be found in \cite{RZ12}, \cite{Za13} and \cite[Chapter 10]{RZ19}. Some self-conformal sets are studied in the theses \cite{Ko11} and \cite{Bo12}, as well as in \cite{FK12}, \cite{KK12,Kombrink25,ARXIV-KOMBRINK-2023}.

In order to recall the definition of fractal curvatures, denote for points $x$ and non-empty subsets $E$ of $\rn$
\begin{\eq}\label{notations}
d(x,E):=\inf_{y\in E}|x-y|,\quad |E|:=\diam E=\sup_{x,y\in E}|x-y|,~~ E^c:=\rn\setminus E,~~\widetilde{E}:=\overline{E^c},
\end{\eq}
where $\overline{E}$ is the closure of $E$. 
(This notation for closure and complement will also be used for other basic spaces in the sequel.) The \emph{parallel set} of $E$ of distance $r>0$ is defined by
\begin{align}
 \label{eq:parset} E(r):=\left\{x\in\rn: d(x,E)\le r\right\}.
\end{align}
Assuming that curvature measures are defined for all the parallel sets of a bounded set $F\subset \R^n$, the ($s$-dimensional) \emph{fractal curvatures} of $F$ are defined by
\begin{align} \label{eq:frac-curv-def0}
 \sC_k^{s}(F)&:=\lim\limits_{\ep\rightarrow 0} \ep^{s-k}\, C_k\left(F (\ep)\right), \quad k=0,\ldots, n-1,
\end{align}
whenever these limits exist. (Replacing $C_k(F(\ep))$ by the volume $C_n(F(\ep)):=\lambda_n(F(\ep))$ makes the analogy to the $s$-dim.\ Minkowski content obvious.) It turns out that for classical sets (e.g. sets with positive reach), $s=k$ is the right scaling exponent, while for self-similar fractals typically $s=\Dim$ is the right choice for all $k$, where $\Dim$ denotes the Minkowski dimension. In the sequel we suppress the dependence on the parameter $s\geq 0$ and concentrate solely on the limits with $s=\Dim$ which we denote by $C_k^{\Frac}(F)$ (although this is not always the right choice, see \cite{Wi08,PoW14} for a discussion of this issue). A further difficulty is that curvature measures might not be defined for all parallel sets and therefore it is advisable to replace the limit in definition \eqref{eq:frac-curv-def0} by an essential limit, for which it is enough to have curvature measures defined for Lebesgue almost all $\ep>0$:
\begin{align} \label{eq:frac-curv-def01}
 C_k^{\Frac}(F)&:=\elim\limits_{\ep\rightarrow 0} \ep^{d-k}\, C_k\left(F (\ep)\right), \quad k=0,\ldots, n-1,
\end{align}
 Even this precaution does not ensure the (essential) limits in \eqref{eq:frac-curv-def01} to exist, in general. Instead of taking (essential) lower and upper limits as usual in such a case, it is sometimes better to study limits along certain discrete sequences $(\ep_n)_{n\in\N}$. In the study of self-similar sets it is well-known that Ces\`aro averaging does often improve the convergence behavior which does also open an alternative here, leading to \emph{average fractal curvatures}:
\begin{align} \label{eq:av-frac-curv}
 \widetilde{C}_k^{\Frac}(F)&:=\lim\limits_{\delta\rightarrow 0} \frac{1}{|\ln\delta|}\int_\delta^1 \ep^{\Dim-k}\, C_k\left(F (\ep)\right)\frac{d\ep}{\ep}, \quad k=0,\ldots, n-1.
\end{align}

In the present paper we consider domains with piecewise self-similar boundary. In this context related fractal curvatures have not yet been considered in the literature, but our approach here is close to some of the former. In particular, it is based on the curvature measures of parallel sets with small radii. Therefore we recall now a number of related notions and results.
\begin{definition}\label{regular pairs} A pair $(r, E)$ is called {\it regular} if and only if
\begin{align}\label{eq:regularity}
\rea\left( \widetilde{E(r)}\right)>0
~\text{ and }~\nor \widetilde{E(r)} \cap\rho(\nor\widetilde{E(r)})=\emptyset.
\end{align}
Here $\rho(x,\nu):=(x,-\nu)$ denotes the normal reflection on the corresponding normal bundle (which was defined in \eqref{eq:norX}).
\end{definition}
A necessary and sufficient condition for \eqref{eq:regularity} is that $r$ is a regular value of the distance function of $E$ , i.e., there is no $x$ with $d(x,E)=r$, which is contained in the
convex hull of its nearest points in $\overline{E}$, see \cite[Appendix]{RWZ23}. For any regular pair $(r,E)$ the curvature measures $C_k(\widetilde{E(r)},\cdot)$ are defined (since $\widetilde{E(r)}$ has positve reach) and therefore one can introduce the {\it Lipschitz-Killing curvature measures} $C_k(E(r),\cdot)$, $k=0,\ldots, n-1$, of $E(r)$ by setting
\begin{\eq} \label{sign_ep}
C_k(E(r),\cdot)=(-1)^{n-1-k}C_k\big(\widetilde{E(r)},\cdot\big).
\end{\eq}
We point out that this definition is consistent with other definitions, e.g.\ the direct one by means of integration as in \eqref{curvposreach} whenever the latter is possible, and in this case the above equation \eqref{sign_ep} is known as \emph{reflection principle},
see \cite[Example 9.10]{RZ19}.

\begin{remark}\label{large distances}
 In view of \cite[Theorem 4.1]{Fu85} we have for any compact $K$ and any $r>\sqrt{2}|K|$, that $(r,K)$ is a regular pair. Moreover, if $k=0,\ldots,d-1$ and $R>\sqrt{2}$ is fixed, then, by \cite[Theorem 4.1]{Za11}, there is a constant $c_k(R)$ such that, for any compact $K\subset\rn$,
$$\sup_{r\ge R|K|}\frac{C_k^{\var}(K(r),\rn)}{r^k}\leq c_k(R).$$
\end{remark}
In many cases the parallel sets of compact $K$ possess the following additional {\it regularity property}:
 $(r,K)$ are regular pairs for Lebesgue almost all $r>0$.
It is known that in $\mathbb{R}^2$ and $\mathbb{R}^3$ this is fulfilled for all compact $K$, see Fu~\cite[Theorem 4.1]{Fu85}. For $\R^n$ with $n\geq 4$ this will be an assumption in the sequel.

In particular, the $C_k(E(\ep),\cdot)$ are signed measures with finite {\it variation measures} $C_k^{\var}(E(r),\cdot)$, and have explicit integral representations as in \eqref{curvposreach}.\\
Moreover, in this case, for all $r'$ sufficiently close to $r$, $(r',E)$ is also a regular pair, 
and we have the {\it weak convergence}
\begin{\eq}\label{reg-convergence}
\lim_{r'\rightarrow r}C_k(E(r'),\cdot)=C_k(E(r),\cdot),
\end{\eq}
see \cite[Theorem~6.1]{RWZ23} or \cite[Proposition~4.2]{Wi15}. For $k=d-1$, this is even true for all $r>0$, regular or not, up to at most countably many exceptions, see \cite[Theorem 1.1]{RW23}. 

Below we will also need a special regularity relationship. In the sequel we denote
\begin{\eq}\label{regular distances}
\reg(E):=\{r>0: (r,E)\mbox{ is a regular pair}\}.
\end{\eq}
Note that the regularity property above means that $\L\left(\reg(E)^c\right)=0$.
If $f_1,f_2,\ldots$ is an arbitrary family of similarities in $\rn$ with similarity ratios $\rho_1,\rho_2,\ldots$, then
\begin{\eq}\label{intersection reg}
\L\left(\reg(E)^c\right)=0\text{ implies } \L\left(\bigcup_{i=1}^\infty\reg(f_i(E))^c\right)=0.
\end{\eq}
(This results from the subadditivity of $\L$ and $\L(\reg(f_i(E))^c)=\rho_i\L(\reg(E)^c)$.)

Furthermore, the following main properties of the curvature measures for regular pairs will be used:
$2\, C_{n-1}(E(r),\cdot)$ agrees with the $(n-1)$-{\it dimensional Hausdorff measure} $\H^{n-1}$ on the boundary $\partial E(r)$.
Taking into account that $\partial E(r)$ is $(n-1)$-rectifiable for any bounded set $E$ and any $r>0$, see \cite[Proposition~2.3]{RW10}), we can consistently extend the definition of $C_{n-1}(E(r),\cdot)$ to all $r>0$ by setting
$$C_{n-1}(E(r),\cdot):=\frac{1}{2}\H^{n-1}(E(r)\cap\cdot).$$
For regular pairs $(r,E)$ with bounded $E$ the {\it total curvatures}, that is, the total values of the curvature measures, are denoted by
\begin{\eq}
C_k(E(r)):=C_k(E(r),\rn),~~ k=0,\ldots ,n-1.
\end{\eq}
Then $C_{n-2}(E(r))$ is known as the {\it integral of mean curvature}.
Furthermore, by an associated Gauss-Bonnet theorem (see \cite[Theorem 4.53]{RZ19}), the {\it total Gauss curvature}
 coincides with the {\it Euler-Poincar\'{e} characteristic}, i.e.,
$C_0(E(r))=\chi(E(r)).$

The curvature measures are {\it motion invariant}, i.e., for any Euclidean motion $g$ and any regular pair $(r,E)$ the pair $(r,g(E))$ is regular and
\begin{\eq}\label{motinv}
C_k(g(E(r)),g(\cdot))=C_k(E(r),\cdot).
\end{\eq}
The $k$-th curvature measure is {\it homogeneous of degree} $k$, i.e.,
\begin{\eq}\label{scaling}
C_k(\lambda E(r),\lambda (\cdot))=\lambda^k\, C_k(E(r),\cdot)\, ,~~\lambda>0\, ,
\end{\eq}
and they are {\it locally determined}, i.e., if $(E,r)$ and $(E',r')$ are regular pairs in the sense of Definition~\ref{regular pairs}, then for any open set $G$,
\begin{\eq}\label{locality}
C_k(E(r),G\cap(\cdot))=C_k(E'(r'),G\cap(\cdot)),~{\rm if}~E(r)\cap G=E'(r')\cap G.
\end{\eq}

\section{Domains with piecewise self-similar fractal boundaries}\label{fractal bound}
In the sequel we consider throughout a bounded and simply connected open set $\dom $ in $\rn$ satisfying $\dom =\Int\big(\overline{\dom }\big)$ such that the boundary is representable as
\begin{\eq}\label{disjoint(j)}
\partial \dom =\bigcup_{j=1}^M F^{(j)}\,,
\end{\eq}
where the $F^{(j)}$ are self-similar sets in the sense of \cite{Hu81} satisfying the Open Set Condition, which all have the same Hausdorff dimension $\Dim <n$.

More precisely, for each $j=1,\ldots,M$, the set $F^{(j)}$ is defined as
follows. There are similarity mappings $S^{(j)}_1,\ldots, S^{(j)}_{N^{(j)}}$ with contraction ratios $r_1^{(j)},\ldots,r^{(j)}_{N^{(j)}}$ in $(0,1)$ satisfying
\begin{\eq}\label{Hdim}
\sum_{i=1}^{N^{(j)}} (r_i^{(j)})^{\Dim}=1\,,
\end{\eq}
such that $F^{(j)}$ is the associated self-similar set, i.e., the unique non-empty compact set such that
\begin{\eq}\label{self-similar}
F^{(j)}=\bigcup_{i=1}^{N^{(j)}} S^{(j)}_i(F^{(j)}).
\end{\eq}
We assume throughout that $\Fj$ satisfies the {\it Open Set Condition} (OSC), meaning there is some bounded open set $O^{(j)}$ such that
\begin{\eq}\label{OSC}
\bigcup_{i=1}^{N^{(j)}} S^{(j)}_i(O^{(j)})\subset O^{(j)} \text{ and } S^{(j)}_i(O^{(j)})\cap S^{(j)}_l(O^{(j)})=\emptyset \text{ for } i\neq l.
\end{\eq}
Any such set $\Oj$ is also called a \emph{feasible} open set for $\Fj$. The set $F^{(j)}$ is said to satisfy the {\it Strong Open Set Condition} (SOSC), if there is some feasible open set $O^{(j)}$ satisfying additionally
\begin{\eq}\label{SOSC}
F^{(j)}\cap O^{(j)}\neq\emptyset.
\end{\eq}
According to a result of Schief \cite{Sch94}, OSC and SOSC are equivalent, although the validity of \eqref{SOSC} depends on the choice of the open set $O^{(j)}$.
Under this condition the value $\Dim >0$ given by \eqref{Hdim} determines the Hausdorff and Minkowski dimension of $F^{(j)}$.

Figure~\ref{fig:1} shows some domains with piecewise self-similar boundary as discussed here, including the standard Koch snowflake domain. Further examples of such domains will be discussed in Example~\ref{ex:Koch-type}, see also Figure~\ref{fig:ex}, and at the end of Section~\ref{sec:rel-fc}, see in particular Figure~\ref{fig:2} and Example~\ref{ex:Koch-type2}.
\begin{figure}[t]
 \includegraphics[width=0.42\textwidth]{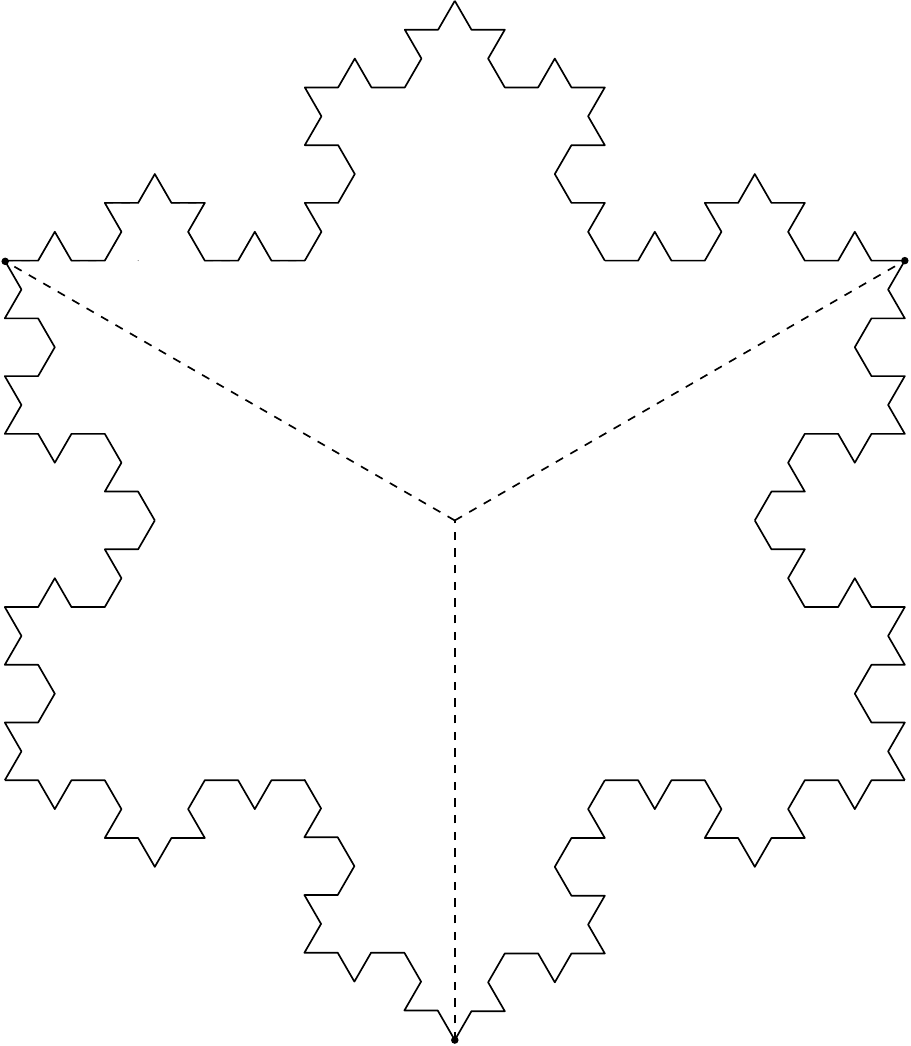}\hfill
 \includegraphics[width=0.58\textwidth]{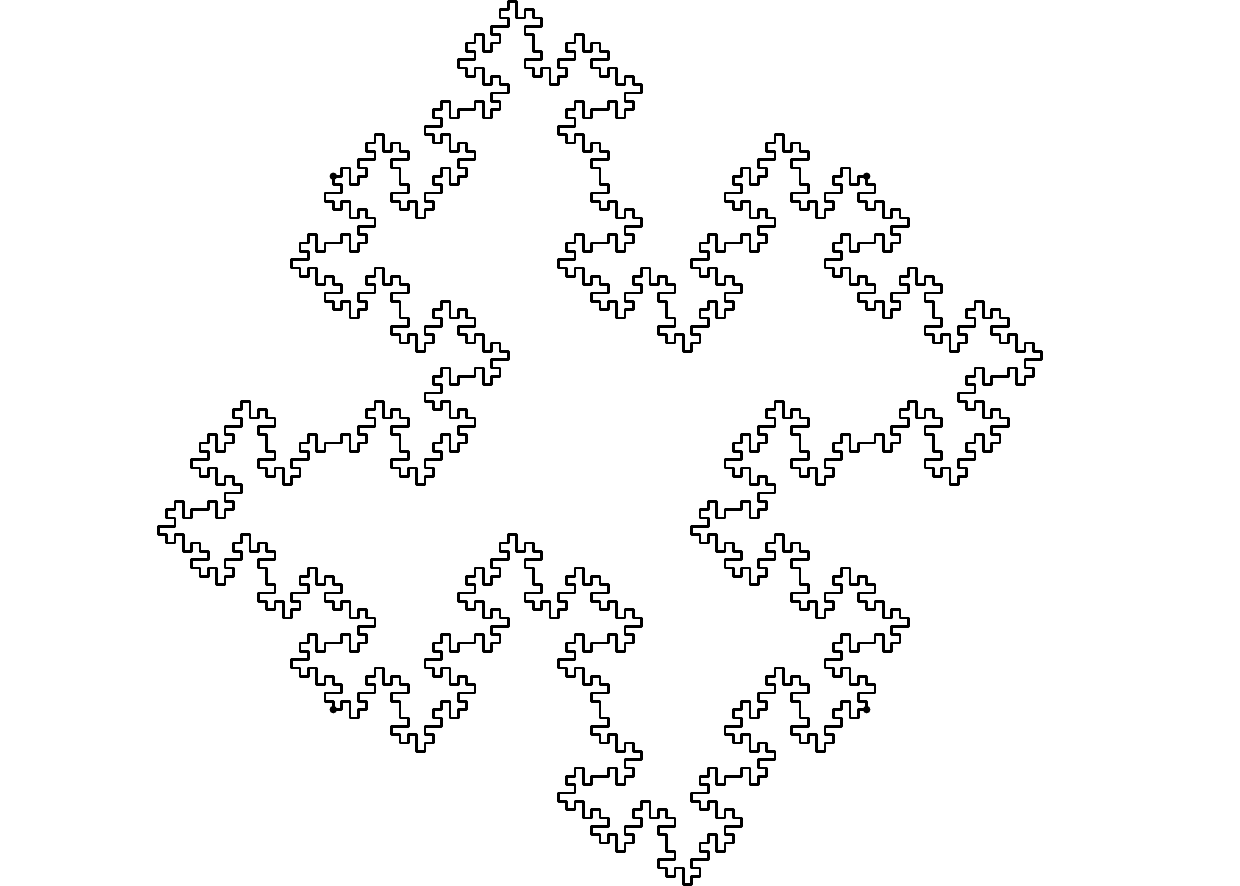}
 \caption{\label{fig:1} Examples of domains with piecewise self-similar boundaries. (Shown are in fact some polygonal approximations and not the limit curves.) \emph{Left:} Koch snowflake domain. Its boundary consists of three Koch curves intersecting only at their endpoints. The endpoints here form an equilateral triangle, but may be chosen to form any triangle. \emph{Right:} Domain $\dom$ bounded by four copies of a fractal curve generated by an IFS with 8 similarities and intersecting only at their endpoints. The Minkowski dimension of $\bd\dom$ is $d=\frac 32$.}
\end{figure}

In view of the heat transfer problems we intent to study on such domains, it seems reasonable to study separately the curvature properties of $\bd \dom$ seen from inside $\dom$ and from outside
$$
\dom_-:=\R^n\setminus \cl{\dom}.$$ This requires to adapt the notion of fractal curvatures to these one-sided situations. We will assume throughout that
the parallel sets of $\bd \dom$ are sufficiently regular in the sense of Definition~\ref{regular pairs}. More precisely, we assume that
\begin{align} \label{reg-pair-prop}
 (\ep,\bd \dom) \text{ are regular pairs for Lebesgue almost all } \ep >0.
\end{align}
Recall that this is always satisfied in $\R^2$ and $\R^3$, and recall the definition of parallel sets from \eqref{eq:parset}. Condition \eqref{reg-pair-prop} implies that curvature measures of $\bd\dom(\ep)$ are well-defined for almost all $\ep>0$ and we can introduce one-sided fractal curvatures as follows:
\begin{align} \label{eq:frac-curv-def1}
 C_k^{\Frac}(\bd\dom, \dom )&:=\elim\limits_{\ep\rightarrow 0} \ep^{\Dim-k}\, C_k\left(\bd\dom (\ep), \dom\right)=\elim_{\ep\rightarrow 0} \ep^{\Dim-k}\, C_k\left(\dom_- (\ep)\right)\\
 \label{eq:frac-curv-def2}
C_k^{\Frac}(\bd\dom, \dom_- )&:=\elim_{\ep\rightarrow 0} \ep^{\Dim-k}\, C_k\left(\bd\dom (\ep), \dom_-\right)=\elim_{\ep\rightarrow 0} \ep^{\Dim-k}\, C_k\left(\dom(\ep)\right)
\end{align}
are the \emph{inner} and \emph{outer fractal curvatures} of $\bd \dom$ (or the \emph{fractal curvatures of $\bd\dom$ relative to $\dom$ and $\dom_-$}, respectively), whenever these essential limits exist (which are understood with respect to the Lebesgue measure). 
Note that in case both of the above limits exist, the fractal curvatures of $\bd\dom$ (as defined in \eqref{eq:frac-curv-def01}) exist as well and satisfy
$$
 C_k^{\Frac}(\bd\dom)= C_k^{\Frac}(\bd\dom, \dom )+ C_k^{\Frac}(\bd\dom, \dom_- ).
$$
In order to improve the convergence behaviour, in these definitions the essential limits may be replaced by limits along discrete sequences of $\ep$, or by average limits in the same way as in \eqref{eq:av-frac-curv}, leading to \emph{inner} and \emph{outer average fractal curvatures.}

\begin{remark}
 Note that under the regularity assumption \eqref{reg-pair-prop}, both of the essential limits in \eqref{eq:frac-curv-def1} are indeed equivalent in case they exist, since $C_k\left(\bd\dom (\ep), \dom\right)=C_k\left(\dom_- (\ep)\right)$
 in this case for any regular pair $(\ep,\bd\dom)$. We point out that for the latter limit in \eqref{eq:frac-curv-def1} to be defined, regularity from one side of $\bd\dom$ would be enough, namely that almost all pairs $(\ep, \dom)$ are regular. Hence it is defined also in slightly more general situations (in $\R^n$, $n\geq 4$), in which the first essential limit might not be defined. A completely analogous remark applies to the limits in \eqref{eq:frac-curv-def2}.
\end{remark}

In the subsequent two sections we discuss two different approaches to computing such one-sided fractal curvatures and obtain existence results and explicit formulas for them for domains with piecewise self-similar boundaries, see Theorems~\ref{main} and~\ref{thm:domain-D}. 

In the remainder of the present section we recall some notations and well-known facts from the theory of self-similar sets.
First we recall the {\it distributions of the logarithmic contraction ratios} and introduce a related lattice condition adapted to our situation of piecewise self-similar boundaries.
\begin{definition}\label{logratios}
Recall that an iterated function system $\{S_1^{(j)},\ldots,S_{N^{(j)}}^{(j)}\}$ or the generated self-similar set $\Fj$ is called \emph{lattice (with lattice constant $h$)}, if the measure
\begin{\eq*}
\tnj_0:=\sum_{i=1}^{\Nj}\1_{(\cdot)}(|\ln \rj_i|)\,(\rj_i)^{\Dim}
\end{\eq*}
is lattice with constant $h$, i.e., concentrated on the set $h\N$. $\Fj$ is called \emph{non-lattice}, if no such $h$ exists.
We call the above generating system of $\partial \dom $ {\it lattice with lattice constant} $h$, if for some (or all) of the $j\in\{1,\ldots,M\}$ the measures $\tnj_0$
are lattice with the same constant $h$ and for the remaining indexes $j$ the $\tnj_0$ are non-lattice. $\partial \dom $ is called {\it completely non-lattice}, if none of the $\tnj_0$ is lattice.
\end{definition}


Next recall that the sets $F^{(j)}$ may be constructed by means of the {\it code spaces}
$$
\Wj:=\{1,\ldots,\Nj\}^\N\,,
$$
i.e., the set of all infinite words over the alphabet $\{1,\ldots,\Nj\}$. We write $\Wj_m:=\{1,\ldots,$ $\Nj\}^m$ for the set of all words $w$ of {\it length} $|w|=m$, $\Wj_*:=\bigcup_{m=1}^\infty \Wj_m$ for the set of all finite words,
 $w|m:=w_1w_2\ldots w_m$, if $w=w_1w_2\ldots w_m w
 _{m+1}\ldots$, for the {\it restriction} of a (finite or infinite) word to the first $m$ letters, and $vw$ for the {\it concatenation} of a finite word $v$ and a word $w$. If $w=w_1\ldots w_m\in\Wj_m$, then we also use the abbreviations
 $$
 \Sj_w:=\Sj_{w_1}\circ\ldots\circ\Sj_{w_m}~\text{ and } ~\rj_w:=r_{w_1}^{(j)}\ldots r_{w_m}^{(j)}
 $$
 for the mapping and the corresponding contraction ratio. Finally, we set $$\Ej_w:=\Sj_w(E)$$ for any set $E\subset\R^n$ and $w\in\Wj_{*}$. (For completeness, we also write $\Ej_\emptyset:=E$ for any $j$.) In these terms the set $\Fj$ is determined by
 $$\Fj=\bigcap_{m=1}^\infty \bigcup_{w\in \Wj_m}K^{(j)}_w$$
for an arbitrary compact set $K\subset\R^n$ with $\Sj_i(K)\subset K$, $i=1,\ldots,\Nj$.
 Iterated applications of the self-similarity property yields
\begin{\eq}\label{general-self-sim}
\Fj=\bigcup_{w\in\Wj_m}\Fj_w,\quad m\in\N.
\end{\eq}
Alternatively, the self-similar fractal $\Fj$ is the image of the code space $\Wj$ under the {\it projection} $\pi^{(j)}:\Wj\to\Fj$ given by
$$
\pi^{(j)}(w):=\lim_{m\rightarrow\infty}\Sj_{w|m}(x_0)
$$
for an arbitrary starting point $x_0$.
The mapping $w\mapsto x=\pi^{(j)}(w)$ is biunique except for a set of points $x\in\Fj$ of vanishing $\Dim$-dimensional Hausdorff measure $\H^{\Dim}$, where the {\it Hausdorff dimension} $\Dim$ of $\Fj$ is determined by \eqref{Hdim}. Up to exceptional points we {\it identify} $x\in\Fj$ {\it with its coding sequence} and write
$x_1x_2\ldots$ for this infinite word, i.e. $\pi^{(j)}(x_1x_2\ldots)=x$. For brevity, we will later on omit $\pi^{(j)}$ in this notation and write
\begin{\eq}\label{xwords}
x|m:=x_1\ldots x_m~~~{\rm and}~~~\Sj_{x|m}=\Sj_{x_1}\circ\ldots\circ
\Sj_{x_m}
\end{\eq}
for the corresponding finite words.

If $\nj$ denotes the infinite product measure on $\Wj$ of the probability measure on the alphabet $\{1,\ldots,\Nj\}$ which assigns to $i$ the probability $(\rj_i)^{\Dim}$, then the normalized $\Dim$-dimensional Hausdorff measure with support $\Fj$ is the image of $\nj$ w.r.t.\ $\pi^{(j)}$, i.e.,
\begin{equation}\label{bernoulli-meas}
\mj:=\H^{\Dim}(\Fj)^{-1}\H^{\Dim}(\Fj\cap(\cdot))=\nj\circ \pi^{-1}.
\end{equation}
It is also called the {\it self-similar measure} on $\Fj$, since we have
\begin{equation} \label{ssmu}
\mj=\sum_{i=1}^{\Nj}(\rj_i)^{\Dim} \,\mj\circ (\Sj_i)^{-1}\, .
\end{equation}
Furthermore, $\Fj$ is a $\Dim$-{\it set}. Indeed, by OSC, there exist positive constants $c_{\Fj}$ and $C_{\Fj}$ such that
\begin{equation}\label{d-set}
c_{\Fj}\, r^{\Dim}\le \H^{\Dim}(\Fj\cap B(x,r))\le C_{\Fj}\, r^{\Dim}
\end{equation}
for any $x\in \Fj$ and any $r\in(0,|\Oj|)$.

If $\Fj$ satisfies SOSC \eqref{SOSC} with respect to $\Oj$, then one obtains for any $p\in\mathbb{N}$,
\begin{equation}\label{intlogdist}
\int|\ln d(y,\partial\Oj)|^p\,\mj(dy)<\infty .
\end{equation}
(See Graf \cite[proof of Proposition 3.4]{Gr95} for $p=1$. The proof for general $p$ is similar.) This implies, in particular, that, for such $\Oj$,
\begin{equation}\label{vanishing-boundary}
\H^{\Dim}(\Fj\cap\partial\Oj)=0\,,
\end{equation}
which can also be seen using other methods.
\section{Associated fractal curvatures - local approach}\label{SecAssociatedfractalcurvatureslocalapproach}
We consider throughout a domain $\dom \subset\R^n$ with piecewise self-similar boundary as introduced in \eqref{disjoint(j)}. In this section we suppose additionally that for each self-similar set $\Fj$ in the boundary $\bd\dom$ there exists a strong feasible set $\Oj$ such that
\begin{align}
 \label{eq:O-disjoint} \Oj\cap O^{(i)}=\emptyset, \quad i\neq j.
\end{align}
In view of our boundary problem, we additionally suppose that the $\Sj_i$ are mapping some local interior neighborhoods of the boundary $\partial \dom $ into itself. More precisely, for each $j=1,\ldots,M$ there exists some open set $\Gj$ such that
\begin{\eq}\label{eq:interior}
\Fj\cap\Oj\subset\Gj\mbox{ and }\Sj_i(\dom \cap\Gj)\subset \dom \cap\Gj\,,~i=1,\ldots,\Nj.
\end{\eq}
We approximate $\dom$ by outer parallel sets of small distances and consider limits of associated rescaled curvatures as defined in \eqref{eq:frac-curv-def2}. In order to use the piecewise self-similar structure of the boundary and the classical Renewal theorem, we apply here an appropriate localization procedure by means of bump functions.

\begin{definition}\label{bump functions}
Let $a>1$ and $g:[0,\infty)\rightarrow[0,1]$ be some function with supp$(g)=[0,1]$, which is continuous on $[0,1]$ and satisfies $g(t)>0$ for $t\in[0,1)$. Then the functions $A_{\partial \dom }(x,\ep):\R^n \to [0,\infty)$, defined by
$$A_{\partial \dom }(x,\ep)(z):=\frac{\ep^{\Dim}\, g\left(\frac{|x-z|}{a\ep}\right)}{\int_{\partial \dom } g\left(\frac{|y-z|}{a\ep}\right)\, \H^{\Dim}(dy)}\, , ~~{\rm if}~|z-x|<a\ep\, ,$$
and $A_{\partial \dom }(x,\ep)(z):=0$ for $|x-z|\ge a\ep$, with parameters $x\in \partial \dom $, $0<\ep\leq 1$, are collectively called a \emph{localizing family of bump functions} associated with $\bd\dom $.
\end{definition}
First observe that the function $A_{\partial \dom }(x,\ep)$ is well-defined for any $x\in \partial \dom $ and any $\ep\leq 1$. Indeed, for fixed $x$ and $z$ with $\frac{|x-z|}{a\ep}<1$ denote $\delta:=1-\frac{|x-z|}{a\ep}$ and $r:=(\delta-\delta')a\ep$ for some $\delta'<\delta$. Then the integral in the denominator is nonzero.
To see this note that for $x\in \Fj$ and $y\in B(x,r)$ we get $\frac{|y-z|}{a\ep}\le \frac{|y-x|}{a\ep}+\frac{|x-z|}{a\ep}\leq \delta-\delta'+1-\delta=1-\delta'$. Hence,
\begin{align*}
&\int_{\partial \dom } g\left(\frac{|y-z|}{a\ep}\right) \H^{\Dim}(dy)\ge\int_{\Fj\cap B(x,r)}g\left(\frac{|y-z|}{a\ep}\right)\H^{\Dim}(dy)\\
&\geq \min_{u\in[0,1-\delta']}\limits g(u)\,\H^{\Dim}(\Fj\cap B(x,r))>\const r^{\Dim}>0.
\end{align*}
Therefore the function $A_{\partial \dom }(x,\ep)$ is well-defined.
Standard examples for $g$ are the indicator function $\1_{[0,1]}$, or $g(u)=\max\{1-u,0\}$, $u>0$.
Note that the function $(x,\ep,z)\mapsto A_{\partial \dom }(x,\ep)(z)$ is Borel-measurable. Moreover, setting
\begin{align} \label{eq:b}
 b:=\max\left\{2a,\max_{m,\eta(t)}|\Oj|\right\}
\end{align}
and recalling the notation \eqref{xwords}, the bump functions $A_{\partial \dom }(x,\ep)$ satisfy the following properties.
\begin{proposition}\label{bump-prop}
\begin{\eq}\label{support}
{\rm supp}(A_{\partial \dom }(x,\ep))=B(x,a\ep)\,,
\end{\eq}
\begin{\eq}\label{integralA}
\int_{\partial \dom } A_{\partial \dom }(x,\ep)(z)\H^{\Dim}(dx)=\ep^{\Dim} \quad \text{ for all } z\in\partial \dom (\ep)\,,
\end{\eq}
\begin{\eq}\label{bounded}
A_{\partial \dom }(x,\ep)(z)\leq C~~\mbox{for all}~~x\in\partial \dom ,\, z\in\partial \dom (\ep),\,\,\ep<1,~~\mbox{and some constant}~C.
\end{\eq}
 Furthermore, if $x\in\Fj$, $\ep\leq 1$ and $m\in\N$ are such that $\ep<b^{-1}d(x,\partial\Oj_{x|m})$ holds, then
\begin{\eq}\label{covariance}
A_{\partial \dom }(x,\ep)(z)=A_{\partial \dom }\big((\Sj_{x|m})^{-1}x,(\rj_{x|m})^{-1}\ep\big)\big((\Sj_{x|m})^{-1}z\big), ~~\text{ for all } z.
\end{\eq}
\end{proposition}
\begin{proof}
Equation \eqref{support} is obvious.
In order to show \eqref{integralA} and \eqref{bounded} note that for any $z\in\partial \dom (\ep)$ there exists a $u\in\partial \dom $ such that $|u-z|=\ep$. Then we get for $0<\delta<a-1$ and any $y\in\partial \dom \cap B(u,\delta\ep)$,
$$\frac{|y-z|}{a\ep}\leq \frac{|y-u|}{a\ep}+\frac{|u-z|}{a\ep}\leq\frac{\delta}{a}+\frac{1}{a}=\frac{\delta+1}{a}=:c<1.$$ Consequently,
\begin{align*}
\int_{\partial \dom } g&\left(\frac{|y-z|}{a\ep}\right) \H^{\Dim}(dy)\ge\int_{\partial \dom \cap B(u,\delta\ep)}g\left(\frac{|y-z|}{a\ep}\right)\H^{\Dim}(dy)\\
\geq &\min_{u\in[0,c]}\limits g(u)\,\H^{\Dim}(\partial \dom \cap B(u,\delta\ep))>\const \delta^{\Dim}\ep^{\Dim}:=C'\,\ep^{\Dim}.
\end{align*}
Note that the constant $C'$ depends only on $\delta$, i.e., on $a$.
In the last inequality we have used again the $\Dim$-set property of the Hausdorff measure on $\partial \dom $ which follows from those for the components $\Fj$. This estimate together with the definition of $A_{\partial \dom }$ implies, in particular, \eqref{integralA}. Furthermore,
since the function $g$ in the numerator of the bump function is not greater than 1, we obtain \eqref{bounded}.\\
Finally,
\eqref{covariance} can be seen as follows. Recall that ${\rm supp}(g)=[0,1]$.
Let $x\in\Fj$, $\ep>0$ and $m\in\N$ such that $\ep\le b^{-1}d(x,\partial\Oj_{x|m})$ holds. Note that this inequality implies $(\rj_{x|m})^{-1}\ep\le 1$. For any $z$ as above with $z\notin B(x,a\ep)$, we have $(\Sj_{x|m})^{-1}z\notin B((\Sj_{x|m})^{-1}x,a(\rj_{x|m})^{-1}\ep)$ and thus both sides of \eqref{covariance} will be zero. So let from now on $z\in B(x,a\ep)$. Since $b\ge 2a$, we obtain for any $y\in\partial \dom $ with $|y-z|\le a\ep$ that $|y-x|\le|y-z|+|z-x|\le 2a\ep\leq d(x,\partial\Oj_{x|m})$. This implies that $y\in\Fj_{x|m}$ by OSC and the disjointness of the $\Oj$ for different $j$. Consequently,
\begin{align}
\int_{\bd \dom } &g\left(\frac{|y-z|}{a\ep}\right)\, \H^{\Dim}(dy)
=\int_{\Fj_{x|m}} g\left(\frac{|y-z|}{a\ep}\right)\, \H^{\Dim}(dy)\notag\\
&=\int_{\Fj_{x|m}} g\bigg(\frac{\big|(\Sj_{x|m})^{-1}y-(\Sj_{x|m})^{-1}z\big|}{a(\rj_{x|m})^{-1}\ep}\bigg)\,\H^{\Dim}(dy)\notag\\
&=(\rj_{x|m})^{\Dim}\int_{\Fj}g\bigg(\frac{\big|y-(\Sj_{x|m})^{-1}z\big|}{a(\rj_{x|m})^{-1}\ep}\bigg)\,\H^{\Dim}(dy) \label{eq:denom-int1},
\end{align}
where for the second equality we have used that $\Sj_{x|m}$ is a similitude with contraction ratio $\rj_{x|m}$, which implies also
$$
g\bigg(\frac{|x-z|}{a\ep}\bigg)=g\bigg(\frac{\big|(\Sj_{x|m})^{-1}x-(\Sj_{x|m})^{-1}z\big|}{a(\rj_{x|m})^{-1}\ep}\bigg).
$$
The third equality is due to the scaling properties of the Hausdorff measure on $\Fj$.
Furthermore, 
we have
$d(x,\partial\Oj_{x|m})=
\rj_{x|m}d((\Sj_{x|m})^{-1}x,\partial\Oj)$ and $|(\Sj_{x|m})^{-1}z-(\Sj_{x|m})^{-1}x|\le a(\rj_{x|m})^{-1}\ep$. Therefore, we obtain for any $y\in\partial \dom $ and any $z$
 with $|y-(\Sj_{x|m})^{-1}z|\le a(\rj_{x|m})^{-1}\ep$ that
$$|y-(\Sj_{x|m})^{-1}x|\le 2a\,\ep(\rj_{x|m})^{-1}\le b\,\ep(\rj_{x|m})^{-1}\le d(\Sj_{x|m})^{-1}x,\partial\Oj).$$ Hence, $y\in\Fj$ and therefore integration over $\Fj$ in \eqref{eq:denom-int1} may be replaced by integration over $\partial \dom $. This yields
$$\int_{\partial \dom } g\left(\frac{|y-z|}{a\ep}\right)\, \H^{\Dim}(dy)=(\rj_{x|m})^{\Dim}\int_{\partial \dom }g\bigg(\frac{\big|y-(\Sj_{x|m})^{-1}z\big|}{a(\rj_{x|m})^{-1}\ep}\bigg)\,\H^{\Dim}(dy).$$
The last two equalities lead to \eqref{covariance}.
 \end{proof}
Now we are ready to state the main result of this section regarding the limiting behaviour of the Lipschitz-Killing curvature measures $C_k(\dom (\ep),\cdot)$, $k=0,\ldots,n-1$, as $\ep\rightarrow 0$, provided the corresponding regularity condition \eqref{reg-pair-prop} is fulfilled.
For brevity, we will write
$$
C_k(\dom (\ep), \varphi):= \int \varphi(z)\,C_k(\dom (\ep),dz) \quad \text{ for~ integrable~ functions } \varphi.
$$
\begin{theorem}\label{main}
Let $\dom $ be a bounded open set in $\rn$ with
piecewise self-similar boundary as defined in \eqref{disjoint(j)}, i.e.,
 $\partial \dom =\bigcup_{j=1}^M\Fj$, where the $\Fj$ are self-similar sets with the same Hausdorff dimension $\Dim$ satisfying SOSC \eqref{SOSC}. For each $\Fj$ there exists a strong feasible set $\Oj$ such that \eqref{eq:O-disjoint} and 
 \eqref{eq:interior} hold.\\
Let $k\in\{0,\ldots, n-1\}$.
If $k\le n-2$, then we suppose additionally that the regularity condition \eqref{reg-pair-prop} is satisfied
(which is always true for $n\le 3$) and that the estimate
\begin{equation}\label{ibound}
\esup_{\ep\leq 1}\limits\frac{1}{f(\ep)}\int_{\Fj}{\bf 1}\left\{d(y,\partial\Oj_{y_1})\leq b\,\ep\right\}\, \ep^{-k}C_k^{\var}\left(\dom (\ep),A_{\partial \dom }(y,\ep)\right)\, \mathcal{H}^{\Dim}(dy)<\infty
\end{equation}
holds for all $j\in\{1,\ldots,M\}$, some bounded function $f:(0,\infty)\mapsto (0,\infty)$ such that $t\mapsto f(e^{-t})$, $t\in(0,\infty)$, is directly Riemann integrable, and some family of bump functions $A_{\bd \dom }(x,\ep)$ as in Definition~\ref{bump functions} with constants $a>1$ and $b$ 
as in \eqref{eq:b}.

Then we obtain the following limits, where $\ej:=\sum_{i=1}^{\Nj}|\ln \rj_i|\,(\rj_i)^{\Dim}$, $j=1,\ldots,M$.
\begin{itemize}
\item[{\rm (i) }]\label{ordlim}
If the generating system for $\partial \dom $ is completely non-lattice (cf.~Def.~\ref{logratios}), then the outer fractal curvatures of $\bd \dom$ (as defined in \eqref{eq:frac-curv-def2}) exist and are given by
$$
C_k^{\Frac}(\bd \dom, \dom_- )
=\sum_{j=1}^M \Ij_k, \text{ where }
$$
$$\Ij_k:=(\ej)^{-1}\int_{\Fj}\int_{b^{-1}d(x,\partial\Oj_{x|1})}^{b^{-1}d(x,\partial\Oj)}\ep^{-(k+1)}C_k\left(\dom (\ep), A_{\partial \dom }(x,\ep)\right)\,d\ep\, \mathcal{H}^{\Dim}(dx).
$$
\item[{\rm (ii)}] \, If there is some $\ell<M$ such that the sets $F^{(1)},\ldots,F^{(\ell)}$ are non-lattice, and the others are lattice with the same lattice constant $h$, then one has for Lebesgue-a.a. $u\in[0,h)$,
$$
\lim_{n\rightarrow\infty}e^{(u+nh)(\Dim-k)}\, C_k\big(\dom (e^{-(u+nh)})\big)
=\sum_{j=1}^\ell \Ij_k+\sum_{j=\ell+1}^M \Lj_k(u),
$$
where $\Ij_k$ is as above and $\Lj_k(u)$ is given by
\begin{align*}
(\ej)^{-1}\sum_{m=1}^\infty e^{-k(u+mh)} \int_{\Fj}\1\left\{b^{-1}d(x,\partial\Oj_{x|1})\leq e^{-(u+mh)}<b^{-1}d(x,\partial\Oj)\right\}\\
 C_k\big(\dom (e^{-(u+mh)}), A_{\bd \dom }(x,e^{-(u+mh)})\big)
\, \mathcal{H}^{\Dim}(dx).
\end{align*}
\end{itemize}
\end{theorem}
\begin{proof}The main arguments for (i) and (ii) are similar as in the proof of Theorem 3.1 in \cite{Za13}, where a special bump function was used. We will give a detailed proof for our situation.\\
First recall from \eqref{regular distances} the definition of the set $\reg$ of regular distances to a set. We now denote
\begin{\eq}\label{reg*}
\reg_*(\dom ):=\bigcap_{j=1}^M\bigcap_{w\in\Wj_*}\reg\left((\Sj_w)^{-1}(\dom )\right).
\end{\eq}
From \eqref{intersection reg} we obtain
\begin{\eq}\label{comp reg}
\L\left(\reg_*(\dom )^c\right)=0.
\end{\eq}
Note that $\ep\in\reg_*(\dom )$ implies that $(r_w)^{-1}\ep\in\reg_*(\dom )$
for all $w\in\Wj_*$, $j=1,\ldots,M$. In the sequel we consider only such $\ep$, i.e., Lebesgue almost all.

Next recall that $\H^{\Dim}$-a.a. points $x\in\Fj$ have a unique coding sequence $x_1x_2x_3\ldots$, which we identify with $x$. If we define for such $x$,
$T(x):=(\Sj_{x_1})^{-1}(x)$, then $T(x)\in\Fj$ and we can again identify $T(x)$ with its coding sequence, i.e., $$T(x)=x_2x_3\ldots.$$
Recall that $b=\max\left\{2a,\max_{m,\eta(t)}|\Oj|\right\}$.
Let now $x\in\Fj$, $\ep\leq 1$ and $m\in\N$ be such that the condition
$\ep<b^{-1}d(x,\partial \Sj_{x|m}\Oj)$ is satisfied. Then
$$d\left(x,\bd\Oj_{x|m}\right)=\rj_{x|m}\, d\left(T^m x,\partial\Oj\right),$$
and therefore $\ep<(\rj_{x|m})^{-1}\ep\leq 1$ because of the choice of the constant $b$.
Furthermore, this condition implies that, for any $a'$ with $a<a'<2a-1$,
\begin{align}\label{locB}
\dom (\ep)\cap B(x,a'\ep)^{\circ}&=\left(\Sj_{x|m}(\dom )(\ep)\right)\cap B(x,a'\ep)^{\circ}\notag\\
&=\Sj_{x|m}\left(\dom ((\rj_{x|m})^{-1}\ep)\right)\cap B(x,a'\ep)^{\circ}.
\end{align}
In order to show the first equality, first note that by SOSC with disjoint $\Oj$ and the inclusion property \eqref{eq:interior} we get
$$\dom \cap B(x,a'\ep)^{\circ}= \Sj_{x|m}(\dom )\cap B(x,a'\ep)^{\circ}.$$
Then it suffices to show that
$$ \left(\Sj_w F^{(l)}\right)(\ep)\cap B(x,a'\ep)^{\circ}=\emptyset$$
for all pairs $(l,w)$ of numbers and words such that either $l=j$, $w\in W^{(j)}_m$, $w\neq x|m$, or $l\neq j$, $w=x|m$ or $w=\emptyset$. (This means that there are no contributions to the parallel set $\dom (\ep)$ within the ball $B(x,a'\ep)^{\circ}$ arising from the other parts of the boundary $\partial \dom $.)\\
Suppose the contrary, i.e., that there is some $y\in(\Sj_w F^{(l)})(\ep)\cap B(x,a'\ep)^{\circ}$ for such a pair $(l,w)$. Then we have $|x-y|<a'\ep$ and there exists some $z\in\Sj_w(F^{(l)})$ such that $|y-z|\le\ep$. Hence,
$|x-z|\leq|x-y|+|y-z|<a'\ep+\ep$ and thus, $d(x,\Sj_w(F^{(l)})\leq(a'+1)\ep<2a\ep$. By SOSC with disjoint $\Oj$ this is in all three cases for the pairs $(l,w)$ a contradiction to the condition $d(x,\partial\Oj_{x|m})\geq b\ep\geq 2a\ep$. Consequently, the above intersection sets are empty, which leads to \eqref{locB}.\\
Therefore we can use the locality \eqref{locality} of curvature measures and ${\rm supp}(A_{\partial \dom }(x,\ep))= B(x,a\ep)$ in order to obtain for $\ep\in
\reg_*(\dom )\cap(0,1]$, $\H^{\Dim}$-a.a.\ $x\in\Fj$, $m\in\mathbb{N}$ and $\ep<b^{-1}d(x,\partial S^{(j)}_{x|m}(O^{(j)}))$ the following.
\begin{align*}
C_k\left(\dom (\ep),A_{\bd \dom }(x,\ep)\right)&=C_k\left(\left(\Sj_{x|m}\dom\right)(\ep), A_{\partial \dom }(x,\ep)\right)\\
&=C_k\left(\Sj_{x|m}\left(\dom \left((\rj_{x|m})^{-1}\ep\right)\right), A_{\partial \dom }(x,\ep)\right).
\end{align*}
By the covariance representation \eqref{covariance} of the bump functions, the last expression is equal to
$$C_k\left(\Sj_{x|m}\left(\dom \left((\rj_{x|m})^{-1}\ep\right)\right), A_{\bd \dom }\left(T^m x,(\rj_{x|m})^{-1}\ep\right)\circ(\Sj_{x|m})^{-1}\right).$$
Finally, using the scaling property of the curvature measures under similarities we infer for $\ep,x,m$ as above,
\begin{\eq}\label{aux}
C_k\big(\dom (\ep),A_{\partial \dom }(x,\ep)\big)=(\rj_{x|m})^k C_k\big(\dom ((\rj_{x|m})^{-1}\ep),A_{\partial \dom }(T^m x,(\rj_{x|m})^{-1}\ep)\big).
\end{\eq}
Recall \eqref{integralA}, i.e.,
$\int_{\partial \dom } A_{\partial \dom }(x,\ep)(z)\H^{\Dim}(dx)=\ep^{\Dim}$ for all $z\in\partial(\dom (\ep))$. Together with Fubini's theorem and \eqref{aux} this allows to localize the rescaled total curvatures as follows.
\begin{align}\label{S1S2}
\ep^{\Dim-k}&C_k\big(\dom (\ep)\big)=
\int \ep^{-k}\int A_{\partial\dom}(x,\ep)(z)\H^{\Dim}(dx)C_k(\dom (\ep),dz)\nonumber\\
& =\int_{\partial \dom }\ep^{-k} C_k\big(\dom (\ep),A_{\partial \dom }(x,\ep)\big)\,\H^{\Dim}(dx)\nonumber\\
&=\sum_{j=1}^M\sum_{m=1}^\infty\int_{\Fj}\1\left(b^{-1}d(x,\partial \Sj_{x|m+1}(\Oj)\le\ep< b^{-1}d(x,\partial\Sj_{x|m}(\Oj)\right)\nonumber\\
&\hspace{3cm}\ep^{-k} C_k\big(\dom (\ep),A_{\partial\dom}(x,\ep)\big)\,\H^{\Dim}(dx)\nonumber\\
&+\sum_{j=1}^M\int_{\Fj}\1\left(b^{-1}d(x,\partial\Sj_{x|1}(\Oj))\le\ep\right)\ep^{-k} C_k\big(\dom (\ep),A_{\partial \dom }(x,\ep)\big)\,\H^{\Dim}(dx)\nonumber\\
&=:S_1(\ep)+S_2(\ep)\, .
\end{align}
By our integrability assumption the last sum $S_2(\ep)$ vanishes as $\ep\rightarrow 0$, where $\ep\in\reg_*(\dom )$. In the first sum $S_1(\ep)$ we now will consider the summands with different $j$ separately and derive the corresponding limit expression. Summation over $j$ then leads to the assertion.
In view of \eqref{aux} we get for such a sum the expression
\begin{align*}
&\sum_{m=1}^\infty\int_{\Fj}\1\left(b^{-1}\rj_{x|m}d(T^mx,\partial\Sj_{(T^mx})_1(\Oj))\le\ep<
b^{-1}\rj_{x|m}d(T^mx,\partial\Oj)\right)\\
&\hspace{2cm} \ep^{-k} C_k\big(\dom (\ep),A_{\partial \dom }(x,\ep)\big)\,\H^{\Dim}(dx)\\
=&\sum_{m=1}^\infty\int_{\Fj}\1\left(b^{-1}d(T^mx,\partial\Sj_{(T^mx)_1}(\Oj))\leq(\rj_{x|m})^{-1}\ep<b^{-1}d(T^mx,\partial\Oj)\right)\\
&\hspace{2cm}\ep^{-k}(\rj_{x|m})^k C_k\big(\dom ((r_{x|m})^{-1}\ep),A_{\partial \dom }(T^mx,(\rj_{x|m})^{-1}\ep)\big)\H^{\Dim}(dx)\, .
\end{align*}
Recall that we identify the points $x\in\Fj$ with their coding sequences and that the normalized Hausdorff measure $\mj$ on $\Fj$ is the image of the product measure $\nj$ on the code side. Substituting $x=x_1\ldots x_mx_{m+1}\ldots=:x_1\ldots x_my_1y_2\ldots=x|m\, y$ in the integrals we obtain for the last expression
\begin{align*}
&\sum_{m=1}^\infty\int_{\Fj}\1\left(b^{-1}d(y,\partial\Sj_{y_1}(\Oj)\le (\rj_{x|m})^{-1}\ep< b^{-1}d(y,\partial\Oj)\right)\\
&\hspace{2cm}\ep^{-k}(\rj_{x|m})^k C_k\big(\dom ((\rj_{x|m})^{-1}\ep),A_{\partial \dom }(y,(\rj_{x|m})^{-1}\ep)\big)
\H^\Dim (d(x|n\, y))\\
=&\sum_{m=1}^\infty\int_{\Fj}\int_{\Fj}\1\left(b^{-1}d(y,\partial\Sj_{y_1}(\Oj))\le(\rj_{x|m})^{-1}\ep< b^{-1}d(y,\partial\Oj)\right)\\
&\hspace{2cm}\ep^{-k}(\rj_{x|m})^k C_k\big(\dom ((\rj_{x|m})^{-1}\ep),A_{\partial \dom }(y,(\rj_{x|m})^{-1}\ep)\big)\H^{\Dim}(dy)\mu(dx).
\end{align*}
In order to translate this into the language of the classical Renewal theorem we denote the distribution of $\rj_{x|m}$ with respect to $\mj$, i.e., of the product of $m$ i.i.d. random variables, by $\nj_m$ and that of their absolute logarithmic values by $\tnj_m$. Then $\tnj_m$ is the $m$-{\it th convolution power} of the discrete distribution
$\tnj_0=\sum_{i=1}^{\Nj}\1_{(\cdot)}(|\ln \rj_i|)\, (\rj_i)^{\Dim}\, .$
We further consider the associated {\it renewal function} $$\Uj(s):=\sum_{m=1}^\infty\tnj_m((0,s])$$ and substitute $\ep$ by $e^{-t}$, i.e., $t\in T$, where
$$T:=\{t>0:\,e^{-t}\in\reg_*(\dom )\}.$$
In these notations we get for the right-hand side of the above equation the expression
\begin{align*}
&\sum_{m=1}^\infty\int\int_{\Fj}\1\left(b^{-1}d(y,\partial\Sj_{y_1}(\Oj))\le r^{-1}\ep< b^{-1}d(y,\partial\Oj)\right)\\
&\hspace{2cm}\ep^{-k}r^k C_k\big(\dom (r^{-1}\ep),A_{\partial \dom }(y,r^{-1}\ep)\big)\H^{\Dim}(dy)\nj_m(dr)\\
=&\sum_{m=1}^\infty\int_0^t\int_{\Fj}\1\left(b^{-1}d(y,\partial\Sj_{y_1}(\Oj))\le e^{s-t}< b^{-1}d(y,\partial\Oj)\right)\\
&\hspace{4cm}e^{(t-s)k} C_k\big(\dom (e^{s-t}),A_{\partial \dom }(y,e^{s-t})\big)\H^{\Dim}(dy)\tnj_m(ds)\\
=&\int_0^t\int_{\Fj}\1\left(b^{-1}d(y,(\Sj_{y_1}(\Oj))\le e^{s-t}< b^{-1}d(y,\partial\Oj)\right)\\
&\hspace{4cm}e^{(t-s)k} C_k\big(\dom (e^{s-t}),A_{\partial \dom }(y,e^{s-t})\big)\H^{\Dim}(dy)d\Uj(s)\\
=&\int_0^t \zj(t-s)\, d\Uj(s)=:\Zj(t)
\end{align*}
where $\zj(t)$ is defined by
$$
\int_{\Fj}\1\left\{d(y,\partial\Oj_{y|1})\le be^{-t}<d(y,\partial\Oj)\right\}e^{tk} C_k\big(\dom (e^{-t}),A_{\partial \dom }(y,e^{-t})\big)\H^{\Dim}(dy)\,,$$
defined on $T$.
We now introduce two auxiliary functions on $(0,\infty)$ by
\begin{align*}
\bzj(t)&:=\begin{cases}
 \zj(t), &{\rm if}~t\in T\\
 \limsup\limits_{\stackrel{t'\to t}{t'\in T}} z(t'), &{\rm if}~t\in (0,\infty)\setminus T,
\end{cases}\\
\bZj(t)&:=\int_0^t\bzj(t-s)d\Uj(s),~t>0.
\end{align*}
Then we infer from the last of the above equalities that
\begin{\eq}\label{Zbar}
\bZj(t)=\Zj(t) \quad \text{ for } t\in T.
\end{\eq}
By assumption \eqref{ibound}, $|\bzj|$ is bounded by a directly Riemann integrable function. (For $k=n-1$ this is always satisfied, see Proposition~\ref{riem-int-suf} below.) Furthermore, if $z(t)$ is continuous at all $t\in T$, then by construction, $\bzj$ is also continuous at these $t$, i.e., at Lebesgue almost all. (It can be seen by means of convergence for sequences of points.) In this case $\bzj$ is directly Riemann integrable, see e.g.\ Asmussen \cite[Prop.~4.1, p.118]{As87}.
Therefore the Renewal theorem in Feller \cite[p.363]{Fel71} tells us that in the non-lattice case
$$\lim_{t\rightarrow\infty}\limits\bZj(t)=(\ej)^{-1}\int_0^\infty\bzj(t)\, dt\, .$$
The right hand side agrees with
$$\frac{1}{\ej}\int_0^\infty\zj(t)\,dt\,,$$
since $\bzj(t)=\zj(t)$ for Lebesgue-a.a. $t$.

Substituting then $e^{-t}$ by $\ep$ and summing over $j$ one obtains the first assertion, taking into regard that $\sum_{j=1}^M\Zj(t)=S_1(e^{-t})$ in equations \eqref{S1S2}.

In the lattice case the Renewal theorem leads for Lebesgue-a.a. $u\in[0,h)$ to
$$\lim_{n\rightarrow\infty}\Zj(u+nh)=\frac{h}{\ej}\sum_{m=1}^\infty\zj(u+mh).$$
Then summation over $j$ yields the second assertion, using the above convergence for those $j$, where the corresponding system is non-lattice.

Thus, it remains to show that $\zj(t)$ is continuous at $t\in T$, equivalently, that the function $\zj(-\ln\ep):=$
$$\int_{\Fj}\hspace{-2mm}\1\left\{b^{-1}d(y,\bd\Sj_{y_1}\Oj)\le\ep< b^{-1}d(y,\bd\Oj)\right\}C_k\big(\dom (\ep),A_{\partial \dom }(y,\ep)\big)\H^{\Dim}(dy)$$
is continuous at $\ep\in\reg_*(\dom )\cap(0,1]$. To this aim we will use the weak convergence \eqref{reg-convergence},
$$\lim_{\ep'\rightarrow\ep}C_k(\dom (\ep'),\cdot)=C_k(\dom (\ep),\cdot)$$
at all $\ep\in\reg_*(\dom )$. For brevity we denote the indicator function under the integral by $\1(y,\ep)$. By Fubini's theorem the above function is equal to
$$\int\int_{\Fj}\1(y,\ep)A_{\partial \dom }(y,\ep)(z)\,\H^{\Dim}(dy)\, C_k(\dom (\ep),dz)\, .$$
Note that the measure $\H^{\Dim}$ restricted to $\Fj$ has no mass on spheres (see \cite[Lemma 2.1.3]{RZ12}). Therefore the function $\int_{\Fj}\1(y,\ep)A_{\partial \dom }(y,\ep)(z)\,\H^{\Dim}(dy)$ is continuous at any $(z,\ep)\in\rn\times(0,1]$ and hence, uniformly continuous on compact neighborhoods. Consequently, the function
$$\int_{\Fj}\1(y,\ep')A_{\partial \dom }(y,\ep')(z)\,\H^{\Dim}(dy)$$ is continuous at $z$ uniformly in $\ep'$ in a neighborhood of $\ep$. This together with the above weak continuity of $C_k$ implies the desired continuity of the above function $\zj(-\ln\ep)$ at all $\ep\in\reg_*(\dom )\cap(0,1]$.
\end{proof}
The following property of the surface area of the parallel sets will be useful in the subsequent statement.
\begin{remark}\label{rem:surface-est}
The surface area of the parallel sets $K(r)$
of arbitrary compact sets $K\subset\rn$ has the following property:
\begin{equation} \label{sa}
\H^{n-1}(\partial(K(r)))\leq \frac nr \H^n(K(r)),\quad r>0.
\end{equation}
This follows from the \emph{Kneser property} of the volume function $r\mapsto \H^n(K(r))$ of $K$, which implies $\frac{d}{dr}\H^{n}(K(r))\leq\frac{n}{r}\H^n(K(r))$, see e.g.\ Rataj, Schmidt and Spodarev \cite[Lemma~4.6 and its proof]{RSS}, and from the fact that $\frac{{\rm d}}{{\rm d}r}\H^n(K(r))=\H^{n-1}(\partial(K(r)))$ for all $r>0$ except countably many. (See \cite{RW10} and \cite{RW13} for more details.)
\end{remark}
Now we provide a sufficient condition for the local estimate \eqref{ibound}.
\begin{proposition}\label{riem-int-suf} If
 $$ c_{\max}:=\max_{m,\eta(t)}\esup_{\ep\le 1}\limits\sup_{y\in\Fj}\ep^{-k}\,
C^{\var}_k\left(\Fj(\ep),B(y,a\ep)\right)<\infty\, ,$$
which is always true for $k=n-1$, then for any $p>1$ there is some constant $C>0$ such that the function $f$ with $$f(e^{-t})=C\min\left(1,t^{-p}\right)$$ fulfills condition \eqref{ibound} of Theorem~\ref{main}.
\end{proposition}
\begin{proof}

First recall that $C_{n-1}(\dom (\ep),\cdot)=\frac{1}{2}\H^{n-1}(\dom (\ep)\cap (\cdot))$. Then for $k=n-1$ the condition of the lemma follows from the corresponding estimate for the volume measure $\H^n(\dom (\ep)\cap(\cdot))$ together with \eqref{sa} (see \cite[Remark 3.22]{RZ12}).

Next note that $\min(1,t^{-p})$ is bounded and directly Riemann integrable on $[0,\infty)$. Moreover,
\begin{align*}
&\hspace{-0.5cm}\left| \int_{\Fj}\1\left\{b^{-1}d\left(y,\partial\Sj_{y_1}\Oj\right)\le\ep\right\}\, \ep^{-k}C_k\left(\dom (\ep),A_{\partial \dom }(y,\ep)\right)\, \mathcal{H}^{\Dim}(dy)\right|\\
&\le \int_{\Fj}\1\left\{b^{-1}d(y,\bd\Sj_{y_1}\Oj)\le\ep\right\}\, \ep^{-k}C_k^{\var}\left(\Fj(\ep),B(y,a\ep)\right)\, \mathcal{H}^{\Dim}(dy)\\
&\le c_{\max} \int_{\Fj} \1\left\{b^{-1}d(y,\partial\Sj_{y_1}\Oj)\le\ep\right\}\, \H^{\Dim}(dy)\\
&\le c_{\max} \int_{\Fj} \1\left\{\left|\ln\left(b^{-1} d(y,\bd \Sj_{y_1}\Oj)\right)\right|\ge |\ln \ep|\right\}\,\H^{\Dim}(dy)\\
&\le c_{\max} |\ln\ep|^{-p}\int_{\Fj}\left|\ln\left(b^{-1} d(y,\bd\Sj_{y_1}\Oj)\right)\right|^p\,\H^{\Dim}(dy)\\
&\le c_{\max} |\ln\ep|^{-p}\int_{\Fj}\left(\const+\left|\ln d(x,\partial\Oj)\right|\right)^p\,\H^{\Dim}(dx)\le \const |\ln\ep|^{-p}
\end{align*}
with varying constants in the estimates. Here we have applied the Markov inequality together with the moment property \eqref{intlogdist} of the logarithmic distance function. Setting $\ep=e^{-t}$ we infer the desired estimate if $|\ln\ep|>1$. In the case $|\ln\ep|\le1$ we use the fact that the whole integral is bounded by $c\,\H^{\Dim}(\Fj)$.
\end{proof}
\begin{remark} \label{rem:avlim}
An analysis of the proofs shows that in Theorem~\ref{main} and Proposition~\ref{riem-int-suf} the essential supremum and the essential limit can always be replaced by those over the set $\reg_*(\dom )$ from \eqref{reg*}. Moreover, completely analogous results hold for the inner fractal curvatures as defined in \eqref{eq:frac-curv-def1}.

Furthermore, in the general case, i.e., independently of whether the generating system is lattice or non-lattice, under some integrability conditions the following average limits exist:
\begin{enumerate}
\item[{\rm(i)}]
for $\mathcal{H}^{\Dim}$-a.a. $x\in \Fj$ the {\it average density} at $x$
\begin{\eq*}
\lim_{\delta\rightarrow 0}\frac{1}{|\ln\delta|}\int_\delta^{b^{-1}d(x,\partial\Oj)} \ep^{-k}C_k\big(\Fj(\ep), A_{\partial \dom }(x,\ep)\big)\, \frac{d\ep}\ep,
\end{\eq*}
exists and assumes the same value $\mathcal{D}^{(j)}_{k,\bd\dom }$, which is given by 
the constant $\Ij_k$ from Theorem~\ref{main} {\rm (i)}, i.e.,
$\mathcal{D}^{(j)}_{k,\bd\dom }=\Ij_k\,,~~j=1,\ldots M;$
\item[{\rm (ii)}]
the {\it average outer fractal curvatures} of $\bd\dom$ exist and are given by
$$\tC_k^{\Frac}(\bd\dom,\dom_- ):=\lim_{\delta\rightarrow 0}\frac{1}{|\ln\delta|}\int_\delta^1 \ep^{d-k}C_k\big(\dom (\ep)\big)
\, \frac{d\ep}\ep=\sum_{j=1}^M \Ij_k,$$
in particular they agree with $C_k^{\Frac}(\bd\dom,\dom_- )$, whenever the latter exist; 
\item[{\rm (iii)}]
the {\it average outer fractal curvature measures} of $\bd\dom$
$$\tC_k^{\Frac}(\bd \dom ,\dom_-,\cdot):=\lim_{\delta\rightarrow 0}\frac{1}{|\ln\delta|}\int_\delta^1 \ep^{d-k}C_k\big(\dom (\ep),\cdot\big)
\, \frac{d\ep}{\ep},$$
exist as weak limits of signed measures, and satisfy
$$\tC_k^{\Frac}(\bd\dom ,\dom_-,\cdot)=\int_{\partial \dom }\,\1_{(\cdot)}(x)\, \mathcal{D}_{k,\bd\dom }(x)\,\H^{\Dim}(dx)\,,$$
where $\mathcal{D}_{k,\bd\dom }(x):=\sum_{j=1}^M \1_{\Fj}(x)\,\mathcal{D}^{(j)}_{k,\bd\dom }$
 for $\H^{\Dim}$-almost all $x\in\bd \dom $.
\end{enumerate}
\end{remark}
The proofs are similar as in \cite{RZ12} for the case of self-similar sets, where a special bump function was used. Here they are omitted.

\begin{figure}
 \includegraphics[width=0.6\textwidth]{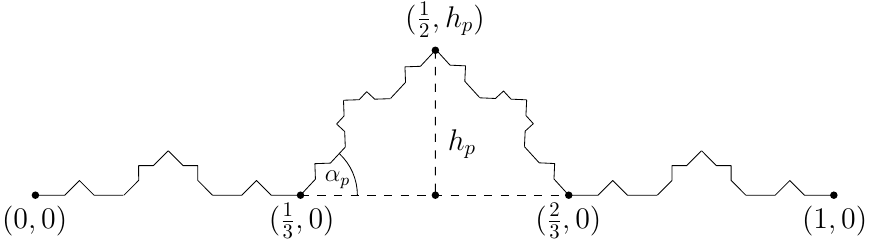}\hspace{-0.16\textwidth}
 \includegraphics[width=0.55\textwidth]{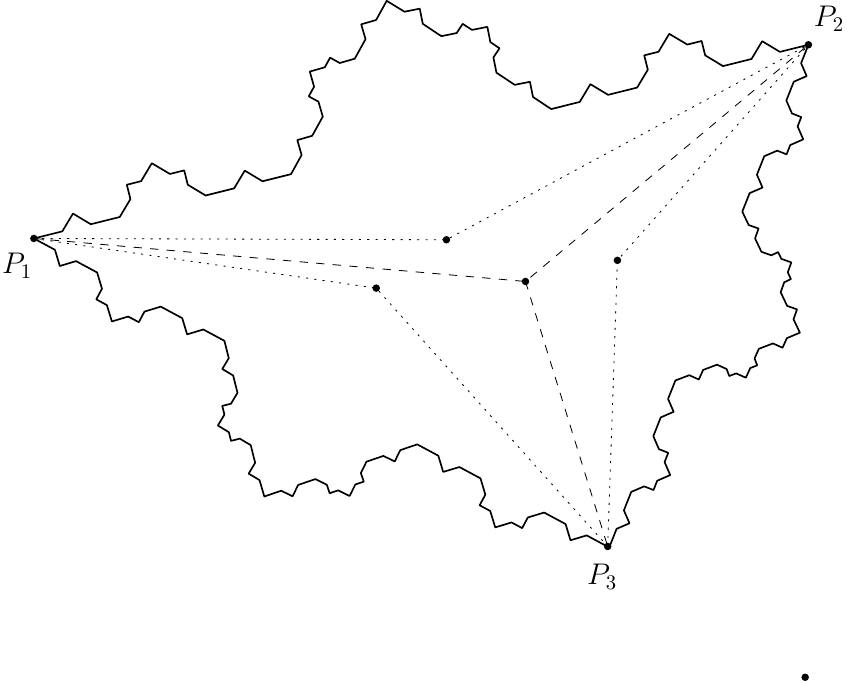}
 \caption{\label{fig:ex} Examples of domains with piecewise self-similar boundaries. \emph{Left:} A Koch-type fractal curve as constructed in Example~\ref{ex:Koch-type} (here with parameter $p=\frac 1{3\sqrt{2}}$). \emph{Right:} Domain bounded by three scaled copies of the same Koch-type curve as in Example~\ref{ex:Koch-type}.}
\end{figure}


\begin{ex} \label{ex:Koch-type}
 To illustrate our results, we discuss a class of snowflake-type domains in $\R^2$. They include the Koch snowflake depicted in Figure~\ref{fig:1} (left) and the domains in Figure~\ref{fig:ex}. For $p\in[\frac 16,\frac 13]$ let $F=F_p$ be the Koch type curve generated by an IFS of four similarity mappings $\{S_1,S_2, S_3, S_4\}$ given as follows:
 $S_1(x)=\frac 13 x$,
 $S_2(x)=pg_p(x)+(\frac 13,0)$,
 $S_3(x)=pg_p^{-1}(x)+(\frac 12,h_p)$,
 $S_4(x)=\frac 13 x+\frac 23$, $x\in\R^2$,
 where $g_p$ denotes the counterclockwise rotation by angle $\alpha_p:=\arccos(\frac 1{6p})$ and $h_p:=\sqrt{p^2-(\frac 16)^2}$. It is easy to see that this IFS satisfies OSC (the interior $O$ of the convex hull of $F_p$ is a feasible open set) and that $F_p$ is a fractal curve with endpoints $(0,0)$ and $(1,0)$. The pieces $S_iF_p$ and $S_{i+1}F_p$ intersect in a single point for $i=1,2,3$. The Minkowski dimension $d$ of $F_p$ is given by the equation $2 (1/3)^d + 2 p^d =1$. For $p=\frac 13$ the standard Koch curve is obtained, and for $p=\frac 16$, the generated set $F$ is the line segment from $(0,0)$ to $(1,0)$. Note that the dimension $d$ of $F$ increases monotonically from 1 to $\ln 4/\ln 3$ as the parameter $p$ varies from $\frac 16$ to $\frac 13$, and that the set $F$ is lattice for $p=\frac 13$, while it is non-lattice for most choices of $p$ (more precisely for all $p$ such that $\ln p/\ln(1/3)\notin\mathbb{Q}$). Observe that $F$ is contained in the closed half-space above the $x_1$-axis.\\
 In order to construct a domain with piecewise self-similar boundary, pick three points $P_1$, $P_2$ and $P_3$ in $\R^2$ in general position (i.e. not on a line).
 Fix some $p\in[\frac 16,\frac 13]$ and the generated set $F=F_p$ as above. Replace each side of the triangle $P_1P_2P_3$ by a scaled copy of $F$ with the same endpoints in such a way that the copy has nonempty intersection with the interior of the triangle. For an equilateral triangle and $p=\frac 13$, one obtains the Koch snowflake, the general situation is depicted in Figure~\ref{fig:ex} (right). The domain $\dom$ formed in this way has a piecewise self-similar boundary as in \eqref{disjoint(j)} with pieces $F^{(1)}$, $F^{(2)}$ and $F^{(3)}$ of the same Minkowski dimension $d\in(1,2)$ for any $p\in(\frac 16,\frac 13]$. Moreover, the set $\Oj:=\Int(\conv(\Fj))$, i.e.\ the interior of the convex hull of $\Fj$, is a strong feasible open sets for $\Fj$, $j=1,2,3$, such that the assumptions \eqref{eq:O-disjoint} and \eqref{eq:interior} are satisfied.
 Hence Theorem~\ref{main} is applicable directly for $k=1$, and for $k=0$, if there exists a function $f$ as required by the hypothesis. For this it suffices to check that the constant $c_{\max}$ in Proposition~\ref{riem-int-suf} is finite. This can indeed be seen as follows: by \eqref{support} and \eqref{bounded} in Proposition~\ref{bump-prop}, there is some constant $C'>0$ such that $C_k^\var(\Fj(\ep),A_{\partial \dom }(y,\ep))\leq C'\cdot C_k^\var (\Fj(\ep),B(y,a\ep))$ holds for all $y\in\Fj$ and almost all $0<\ep<1$ (namely all $\ep$ such that $(\ep,\Fj)$ is regular).
 Now, by \cite[Lemma 3.3]{WZ13}, there is some constant $C''$ such that
 $$
 C_k^\var (\Fj(\ep),B(y,a\ep))\leq C'' \ep^k \#\widehat\Sigma^{(j)}(B(y,a\ep),\ep),
 $$
 where, for any set $B\subset\R^n$ and some fixed $R>\sqrt{2}|\Fj|$,
 $$
 \widehat\Sigma^{(j)}(B,\ep):=\left\{\s\in\Sigma^{(j)}_*: R r_\s\leq\ep<R r_{\s||\s|-1} \text{ and } (\Sj_\s \Fj)(\ep)\cap B\neq\emptyset\right\}.
 $$
 (That the assumptions of \cite[Lemma 3.3]{WZ13} are satisfied is shown for the Koch curve in \cite[Example 5.3]{Wi11} and the same argument carries over to all the curves considered here.)
 By a standard argument related to OSC, one can see that, for fixed $\ep$, the cardinality of the families $\widehat\Sigma^{(j)}(B(y,a\ep),\ep)$ is bounded uniformly in $y$, and by a scaling argument, that the bound is the same for all $\ep$. Since this is true for any $j$, this allows to conclude the existence of the constant $c_{\max}$ in Proposition~\ref{riem-int-suf}.
\end{ex}
\section{Inner and outer fractal curvatures via tilings}\label{sec:rel-fc}

Formulas for fractal curvatures of a domain with piecewise self-similar boundary may be also be obtained from 'generator formulas' as discussed in \cite{Wi15}. They are based on a tiling constructed on some feasible open set from OSC. As a preparation, we recall the tiling construction for a single self-similar set $F\subset\R^n$ generated by the IFS $\{S_1,\ldots,S_N\}$ and satisfying OSC. We discuss a mild generalization of a result from \cite{Wi15} regarding the existence of relative fractal curvatures. 
Later we apply this result to the situation of a domain $\dom $ with piecewise self-similar boundary as described above. We use the notation from Section~\ref{fractal bound} but omit the upper index $j$ as long as we are discussing a single self-similar set. Beside OSC we assume as before that $F\subset\R^n$ has Minkowski dimension $d$ strictly less than $n$.

Let $O$ be some feasible open set for $F$. We are interested in the limits
\begin{align*}
 C_k^{\Frac}(F,O)&:=\elim_{\ep\searrow 0} \ep^{\Dim-k} C_k(F(\ep),O) \quad\text{ and } \\ 
 \tC_k^{\Frac}(F,O)&:=\lim_{\delta \searrow 0}\frac{1}{|\ln\delta|} \int_\delta^1
\eps^{\Dim-k} C_k(F(\ep),O) \frac{d\eps}{\eps},
\end{align*}
which we call the \emph{$k$-th fractal curvature} and \emph{average $k$-th fractal curvature of $F$ relative to $O$}, respectively. In order to obtain a generator formula for $C_k^{\Frac}(F,O)$ or $\tC_k^{\Frac}(F,O)$, we require a feasible set $O$ such that the following \emph{projection condition} is satisfied (which we recall from \cite{Wi15}):
\begin{align}\label{eqn:PC}
S_i O\subset \cl{\pi^{-1}_F(S_i F)}, \quad \text{ for } i=1,\ldots,N.
\end{align}
Here $\pi_F$ denotes the metric projection onto the set $F$. Note that $\pi_F$ is defined on the set $\Unp(F)$ of points $x\in\R^{n}$ which have a unique nearest point in $F$. The projection condition has the following important implication, which we recall from \cite[Lemma 3.19(i)]{Wi15}:
\begin{lemma}
 \label{lem:PC} 
If the projection condition \eqref{eqn:PC} holds, then, for each $\ep>0$ and $i=1,\ldots, N$,
\begin{align}
 \label{eqn:PC2} F(\ep)\cap S_i O=(S_i F)(\ep)\cap S_i O.
\end{align}
\end{lemma}

Let
\begin{align*}
 G:=O\setminus\bigcup_{i=1}^N \cl{S_i O} \quad \text{ and } \quad \Gamma:= O\setminus\bigcup_{i=1}^N S_i O.
\end{align*}
Note that $G$ is an open set but $\Gamma$ is not closed in general. It satisfies $G\subset\Gamma\subset G\cup\bigcup_{i=1}^N \bd S_iO$. Moreover, $G=\Int(\Gamma)$. (Indeed, if $x\in\Int(\Gamma)\setminus G$, then on the one hand there is some $r>0$ such that the ball $B(x,r)$ is contained in $\Int(\Gamma)$. On the other hand, $x\notin G$ implies $x\in\bd S_{m,\eta(t)} O$ for some $j\in\{1,\ldots,N\}$. Hence there exists a point $y\in S_{m,\eta(t)} O$ with $d(x,y)<r$, i.e., $y\in B(x,r)$, a contradiction.) Consider the set family
\begin{align*}
 \sT(O):=\left\{ S_\sigma G:\sigma\in \Sigma_*\right\}
\end{align*}
of all copies of $G$ generated by applying finitely many maps $S_i$ of the IFS. $\sT(O)$ is a \emph{tiling} of the set $O$ in the following sense: the \emph{tiles} $S_\sigma G$ are pairwise disjoint and their closed union covers $O$, that is,
\begin{align*}
 R\cap R'=\emptyset \text{ for all } R,R'\in\sT(O) \text{ with } R\neq R' \text{ and } \overline{O}=\cl{\bigcup_{R\in\sT(O)} R},
\end{align*}
see \cite[Theorem 5.7]{PeW12}. $G$ is called the \emph{generator} of the tiling $\sT(O)$.

We consider the \emph{relative inradius} of $G$ w.r.t.\ $F$, given by
\begin{align} \label{eq:rel-inr}
 \tilde{g}:=\sup\{d(x,F):x\in G\}.
\end{align}
Note that $\tilde{g}$ is equivalently given by $\sup\{d(x,F):x\in O\}$, see \cite[p.~303]{Wi15} for a proof. 


The following statement is a generalization of \cite[Proposition 4.9]{Wi15} in that we do not assume $O$ to be a strong feasible set here. This is an important extension in view of the later application to domains $\dom$, for which the feasible set $O$ should be chosen as a subset of $\dom$ or of the exterior $\dom_-$ which precludes $O$ to be strong. 

\begin{theorem}{} \label{thm:rel-curv-O}
Let $F$ be a self-similar set in $\R^n$ satisfying OSC with similarity dimension $\Dim <n$ and let $k\in\{0,\ldots,n-1\}$.
Let $O$ be a feasible open set for $F$ such that the projection condition \eqref{eqn:PC} and the following three conditions (i)-(iii) are satisfied:
\begin{enumerate}[(i)]
\item $k=n-1$ or almost all $\ep\in(0,\tilde g)$ are regular for $F$; 
\item $C_k^\var(F(\ep),\bd \Gamma)=0$ for almost all $\ep\in(0,\tilde g)$;
\item there are positive constants $c,\gamma$ such that for almost all $\ep\in(0,\tilde g)$
\begin{equation}\label{eqn:C_k-G-O}
C_k^{\var}(F(\ep),\Gamma)\le c \ep^{k-\Dim+\gamma}.
 \end{equation}
\end{enumerate}
Then $\ep^{\Dim-k}C_k^{\var}(F(\ep), O)$ is bounded for $\ep\in(0,\tilde g)$. Moreover, the average $k$-th fractal curvature of $F$ relative to $O$
 exists and is given by
\begin{align}\label{eqn:Yk-Gamma}
\asC_k^{\Dim}(F,O) =\frac{1}{\eta} \int_0^{\tilde{g}} \ep^{\Dim-k-1} C_k(F(\ep),\Gamma)\ d\ep.
\end{align}
If $F$ is nonlattice, then also the $k$-th fractal curvature $\sC_k^{\Dim}(F,O)$ 
of $F$ relative to $O$ exists and equals the expression in \eqref{eqn:Yk-Gamma}.
\end{theorem}

 \begin{proof}
Fix some $k\in\{0,\ldots,n-2\}$. (The slightly easier case $k=n-1$ is discussed separately at the end.) Recall from \eqref{regular distances} the set $\reg(F)$ of regular values of $F$ and let $\reg'(F)$ be the set of regular values such that $C^{\var}_k(F(\ep),\bd \Gamma)=0$. 
Set $\sN:=\reg'(F)^c$ and observe that, by conditions (i) and (ii), $\sN$ is a Lebesgue null set. Let $\sN^*:=\bigcup_{\sigma\in \W_*} r_\sigma \sN$ and
note that $\sN^*$ is also a null set. Moreover, if $\ep\in\sN^*$ and $\sigma\in \W_*$, then $r_\sigma\eps\in\sN^*$. (Indeed, $\ep\in\sN^*$ means there are $\tau\in \W_*$ and $\ep'\in\sN$ such that $\ep=r_\tau\ep'$. But then $r_\sigma\eps=r_\sigma r_\tau\ep'=r_{\sigma\tau}\ep'$ with $\sigma\tau \in \W_*$ and thus $r_\sigma\eps\in\sN^*$.) Set $\reg^*:=(0,\infty)\setminus\sN^*$ and note that the observed invariance property for $\sN^*$ implies $r_\sigma^{-1}\reg^*\subset\reg^*$ for any $\sigma\in \W_*$, i.e., $\ep\in\reg^*$ implies $r_\sigma^{-1}\ep\in\reg^*$.
(Assume not. Then $r_\sigma^{-1}\ep\in\sN^*$ and thus $r_\sigma r_\sigma^{-1}\ep=\ep\in\sN^*$, a contradiction.)

For any $\ep\in \reg^*$, we have, by Lemma~\ref{lem:PC} and the locality and scaling properties of the curvature measures,
\begin{align*}
 C_k(F(\ep),O)&=\sum_{i=1}^N C_k(F(\ep),S_iO)+ C_k(F(\ep),\Gamma)\\
 &=\sum_{i=1}^N C_k((S_iF)(\ep),S_iO)+ C_k(F(\ep),\Gamma)\\
 &=\sum_{i=1}^N r_i^k C_k(F(\ep/r_i),O)+ C_k(F(\ep),\Gamma).
\end{align*}
Indeed, $\ep\in\reg^*$ implies $\ep/r_i\in\reg^*$ such that $C_k(F(\ep/r_i),\cdot)$ and thus also $C_k((S_iF)(\ep),\cdot)$ are well-defined.
Set
\begin{align*}
 Z(t):=\tilde{g}^{\Dim-k}e^{(\Dim-k)t}C_k(F(\tilde g e^{-t}),O) \quad \text{ and } \quad z(t):=\tilde{g}^{\Dim-k}e^{(\Dim-k)t}C_k(F(\tilde g e^{-t}),\Gamma)
\end{align*}
for each $t\in T^*:=\{t\in\R: \tilde g e^{-t}\in\reg^*\}$.
Note that both functions vanish for all $t\in T^*\cap(-\infty,0]$, since this implies $\tilde g e^{-t}\geq\tilde g$ such that $O$ (and thus $\Gamma$) are contained in the interior of $F(\tilde g e^{-t})$, where the curvature measures have no mass.
Substituting $\ep=\tilde g e^{-t}$ and multiplying $\eps^{\Dim-k}=\tilde{g}^{\Dim-k}e^{(\Dim-k)t}$ in the above equation, we obtain
\begin{align} \label{eq:renewalZ}
 Z(t)=\sum_{i=1}^N r_i^{\Dim} Z(t-|\ln r_i|) + z(t)
\end{align}
for all $t\in T^*$.
In terms of the discrete measure $\tilde\nu_0:=\sum_{i=1}^{N}\1_{(\cdot)}(|\ln r_i|)\, r_i^{\Dim}$ this renewal equation can be rewritten as
\begin{align*}
 Z(t)=\int_0^t Z(t-s) \tilde\nu_0(ds) + z(t), \quad t\in T^*.
\end{align*}
Denote by $\tilde\nu_m$ is the $m$-{\it th convolution power} of $\tilde\nu_0$ and
the associated {\it renewal function} by $U(s):=\sum_{m=0}^\infty\tilde\nu_m((0,s])$.
Iterating the above equation, we get
\begin{align*}
Z(t)=&\int_0^t z(t-s)\, dU(s)
\end{align*}
for any $t\in T^*$. In order to apply the Renewal theorem, we introduce two auxiliary functions on $\R$ by
\begin{align*}
\bz(t)&:=\begin{cases}
 z(t), &{\rm if}~t\in T^*\\
 \limsup\limits_{\stackrel{t'\to t}{t'\in T^*}} z_k(t'), &{\rm if}~t\notin T^*,
\end{cases}\\
\bZ(t)&:=\begin{cases}\int_0^t\bz(t-s)dU(s), & t>0,\\
0 & t\leq 0.
\end{cases}
\end{align*}
First note that $\bz$ vanishes on $(-\infty,0)$. Moreover, by the last equation, we have
\begin{align*}
\bZ(t)=Z(t) \quad \text{ for all } t\in T^*.
\end{align*}
Condition (iii) implies there are positive constants $c,\gamma$ such that
\begin{align}\label{eq:z-estimate}
 |\bz(t)|\leq c\tilde{g}^\gamma e^{-\gamma t}, \quad t>0.
\end{align} Hence $|\bz|$ is bounded by a directly Riemann integrable function. Moreover, the weak convergence \eqref{reg-convergence} implies $C_k(F(\ep'),\Gamma)\to C_k(F(\ep),\Gamma)$ as $\ep'\to\ep$ for any $\ep\in\reg^*$, since, by assumption (ii), $\Gamma$ is a continuity set of the measure $C_k(F(\ep),\cdot)$. This implies the continuity of $z$ and thus of $\bz$ at each $t\in T^*$. Hence $\bz$ is continuous almost everywhere and, by \cite[Prop. 4.1, p.118]{As87}, we can conclude that $\bz$ itself is directly Riemann integrable. Therefore, by the Renewal theorem in Feller \cite[p.363]{Fel71}, we get in the nonlattice case
$$\lim_{t\rightarrow\infty}\limits\bZ(t)=\frac 1\eta\int_0^\infty\bz(t)\, dt.$$
Since $\bz(t)=z(t)$ for Lebesgue-a.a. $t$, the right hand side agrees with
\begin{align} \label{eq:limit-expr-z}
 \frac{1}{\eta}\int_0^\infty z(t)\,dt= \frac 1\eta \int_0^{\tilde g} \eps^{\Dim-k} C_k(F(\ep),\Gamma) \frac{d\ep}{\ep}.
\end{align}
Now the existence of the essential limit $\sC_k^{\Dim}(F,O)$ and the expression stated for it in \eqref{eqn:Yk-Gamma} are transparent.
If $F$ is lattice with lattice constant $h$, then Feller's Renewal theorem states that
for Lebesgue-a.a. $u\in[0,h)$
\begin{align}
\label{eq:discrete-limits} \lim_{n\rightarrow\infty}\bZ(u+nh)=\frac{h}{\eta}\sum_{m=1}^\infty\bz(u+mh).
\end{align}
In this relation we can replace $\bZ$ by $Z$ and $\bz$ by $z$ for Lebesgue a.a. $u\in[0,h)$. (More precisely, we can do this for all $u\in[0,h)$ such that $u+kh\in T^*$ for all $k\in\N_0$, which is a subset of $[0,h)$ of full Lebesgue measure.)
Finally, the existence of the limits in \eqref{eq:discrete-limits} implies that also the average limit
\begin{align*}
 \lim_{T\to\infty}\frac 1T \int_0^T \bZ(t) dt= \lim_{T\to\infty}\frac 1T \int_0^T Z(t) dt=\asC_k(F,O)
\end{align*}
exists and equals the expression in \eqref{eq:limit-expr-z},
see \cite[p.1959]{Ga00} or \cite[Lemma 3.2]{LW06}.

In the case $k=n-1$ 
we have $C_{n-1}(F(\ep), \cdot)=\frac 12 \H^{n-1}(\bd F(\ep)\cap \cdot)$,
which is well-defined for all $\ep>0$ (and not only for regular values of $F$, which is why the regularity condition in (i) is not needed in this case). Hence the functions $z$ and $Z$ are defined for all $t$ and the renewal equation \eqref{eq:renewalZ} can be derived in the same way as above but for all $t$. (Here the auxiliary functions $\bz$ and $\bZ$ are not needed.) 
The continuity of $z$ almost everywhere follows now in the same way as before from condition (ii).
Since condition (iii) implies the estimate \eqref{eq:z-estimate} (for $z$ instead of $\bz$), we can conclude the direct Riemann integrability of $z$. Now all the assertions follow as before from Feller's Renewal theorem.
\end{proof}

\begin{remark} \label{rem:thm10}
 (i) In formula \eqref{eqn:Yk-Gamma} the relative inradius $\tilde g$ appears as the endpoint of the domain of integration. It may be replaced with any larger constant without changing the value of the integral, since $\bd (F(\ep))\cap\Gamma=\emptyset$ for $\ep>\tilde g$ and thus $C_k(F(\ep),\Gamma)=0$.

 (ii) Since $\Gamma=G \cup(\bd\Gamma\cap\Gamma)$, condition (ii) in Theorem~\ref{thm:rel-curv-O} implies that
 $C_k(F(\ep),\Gamma)=C_k(F(\ep),G)$ holds for a.a.\ $\ep>0$. Hence in formula \eqref{eqn:Yk-Gamma} the set $\Gamma$ may be replaced by $G$ (and similarly in condition (iii)). 
%
\end{remark}

 \begin{remark}
 For the case $k=n$, i.e.\ for the measures $C_n(F(\ep),\cdot)=\lambda^n(F(\ep)\cap\cdot)$ a statement analogous to Theorem~\ref{thm:rel-curv-O} holds. In this case the assumptions (i) and (ii) can be omitted and (iii) simplifies to the existence of constants $c,\gamma>0$ such that, for all $\ep\in(0,\tilde g)$,
\begin{equation}\label{eqn:C_n-G-O}
\lambda_n(F(\ep)\cap\Gamma)\le c \ep^{n-\Dim+\gamma}.
\end{equation}
Moreover, in the conclusion a slightly different formula is obtained for the (average) relative Minkowski content of $F$ relative to $O$:
\begin{align}\label{eqn:Yn}
\asM^{\Dim}(F,O) =\frac{1}{\eta} \int_0^{\infty} \ep^{\Dim-k-1} \lambda_n(F(\ep)\cap\Gamma)\ d\ep.
\end{align}
Note that, in contrast to \eqref{eqn:Yk-Gamma}, in this formula the integration extends over the whole positive axis, which is related to the fact that the volume does not vanish for $\eps>\tilde g$. Since $\lambda_n(F(\ep)\cap\Gamma)=\lambda_n(\Gamma)$ for $\ep\geq\tilde g$, \eqref{eqn:Yn} may also be written as
\begin{align*}
\asM^{\Dim}(F,O) =\frac{1}{\eta} \int_0^{\tilde g} \ep^{\Dim-k-1} \lambda_n(F(\ep)\cap\Gamma)\ d\ep + \frac{\lambda_n(\Gamma)}{\eta}\frac{\tilde g^{\Dim-n}}{n-\Dim}.
\end{align*}
The statement for $k=n$ may be proved following the lines of the proof of \cite[Proposition 3.18]{Wi15}. We do not assume here that the set $O$ is strong, but instead the required estimates for the function $\varphi$ are now derived from the assumption \eqref{eqn:C_n-G-O} and the projection condition \eqref{eqn:PC}, which, by Lemma~\ref{lem:PC}, implies $(S_iF)(\ep)\cap S_iO=F(\ep)\cap S_iO$. Finally, to obtain the formula \eqref{eqn:Yn}, one needs to combine this with the computation in the proof of \cite[Theorem 3.17, see p.306/307 starting from eq.~(3.39)]{Wi15}, which carries over to the present situation.
 \end{remark}

Now we will employ Theorem~\ref{thm:rel-curv-O}, to compute inner and outer fractal curvatures for domains $\dom \subset\R^n$ with piecewise self-similar boundary as defined in \eqref{disjoint(j)}. 
We will concentrate on the approximation of $\dom $ from inside and derive formulas for inner fractal curvatures.
For the outer fractal curvatures 
corresponding results can be derived, see Remark~\ref{rem:outer-approx} at the end.
Instead of working with a strong feasible open set for each $\Fj$ as in the previous section, we will choose feasible open sets $\Uj$ which are adapted to the domain $\dom $. More precisely, we assume that for each $j=1,\ldots,M$ there is a nonempty feasible open set $\Uj\subset\dom$ satisfying 
the following conditions: 
\begin{align}
\label{eq:Uj-disjoint} \Uj\cap U^{(\ell)}&=\emptyset, \quad \text{ for } j\neq\ell,\\
\label{eqn:PCj}
 \Sj_i \Uj&\subset \cl{\pi^{-1}_{\Fj}(\Sj_i \Fj)}, \quad \text{ for } i=1,\ldots,\Nj,\\
 \label{eq:Uj-in-D} \Uj&\subset \cl{\pi^{-1}_{\bd \dom } (\Fj)}.
\end{align}
Moreover, we assume that the sets $\Uj$ cover the essential part of $\dom $ in the following sense: there are positive constants $\ep_0,c,\gamma$ such that
 for almost all $0<\ep< \ep_0$
\begin{equation}\label{eq:D-in-Uj}
C_k^{\var}(\bd \dom (\ep),X)\le c \ep^{k-\Dim+\gamma},
 \end{equation}
 where $X:=\dom \setminus \bigcup_{m,\eta(t)} \Uj$.

Condition \eqref{eq:Uj-disjoint} is the same as \eqref{eq:O-disjoint} and the assumption $\Uj\subset\dom$ should be compared with \eqref{eq:interior}. Condition \eqref{eqn:PCj} is just the projection condition \eqref{eqn:PC} from above, which we require here for each of the sets $\Uj$, and \eqref{eq:Uj-in-D} is similar to the projection condition ensuring that the points of $\Uj$ are projected to the right piece $\Fj$ of $\bd \dom $ under the metric projection onto $\bd \dom $.

The assumptions are not very restrictive from the point of view of the domain $\dom $. We just use our freedom to choose suitable open sets $\Uj$. In particular, in the examples depicted in Figures~\ref{fig:1} and~\ref{fig:ex} the sets $\Uj$ can be chosen in such a way that all the assumptions \eqref{eq:Uj-disjoint}--\eqref{eq:D-in-Uj} are satisfied. For a certain subclass of the sets in Example~\ref{ex:Koch-type} we discuss this in detail in Example~\ref{ex:Koch-type2} below.
However, all four conditions \eqref{eq:Uj-disjoint}--\eqref{eq:D-in-Uj} might not always be satisfiable together, see Remark~\ref{rem:chooseVj} for a further discussion of this.


\begin{theorem} \label{thm:domain-D}
Let $\dom \subset\R^n$ be a domain with piecewise self-similar boundary as in \eqref{disjoint(j)}, composed of self-similar sets $\Fj$, $j=1,\ldots,M$ satisfying OSC. Suppose that for the $\Fj$ feasible open sets $\Uj\subset\dom$ exist such that the conditions 
\eqref{eq:Uj-disjoint}--\eqref{eq:D-in-Uj} hold.
Let $k\in\{0,\ldots, n-1\}$.
 Assume that for each $j$ conditions (i)-(iii) of Theorem~\ref{thm:rel-curv-O} are satisfied (for the set $\Uj$ and the generated tiling). If $k\leq n-2$, assume additionally that almost all $\ep\in(0,g)$ are regular for $\bd \dom $, where $g$ denotes the inradius of $\dom $.
Then the inner average $k$-th fractal curvature $\tC_k^{\Frac}(\bd \dom ,\dom )$ of $\bd \dom $
 exists and is given by
 \begin{align} \label{eq:av-rel-frac-curv-relation}
 \tC_k^{\Frac}(\bd \dom ,\dom )=\sum_{j=1}^M \tC_k^{\Frac}(\Fj,\Uj).
 \end{align}
If $\dom $ is completely non-lattice, then also the inner $k$-th fractal curvature $C_k^{\Frac}(\bd \dom ,\dom )$ of $\bd \dom $ 
exists and is given by
\begin{align} \label{eq:rel-frac-curv-relation}
 C_k^{\Frac}(\bd \dom ,\dom )=\sum_{j=1}^M C_k^{\Frac}(\Fj,\Uj).
 \end{align}
\end{theorem}
\begin{proof}
 First observe that the assumptions on the sets $\Fj$ and $\Uj$ ensure that, by Theorem~\ref{thm:rel-curv-O} above, $\tC_k^{\Frac}(\Fj,\Uj)$ exists in general and $C_k^{\Frac}(\Fj,\Uj)$ exists in case $\Fj$ is non-lattice. Moreover, we can decompose the curvature measure of $\bd \dom(\ep)$ inside $\dom $ as follows for any $\ep\in(0,g)$ that is regular for $\bd \dom $ and all the $\Fj$. By \eqref{eq:Uj-disjoint} and the definition of $X$, we have for all such $\ep$
 \begin{align}
 \notag C_k((\bd \dom )(\ep),\dom )&=\sum_{j=1}^M C_k((\bd \dom )(\ep),\Uj)+ C_k((\bd \dom )(\ep),X)\\ \label{eq:decomp-cm}
 &=\sum_{j=1}^M C_k(\Fj(\ep),\Uj)+ C_k((\bd \dom )(\ep),X).
 \end{align}
 For the last equality note that \eqref{eq:Uj-in-D} implies
 $$
 (\bd \dom )(\ep)\cap\Uj=\Fj(\ep)\cap\Uj
 $$
 for such $\ep$ and all $j$, and therefore, by \eqref{locality}, $C_k((\bd \dom )(\ep),\Uj)=C_k(\Fj(\ep),\Uj)$. (For a proof of the above set equality note that one inclusion is immediate from $\Fj\subset\bd \dom $. For the other inclusion let $x\in(\bd \dom )(\ep)\cap\Uj$. Then, by \eqref{eq:Uj-in-D}, $x\in\cl{\pi^{-1}_{\bd \dom } (\Fj)}$, implying the existence of a sequence $(x_n)_{n\in\N}$ with $ x_n\to x$ as $n\to\infty$ and $x_n\in\pi^{-1}_{\bd \dom } (\Fj)$ for each $n$. The latter condition implies $d(x_n,\Fj)\leq d(x_n,\bd \dom \setminus \Fj)$ for each $n$ which is preserved in the limit, i.e.\ for $x$, since $\Fj$ is closed. Thus, we conclude the existence of some point $z\in\Fj$ such that
 $d(x,\Fj)=|x-z|=d(x,\bd \dom )\leq\ep$ which yields $x\in \Fj(\ep)$.)
 If the generating system for $\bd \dom $ is completely non-lattice, then we can multiply $\ep^{\Dim-k}$ in equation \eqref{eq:decomp-cm} and take the essential limit as $\ep\searrow 0$. Then, by Theorem~\ref{thm:rel-curv-O}, all the summands on the right of equation \eqref{eq:decomp-cm} converge (the $j$-th to $C^{\Frac}_k(\Fj,\Uj)$) except the last one, which vanishes due to \eqref{eq:D-in-Uj}. Indeed, the absolute value of $C_k((\bd \dom )(\ep),X)$ is bounded from above by $C_k^{\var}((\bd \dom )(\ep),X)$ and thus, by \eqref{eq:D-in-Uj}, $\ep^{\Dim-k}|C_k((\bd \dom )(\ep),X)|\le c \ep^\gamma$ which vanishes as $\ep\searrow 0$. This shows the existence of $C_k^{\Frac}(\bd \dom , \dom )$ and the relation \eqref{eq:rel-frac-curv-relation} in the completely non-lattice case.
 In the general case one can similarly use \eqref{eq:decomp-cm} to decompose the integral $\int_\delta^g
\eps^{\Dim-k} C_k((\bd \dom )(\ep), \dom ) \frac{d\eps}{\eps}$ into
\begin{align*}
 \sum_{j=1}^M \int_\delta^g
\eps^{\Dim-k} C_k(\Fj(\ep), \Uj) \frac{d\eps}{\eps} + \int_\delta^g
\eps^{\Dim-k} C_k((\bd \dom )(\ep), X) \frac{d\eps}{\eps}.
\end{align*}
 Now, again by \eqref{eq:D-in-Uj}, the absolute value of the last integral can be bounded by some constant. Thus, multiplying $|\ln \delta|^{-1}$ and taking the limit as $\delta\searrow 0$, the last summand vanishes.
 Doing the same with the other integrals yields the average limits $\tC^{\Frac}_k(\Fj,\Uj)$, which exist by Theorem~\ref{thm:rel-curv-O}. Note here that the relative inradii $\widetilde g^{(j)}$ of the sets $\Uj$ appearing in the formulas in Theorem~~\ref{thm:rel-curv-O} satisfy $\widetilde g^{(j)}\leq g$ for each $j$ and thus, by Remark~\ref{rem:thm10}(i), we can replace the upper bound $g$ of the $j$-th integral by $\widetilde g^{(j)}$ without changing its value.
 This yields the existence of the average limit $\tC^{\Frac}_k(\bd \dom ,\dom )$
 and the formula stated in \eqref{eq:av-rel-frac-curv-relation}.
 Finally observe that for the for case $k=n-1$ the additional regularity assumption on $\bd \dom $ is not needed and all of the above arguments work without this.
 \end{proof}


\begin{ex} (Example~\ref{ex:Koch-type} continued) \label{ex:Koch-type2}
Let $\dom$ be a snowflake-type domain as introduced in Example~\ref{ex:Koch-type} generated by placing a copy $\Fj$ of the same Koch-type curve $F_p$ on each side of a triangle with vertices $P_1,P_2$ and $P_3$. For simplicity let us assume now that the triangle is equilateral. (In the general case the considerations below are more involved.) For each $\Fj$, $j=1,2,3$ we choose $\Uj$ as follows. The three segments $P_iS$, $i=1,2,3$ from $P_i$ to the centre $S$ of the triangle divide $\dom$ into three open parts, see Figure~\ref{fig:1} (left). Let $\Uj$ be the one closest to $\Fj$. Note that $\Uj=\{x\in\dom: d(x,\Fj)<d(x,\bd\dom\setminus \Fj)\}$, from which it is clear that these sets satisfy $\Uj\subset\dom$ as well as conditions \eqref{eq:Uj-disjoint} and \eqref{eq:Uj-in-D}. Moreover, by looking at the images $S_i^{(j)}\Uj$, it is easy to verify that $\Uj$ is a feasible set for $\Fj$ and that the projection condition \eqref{eqn:PCj} holds. Finally, the set $X=\dom \setminus \bigcup_{m,\eta(t)} \Uj$ consists of the three segments $P_iS$ excluding the points $P_i$. One can show that, for any $\eps\in(0,g)$, where $g$ is the inradius of $\dom$, the intersection $\bd(\bd\dom(\ep))\cap X$ consists of three points. Since, for any $K\subset\R^2$ and almost every $\eps>0$, $C_0^\var(K(\ep),\{x\})\leq \frac 12$ for any point $x$, condition \eqref{eq:D-in-Uj} is satisfied. Hence, Theorem~\ref{thm:domain-D} applies and the fractal curvatures of $\bd\dom$ may be determined by computing the relative fractal curvatures $C_k^{\Frac}(\Fj,\Uj)$.
\end{ex}

\begin{remark} \label{rem:chooseVj}
It is a natural question to ask whether for any domain $\dom \subset\R^n$ with piecewise self-similar boundary as in \eqref{disjoint(j)} it is possible to choose feasible open sets $\Uj$ such that all the conditions 
\eqref{eq:Uj-disjoint}--\eqref{eq:D-in-Uj} hold. It turns out that the setting is too general and that there are simple situations in which some of the conditions are not satisfiable.
For instance, if one of the generating IFS contains a map $\Sj_i$ such that $\Sj_i\dom \cap \dom =\emptyset$, then condition \eqref{eq:Uj-in-D} is not satisfiable. (Such a situation occurs e.g.\ if the Koch snowflake is generated by IFSs consisting of two maps.) The generating maps have to preserve (at least locally) the side of the boundary, i.e.\ images $\Sj_i(x)$ of points $x\in \dom $ close to $\Fj$ should also be inside $\dom $. It seems that most of the time this can be ensured by choosing suitable IFS's generating the $\Fj$.

Similarly, if the self-similar sets $\Fj$ generating $\bd \dom $ have strong overlaps, then it will be impossible to find feasible open sets $\Uj$ such that \eqref{eq:Uj-disjoint} and \eqref{eq:D-in-Uj} hold at the same time. An example is provided by a domain generated by six Koch curves arranged as indicated in Figure~\ref{fig:2} (right). Some of them coincide in a whole first level copy.

Therefore some additional restrictions are necessary to ensure the existence of suitable sets $\Uj$. However, in many cases (which avoid the two situations described above) the following considerations help to find suitable open sets $\Uj$: From \cite[Proposition~3.16]{Wi15} it is known that for any self-similar set $\Fj$ satisfying OSC there is a strong feasible open set $V_c^{(j)}$ satisfying the projection condition \eqref{eqn:PCj}. In \cite{BHR05} it was called the \emph{central open set}. Moreover, a feasible set $\Uj$ satisfies \eqref{eqn:PCj} if and only if it is a subset of $V_c^{(j)}$, cf.\ \cite[p.302]{Wi15}. Therefore, any candidate for $\Uj$ satisfying conditions \eqref{eqn:PCj} and \eqref{eq:Uj-in-D} has to be a subset of the set $\Vj:=V_c^{(j)}\cap \dom \cap \Int\left(\cl{{\pi^{-1}_{\bd \dom } (\Fj)}}\right)$. In many situations $\Vj$ itself is a feasible set satisfying also \eqref{eq:Uj-disjoint} and \eqref{eq:D-in-Uj}. The Koch-snowflake, depicted in Figure~\ref{fig:1} (left), is such a case. More generally, if $\dom $ is such a domain in $\R^2$ and the pairwise intersections of the generating self-similar sets $\Fj$ contain at most one point, then condition \eqref{eq:Uj-disjoint} will be satisfied for $V^{(1)},\ldots, V^{(N)}$ (and hence also for any collection of sets $U^{(1)},\ldots,U^{(N)}$ such that $\Uj\subseteq\Vj$ for each $j$). However, in general the sets $\Vj$ itself will not be feasible. It is easy to see that these sets satisfy the disjointness condition $\Sj_i\Vj\cap \Sj_k\Vj=\emptyset$ for $i\neq k$, which they inherit from the sets $\Vj_c$. But the inclusion $\Sj_i\Vj\subseteq\Vj$ is not true in general. 
It is not clear at the moment how to find good feasible subsets in general. Finally, we point out that the set $X=\dom \setminus\bigcup_{m,\eta(t)} \Vj$ may be large, as in the example in Figure~\ref{fig:2} (left), but this does not necessarily preclude \eqref{eq:D-in-Uj} to hold.
 \end{remark}

\begin{figure}
 \includegraphics[width=0.55\textwidth]{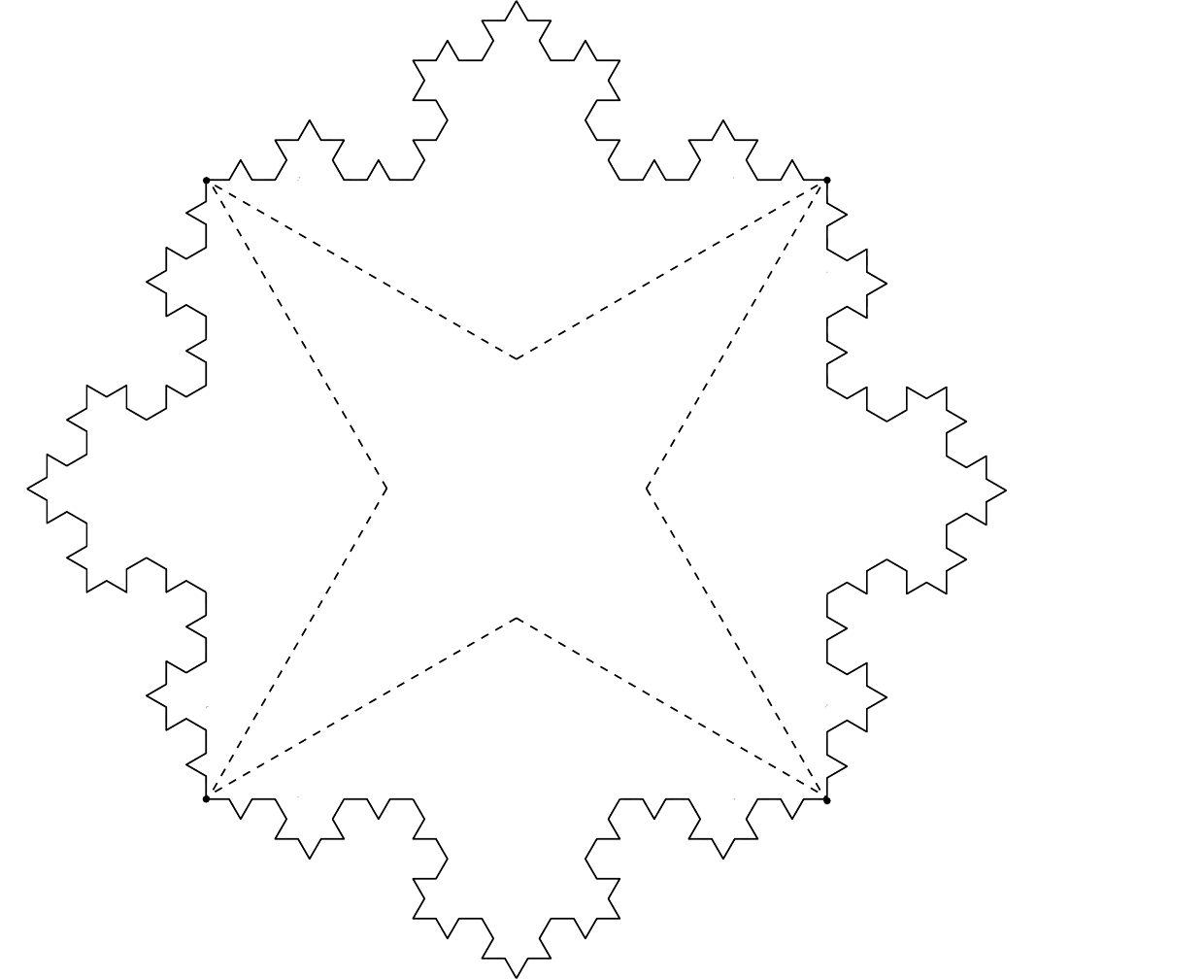}\hfill
 \includegraphics[width=0.45\textwidth]{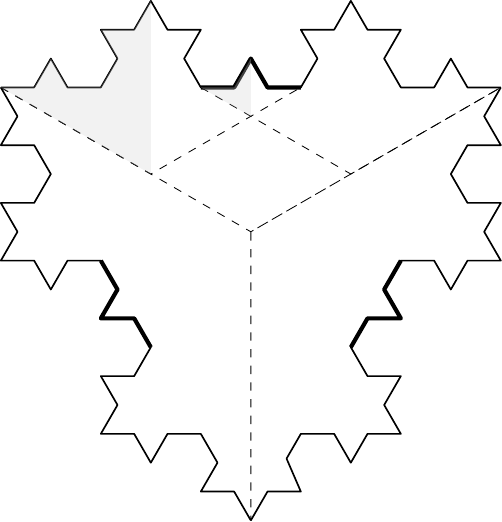}
 \caption{\label{fig:2} Further examples of domains with piecewise self-similar boundaries. \emph{Left:} Domain $\dom$ for which a large set $X$ 
 cannot be avoided. The feasible sets $\Uj$ chosen are the maximal subsets of $\dom$ satisfying \eqref{eqn:PCj}. They also satisfy conditions \eqref{eq:Uj-disjoint}, \eqref{eq:Uj-in-D} and even \eqref{eq:D-in-Uj}, despite the fact that $X$ is rather large. \emph{Right:} Domain bounded by six Koch curves, some of which overlap in a whole first level copy (overlap indicated by the bold line).}
\end{figure}

\begin{remark}
 \label{rem:outer-approx} A statement completely analogous to Theorem~\ref{thm:domain-D} holds for the outer approximation of $\dom$, leading to outer fractal curvatures of $\bd \dom$ as defined in \eqref{eq:frac-curv-def2}. All one has to do is to replace each instance of $\dom=\dom_+$ by the complementary domain $\dom_-:=\overline{\dom}^c$. In this case all feasible open sets $\Uj$ have to be constructed in $\dom_-$.
\end{remark}

%
\bibliographystyle{plain}
\bibliography{biblio}

\begin{thebibliography}{10}

\bibitem{ADT}
Eric Akkermans, Gerald~V Dunne, and Alexander Teplyaev.
\newblock Physical consequences of complex dimensions of fractals.
\newblock {\em Europhysics Letters}, 88(4):40007, 2009.

\bibitem{StrSig}
Adam Allan, Michael Barany, and Robert~S. Strichartz.
\newblock Spectral operators on the {S}ierpinski gasket. {I}.
\newblock {\em Complex Var. Elliptic Equ.}, 54(6):521--543, 2009.

\bibitem{PARgaps}
Patricia Alonso~Ruiz.
\newblock Minimal gap in the spectrum of the {S}ierpi\'{n}ski gasket.
\newblock {\em Int. Math. Res. Not. IMRN}, 23:18874--18894, 2022.

\bibitem{alonso2023oscillations}
Patricia Alonso-Ruiz and Fabrice Baudoin.
\newblock Oscillations of bv measures on unbounded nested fractals.
\newblock {\em Journal of Fractal Geometry}, 9(3):373--396, 2023.

\bibitem{alonso2021besov}
Patricia Alonso-Ruiz, Fabrice Baudoin, Li~Chen, Luke Rogers, Nageswari
  Shanmugalingam, and Alexander Teplyaev.
\newblock Besov class via heat semigroup on dirichlet spaces iii: Bv functions
  and sub-gaussian heat kernel estimates.
\newblock {\em Calculus of Variations and Partial Differential Equations},
  60(5):170, 2021.

\bibitem{AMBROSIO-2000}
{L}. {A}mbrosio, {N}. {F}usco, and {D}. {P}allara.
\newblock {\em {F}unctions of bounded variation and free discontinuity
  problems}.
\newblock {O}xford {M}athematical {M}onographs, 2000.

\bibitem{As87}
S{o}ren Asmussen.
\newblock {\em Applied probability and queues}, volume~51 of {\em Applications
  of Mathematics (New York)}.
\newblock Springer-Verlag, New York, second edition, 2003.
\newblock Stochastic Modelling and Applied Probability.

\bibitem{BHR05}
Christoph Bandt, Nguyen~Viet Hung, and Hui Rao.
\newblock On the open set condition for self-similar fractals.
\newblock {\em Proc. Amer. Math. Soc.}, 134(5):1369--1374, 2006.

\bibitem{BARDOS-2016}
{C}. {B}ardos, {D}. {G}rebenkov, and {A}. {R}ozanova{-}{P}ierrat.
\newblock {S}hort-time heat diffusion in compact domains with discontinuous
  transmission boundary conditions.
\newblock {\em {M}ath. {M}odels {M}ethods {A}ppl. {S}ci.}, 26(01):59--110,
  {J}an 2016.

\bibitem{Bo12}
Tilman~Johannes Bohl.
\newblock Fractal curvatures and minkowski content of self-conformal sets,
  2012.

\bibitem{CAPITANELLI-2011}
{R}. {C}apitanelli and {M}.~{A}. {V}ivaldi.
\newblock {I}nsulating layers and {R}obin {P}roblems on {K}och mixtures.
\newblock {\em {J}ournal of {D}ifferential {E}quations}, 251(4-5):1332--1353,
  {A}ug 2011.

\bibitem{CLARET-2023}
{G}. {C}laret, {M}. {H}inz, {A}. {R}ozanova{-}{P}ierrat, and {A}. {T}eplyaev.
\newblock {L}ayer potentials in the imaging for transmission problems on
  domains with non-{L}ipschitz boundaries.
\newblock {\em {P}reprint}, 2023.

\bibitem{claret:hal-04316274}
Gabriel Claret and Anna Rozanova-Pierrat.
\newblock {Existence of optimal shapes for heat diffusions across irregular
  interfaces}.
\newblock {\em {P}reprint (to appear)}, November 2023.

\bibitem{CLARET-2024-1}
Gabriel Claret, Anna Rozanova-Pierrat, and Alexander Teplyaev.
\newblock {Convergence of layer potentials and Riemann-Hilbert problem on
  extension domains}.
\newblock {\em {P}reprint}, October 2024.

\bibitem{StrVicsek}
Sarah Constantin, Robert~S. Strichartz, and Miles Wheeler.
\newblock Analysis of the {L}aplacian and spectral operators on the {V}icsek
  set.
\newblock {\em Commun. Pure Appl. Anal.}, 10(1):1--44, 2011.

\bibitem{StrDiffEq}
Kyallee Dalrymple, Robert~S. Strichartz, and Jade~P. Vinson.
\newblock Fractal differential equations on the {S}ierpinski gasket.
\newblock {\em J. Fourier Anal. Appl.}, 5(2-3):203--284, 1999.

\bibitem{DE_GENNES-1982}
{P}.-{G}. de~{G}ennes.
\newblock {P}hysique des surfaces et des interfaces.
\newblock {\em {C}. {R}. {A}cad. {S}c. s{\'e}rie {I}{I}}, 295:1061--1064, 1982.

\bibitem{DEKKERS-2022}
{A}. {D}ekkers, {A}. {R}ozanova{-}{P}ierrat, and {A}. {T}eplyaev.
\newblock {M}ixed boundary valued problems for linear and nonlinear wave
  equations in domains with fractal boundaries.
\newblock {\em {C}alculus of {V}ariations and {P}artial {D}ifferential
  {E}quations}, 61(2), {M}ar 2022.

\bibitem{Grabner2008}
Gregory Derfel, Peter~J. Grabner, and Fritz Vogl.
\newblock The zeta function of the {L}aplacian on certain fractals.
\newblock {\em Trans. Amer. Math. Soc.}, 360(2):881--897, 2008.

\bibitem{Grabner2012}
Gregory Derfel, Peter~J. Grabner, and Fritz Vogl.
\newblock Laplace operators on fractals and related functional equations.
\newblock {\em J. Phys. A}, 45(46):463001, 34, 2012.

\bibitem{FALCONER-1997}
{K}.~{J}. {F}alconer.
\newblock {\em {T}echniques of {F}ractal {G}eometry}.
\newblock {W}iley, 1997.

\bibitem{Gittins23}
Sam Farrington and Katie Gittins.
\newblock Heat flow in polygons with reflecting edges.
\newblock {\em Integral Equations Operator Theory}, 95(4):Paper No. 27, 37,
  2023.

\bibitem{Fe59}
Herbert Federer.
\newblock Curvature measures.
\newblock {\em Trans. Amer. Math. Soc.}, 93:418--491, 1959.

\bibitem{Fel71}
William Feller.
\newblock {\em An introduction to probability theory and its applications.
  {Vol}. {II}. 2nd ed.}
\newblock Wiley Ser. Probab. Math. Stat. John Wiley \& Sons, Hoboken, NJ, 1971.

\bibitem{FLECKINGER-1995}
{J}. {F}leckinger, {M}. {L}evitin, and {D}. {V}assiliev.
\newblock {\NoAutoSpaceBeforeFDP} {{}H}eat {E}quation on the {T}riadic {V}on
  {K}och {S}nowflake: {A}symptotic and {N}umerical
  {A}nalysis{\AutoSpaceBeforeFDP}.
\newblock {\em {P}roceedings of the {L}ondon {M}athematical {S}ociety},
  s3-71(2):372--396, {S}ep 1995.

\bibitem{FK12}
Uta Freiberg and Sabrina Kombrink.
\newblock Minkowski content and local {M}inkowski content for a class of
  self-conformal sets.
\newblock {\em Geom. Dedicata}, 159:307--325, 2012.

\bibitem{Fu85}
Joseph H.~G. Fu.
\newblock Tubular neighborhoods in {E}uclidean spaces.
\newblock {\em Duke Math. J.}, 52(4):1025--1046, 1985.

\bibitem{Ga00}
Dimitris Gatzouras.
\newblock Lacunarity of self-similar and stochastically self-similar sets.
\newblock {\em Trans. Amer. Math. Soc.}, 352(5):1953--1983, 2000.

\bibitem{GILKEY-2003}
{P}.~{B}. {G}ilkey and {K}. {K}irsten.
\newblock {H}eat {C}ontent asymptotics with transmittal and transmission
  boundary conditions.
\newblock {\em {J}ournal of the {L}ondon {M}athematical {S}ociety},
  68(02):431--443, {S}ep 2003.

\bibitem{Grabner1997}
Peter~J. Grabner.
\newblock Functional iterations and stopping times for {B}rownian motion on the
  {S}ierpi\'{n}ski gasket.
\newblock {\em Mathematika}, 44(2):374--400, 1997.

\bibitem{GrabnerWoess1997}
Peter~J. Grabner and Wolfgang Woess.
\newblock Functional iterations and periodic oscillations for simple random
  walk on the {S}ierpi\'{n}ski graph.
\newblock {\em Stochastic Process. Appl.}, 69(1):127--138, 1997.

\bibitem{Gr95}
Siegfried Graf.
\newblock On {Bandt}'s tangential distribution for self-similar measures.
\newblock {\em Monatsh. Math.}, 120(3-4):223--246, 1995.

\bibitem{HAJLASZ-2008}
{P}. Hajlasz, {P}. {K}oskela, and {H}. {T}uominen.
\newblock {S}obolev embeddings, extensions and measure density condition.
\newblock {\em {J}ournal of {F}unctional {A}nalysis}, 254(5):1217--1234, {M}ar
  2008.

\bibitem{hambly2011asymptotics}
Ben~M Hambly.
\newblock Asymptotics for functions associated with heat flow on the sierpinski
  carpet.
\newblock {\em Canadian Journal of Mathematics}, 63(1):153--180, 2011.

\bibitem{HSTZ}
Kathryn~E. Hare, Benjamin~A. Steinhurst, Alexander Teplyaev, and Denglin Zhou.
\newblock Disconnected {J}ulia sets and gaps in the spectrum of {L}aplacians on
  symmetric finitely ramified fractals.
\newblock {\em Math. Res. Lett.}, 19(3):537--553, 2012.

\bibitem{HENROT-e}
Antoine Henrot and Michel Pierre.
\newblock {\em Shape variation and optimization}, volume~28 of {\em EMS Tracts
  in Mathematics}.
\newblock European Mathematical Society (EMS), Z\"{u}rich, 2018.
\newblock English version of the French publication with additions and updates.

\bibitem{HINZ-2021-1}
{M}. {H}inz, {A}. {R}ozanova{-}{P}ierrat, and {A}. {T}eplyaev.
\newblock {N}on-{L}ipschitz {U}niform {D}omain {S}hape {O}ptimization in
  {L}inear {A}coustics.
\newblock {\em {S}{I}{A}{M} {J}ournal on {C}ontrol and {O}ptimization},
  59(2):1007--1032, {J}an 2021.

\bibitem{hinz2018fractal}
Michael Hinz, Maria~Rosaria Lancia, Alexander Teplyaev, and Paola Vernole.
\newblock Fractal snowflake domain diffusion with boundary and interior drifts.
\newblock {\em Journal of Mathematical Analysis and Applications},
  457(1):672--693, 2018.

\bibitem{hinz2023boundary}
Michael Hinz, Anna Rozanova-Pierrat, and Alexander Teplyaev.
\newblock Boundary value problems on non-lipschitz uniform domains: stability,
  compactness and the existence of optimal shapes.
\newblock {\em Asymptotic Analysis}, 134(1-2):25--61, 2023.

\bibitem{Hu81}
John~E. Hutchinson.
\newblock Fractals and self-similarity.
\newblock {\em Indiana Univ. Math. J.}, 30(5):713--747, 1981.

\bibitem{kajino2014log}
Naotaka Kajino.
\newblock Log-periodic asymptotic expansion of the spectral partition function
  for self-similar sets.
\newblock {\em Communications in Mathematical Physics}, 328:1341--1370, 2014.

\bibitem{KK12}
Marc Kesseb{\"o}hmer and Sabrina Kombrink.
\newblock Fractal curvature measures and {M}inkowski content for self-conformal
  subsets of the real line.
\newblock {\em Adv. Math.}, 230(4-6):2474--2512, 2012.

\bibitem{ARXIV-KOMBRINK-2023}
{S}. {K}ombrink and {L}. {S}chmidt.
\newblock {O}n bounds for the remainder term of counting functions of the
  {N}eumann {L}aplacian on domains with fractal boundary.
\newblock 2023.

\bibitem{Ko11}
Sabrina Kombrink.
\newblock Fractal curvature measures and minkowski content for limit sets of
  conformal function systems.
\newblock {\em PhD thesis, Universit{\"a}t Bremen}, 2011.
\newblock {http://elib.suub.uni-bremen.de/edocs/00102477-1.pdf}.

\bibitem{Kombrink25}
Sabrina Kombrink and Lucas Schmidt.
\newblock Eigenvalue counting functions and parallel volumes for examples
  of fractal sprays generated by the koch snowflake.
\newblock In Julien Barral, Athanasios Batakis, and St{\'e}phane Seuret,
  editors, {\em Recent Developments in Fractals and Related Fields}, pages
  237--256, Cham, 2025. Springer Nature Switzerland.

\bibitem{KT2004}
Bernhard Kr\"{o}n and Elmar Teufl.
\newblock Asymptotics of the transition probabilities of the simple random walk
  on self-similar graphs.
\newblock {\em Trans. Amer. Math. Soc.}, 356(1):393--414, 2004.

\bibitem{LapidusPearse2006}
Michel~L. Lapidus and Erin P.~J. Pearse.
\newblock A tube formula for the {K}och snowflake curve, with applications to
  complex dimensions.
\newblock {\em J. London Math. Soc. (2)}, 74(2):397--414, 2006.

\bibitem{Lapidus-P-2008}
Michel~L. Lapidus and Erin P.~J. Pearse.
\newblock Tube formulas for self-similar fractals.
\newblock In {\em Analysis on graphs and its applications}, volume~77 of {\em
  Proc. Sympos. Pure Math.}, pages 211--230. Amer. Math. Soc., Providence, RI,
  2008.

\bibitem{Lapidus-P-2010}
Michel~L. Lapidus and Erin P.~J. Pearse.
\newblock Tube formulas and complex dimensions of self-similar tilings.
\newblock {\em Acta Appl. Math.}, 112(1):91--136, 2010.

\bibitem{Lapidus-P-W-2011}
Michel~L. Lapidus, Erin P.~J. Pearse, and Steffen Winter.
\newblock Pointwise tube formulas for fractal sprays and self-similar tilings
  with arbitrary generators.
\newblock {\em Adv. Math.}, 227(4):1349--1398, 2011.

\bibitem{LapidusPearse}
Michel~L. Lapidus, Erin P.~J. Pearse, and Steffen Winter.
\newblock Minkowski measurability results for self-similar tilings and fractals
  with monophase generators.
\newblock In {\em Fractal geometry and dynamical systems in pure and applied
  mathematics. {I}. {F}ractals in pure mathematics}, volume 600 of {\em
  Contemp. Math.}, pages 185--203. Amer. Math. Soc., Providence, RI, 2013.

\bibitem{LEVITIN-1996}
{M}. {L}evitin and {D}. {V}assiliev.
\newblock {S}pectral {A}symptotics, {R}enewal {T}heorem, and the {B}erry
  {C}onjecture for a {C}lass of {F}ractals.
\newblock {\em {P}roceedings of the {L}ondon {M}athematical {S}ociety},
  s3-72(1):188--214, {J}an 1996.

\bibitem{LW06}
Marta Llorente and Steffen Winter.
\newblock A notion of {E}uler characteristic for fractals.
\newblock {\em Math. Nachr.}, 280:152--170, 2007.

\bibitem{MCKEAN-1967}
{H}.~{P}. {M}c{K}ean, {J}r. and {I}.~{M}. {S}inger.
\newblock {C}urvature and the eigenvalues of the {L}aplacian.
\newblock {\em {J}ournal of {D}ifferential {G}eometry}, 1(1-2):43--69, 1967.

\bibitem{PeW12}
Erin P.~J. Pearse and Steffen Winter.
\newblock Geometry of canonical self-similar tilings.
\newblock {\em Rocky Mountain J. Math.}, 42(4):1327--1357, 2012.
\newblock \arxiv{0811.2187}.

\bibitem{PIROZHENKO-2005}
{I}.~{G}. {P}irozhenko, {V}.~{V}. {N}esterenko, and {M}. {B}ordag.
\newblock {I}ntegral equations for heat kernel in compound media.
\newblock {\em {J}ournal of {M}athematical {P}hysics}, 46(4):042305, {A}pr
  2005.

\bibitem{Po2023}
Du{\v{s}}an Pokorn{\'y}.
\newblock Curvatures for unions of {WDC} sets.
\newblock {\em Mathematika}, 69(3):665--691, 2023.

\bibitem{PR2013}
Du{\v{s}}an Pokorn{\'y} and Jan Rataj.
\newblock Normal cycles and curvature measures of sets with d.c. boundary.
\newblock {\em Adv. Math.}, 248:963--985, 2013.

\bibitem{PoW14}
Du{\v{s}}an Pokorn{\'y} and Steffen Winter.
\newblock Scaling exponents of curvature measures.
\newblock {\em J. Fractal Geom.}, 1(2):177--219, 2014.
\newblock \arxiv{1307.5053}.

\bibitem{RSS}
Jan Rataj, Volker Schmidt, and Evgeny Spodarev.
\newblock On the expected surface area of the {W}iener sausage.
\newblock {\em Math. Nachr.}, 282(4):591--603, 2009.

\bibitem{RW10}
Jan Rataj and Steffen Winter.
\newblock On volume and surface area of parallel sets.
\newblock {\em Indiana Univ. Math. J.}, 59(5):1661--1685, 2010.

\bibitem{RW13}
Jan Rataj and Steffen Winter.
\newblock Characterization of {M}inkowski measurability in terms of surface
  area.
\newblock {\em J. Math. Anal. Appl.}, 400:120--132, 2013.
\newblock \arxiv{1111.1825}.

\bibitem{RW23}
Jan Rataj and Steffen Winter.
\newblock On volume and surface area of parallel sets. {II.} {S}urface measures
  and (non-)\-differentiability of the volume.
\newblock {\em Preprint}, 2023.
\newblock \arxiv{2301.07429}.

\bibitem{RWZ23}
Jan Rataj, Steffen Winter, and Martina Z{\"a}hle.
\newblock Mean {Lipschitz}-{Killing} curvatures for homogeneous random
  fractals.
\newblock {\em J. Fractal Geom.}, 10(1-2):1--42, 2023.

\bibitem{RZ03}
Jan Rataj and Martina Z{\"a}hle.
\newblock Normal cycles of {L}ipschitz manifolds by approximation with parallel
  sets.
\newblock {\em Differential Geom. Appl.}, 19(1):113--126, 2003.

\bibitem{RZ05}
Jan Rataj and Martina Z{\"a}hle.
\newblock General normal cycles and {L}ipschitz manifolds of bounded curvature.
\newblock {\em Ann. Global Anal. Geom.}, 27(2):135--156, 2005.

\bibitem{RZ12}
Jan Rataj and Martina Z{\"a}hle.
\newblock Curvature densities of self-similar sets.
\newblock {\em Indiana Univ. Math. J.}, 61(4):1425--1449, 2012.

\bibitem{RZ19}
Jan Rataj and Martina Z\"{a}hle.
\newblock {\em Curvature measures of singular sets}.
\newblock Springer Monographs in Mathematics. Springer, Cham, 2019.

\bibitem{ROZANOVA-PIERRAT-2012}
{A}. {R}ozanova{-}{P}ierrat, {D}.~{S}. {G}rebenkov, and {B}. {S}apoval.
\newblock {F}aster {D}iffusion across an {I}rregular {B}oundary.
\newblock {\em {P}hys. {R}ev. {L}ett.}, 108(24):240602, {J}un 2012.

\bibitem{Sch94}
Andreas Schief.
\newblock Separation properties for self-similar sets.
\newblock {\em Proc. Amer. Math. Soc.}, 122(1):111--115, 1994.

\bibitem{Schn78}
Rolf Schneider.
\newblock Curvature measures of convex bodies.
\newblock {\em Ann. Mat. Pura Appl. (4)}, 116:101--134, 1978.

\bibitem{ST}
Benjamin~A. Steinhurst and Alexander Teplyaev.
\newblock Existence of a meromorphic extension of spectral zeta functions on
  fractals.
\newblock {\em Lett. Math. Phys.}, 103(12):1377--1388, 2013.

\bibitem{StrHarmonic}
Robert~S. Strichartz.
\newblock Harmonic analysis as spectral theory of {L}aplacians.
\newblock {\em J. Funct. Anal.}, 87(1):51--148, 1989.

\bibitem{StrBook}
Robert~S. Strichartz.
\newblock {\em Differential equations on fractals}.
\newblock Princeton University Press, Princeton, NJ, 2006.
\newblock A tutorial.

\bibitem{StrExact}
Robert~S. Strichartz.
\newblock Exact spectral asymptotics on the {S}ierpinski gasket.
\newblock {\em Proc. Amer. Math. Soc.}, 140(5):1749--1755, 2012.

\bibitem{StrRevisited}
Robert~S. Strichartz.
\newblock Spectral asymptotics revisited.
\newblock {\em J. Fourier Anal. Appl.}, 18(3):626--659, 2012.

\bibitem{strichartz2012spectral}
Robert~S Strichartz and Alexander Teplyaev.
\newblock Spectral analysis on infinite sierpi{\'n}ski fractafolds.
\newblock {\em Journal d'Analyse Math{\'e}matique}, 116(1):255--297, 2012.

\bibitem{MR1305782}
M.~van~den Berg.
\newblock Heat content and {B}rownian motion for some regions with a fractal
  boundary.
\newblock {\em Probab. Theory Related Fields}, 100(4):439--456, 1994.

\bibitem{MR1335459}
M.~van~den Berg.
\newblock Heat content asymptotics for some open sets with a fractal boundary.
\newblock In {\em Stochastic analysis ({I}thaca, {NY}, 1993)}, volume~57 of
  {\em Proc. Sympos. Pure Math.}, pages 11--22. Amer. Math. Soc., Providence,
  RI, 1995.

\bibitem{MR1785451}
M.~van~den Berg.
\newblock Heat equation on the arithmetic von {K}och snowflake.
\newblock {\em Probab. Theory Related Fields}, 118(1):17--36, 2000.

\bibitem{MR1831409}
M.~van~den Berg.
\newblock Renewal equation for the heat equation of an arithmetic von {K}och
  snowflake.
\newblock In {\em Infinite dimensional stochastic analysis ({A}msterdam,
  1999)}, volume~52 of {\em Verh. Afd. Natuurkd. 1. Reeks. K. Ned. Akad. Wet.},
  pages 25--37. R. Neth. Acad. Arts Sci., Amsterdam, 2000.

\bibitem{VANDENBERG-1994}
M.~van~den Berg and P.~Gilkey.
\newblock Heat content asymptotics of a {R}iemannian manifold with boundary.
\newblock {\em J. Funct. Anal.}, 120(1):48--71, 1994.

\bibitem{vdBG15}
M.~van~den Berg and P.~Gilkey.
\newblock Heat flow out of a compact manifold.
\newblock {\em J. Geom. Anal.}, 25(3):1576--1601, 2015.

\bibitem{MR1620833}
M.~van~den Berg and P.~B. Gilkey.
\newblock A comparison estimate for the heat equation with an application to
  the heat content of the {$S$}-adic von {K}och snowflake.
\newblock {\em Bull. London Math. Soc.}, 30(4):404--412, 1998.

\bibitem{vdBG16}
M.~van~den Berg and K.~Gittins.
\newblock On the heat content of a polygon.
\newblock {\em J. Geom. Anal.}, 26(3):2231--2264, 2016.

\bibitem{VAN_DEN_BERG-1994}
M.~van~den Berg and {J}.~{F}. {L}e {G}all.
\newblock {M}ean curvature and the heat equation.
\newblock {\em {M}athematische {Z}eitschrift}, 215(1):437--464, {J}an 1994.

\bibitem{VAN_DEN_BERG-1990}
{M}. van~den Berg and {S}. {S}risatkunarajah.
\newblock {H}eat flow and {B}rownian motion for a region in $\mathbb{{R}}^2$
  with a polygonal boundary.
\newblock {\em {P}robability {T}heory and {R}elated {F}ields}, 86(1):41--52,
  {M}ar 1990.

\bibitem{VASSILEVICH-2003}
{D}. {V}assilevich.
\newblock {\NoAutoSpaceBeforeFDP} {{}H}eat kernel expansion: user's
  manual{\AutoSpaceBeforeFDP}.
\newblock {\em {P}hysics {R}eports}, 388(5-6):279--360, {D}ec 2003.

\bibitem{Wi08}
Steffen Winter.
\newblock Curvature measures and fractals.
\newblock {\em Dissertationes Math. (Rozprawy Mat.)}, 453:1--66, 2008.

\bibitem{Wi11}
Steffen Winter.
\newblock Curvature bounds for neighborhoods of self-similar sets.
\newblock {\em Comment. Math. Univ. Carolin.}, 52(2):205--226, 2011.
\newblock \arxiv{1010.2032}.

\bibitem{Wi15}
Steffen Winter.
\newblock Minkowski content and fractal curvatures of self-similar tilings and
  generator formulas for self-similar sets.
\newblock {\em Adv. Math.}, 274:285--322, 2015.
\newblock \arxiv{1403.5201}.

\bibitem{WZ13}
Steffen Winter and Martina Z{\"a}hle.
\newblock Fractal curvature measures of self-similar sets.
\newblock {\em Adv. Geom.}, 13(2):229--244, 2013.

\bibitem{Za11}
M.~Z{\"a}hle.
\newblock Lipschitz-{K}illing curvatures of self-similar random fractals.
\newblock {\em Trans. Amer. Math. Soc.}, 363(5):2663--2684, 2011.

\bibitem{Za86b}
Martina Z{\"a}hle.
\newblock Integral and current representation of {F}ederer's curvature
  measures.
\newblock {\em Arch. Math. (Basel)}, 46(6):557--567, 1986.

\bibitem{Za13}
Martina Z\"ahle.
\newblock Curvature measures of fractal sets.
\newblock In {\em Fractal geometry and dynamical systems in pure and applied
  mathematics. {I}. {F}ractals in pure mathematics}, volume 600 of {\em
  Contemp. Math.}, pages 381--399. Amer. Math. Soc., Providence, RI, 2013.

\bibitem{Za23}
Martina Z\"ahle.
\newblock Mean {M}inkowski and {$S$}-contents of {$V$}-variable random
  fractals.
\newblock {\em Asian J. Math.}, 27(6):955--970, 2023.

\end{thebibliography}

\end{document}